\documentclass[11pt, a4paper]{amsart}
\usepackage{amssymb,amsmath,amsthm,bbold}
\usepackage{geometry}
\usepackage{tikz}
\usetikzlibrary{cd, intersections, calc}
\usepackage{amscd}
\usepackage[all]{xy}
\definecolor{darkspringgreen}{rgb}{0.09, 0.45, 0.27}
\definecolor{dartmouthgreen}{rgb}{0.05, 0.5, 0.06}
\usepackage[colorlinks,final,hyperindex]{hyperref}
\hypersetup{citecolor=dartmouthgreen, urlcolor=blue, linkcolor=darkspringgreen}
\usepackage[font=small, labelfont=bf, margin=10pt]{caption}
\usepackage{csquotes}

\newtheorem*{conjecture*}{Conjecture}
\newtheorem{theorem}{Theorem}[subsection]
\newtheorem*{thm*}{Theorem}
\newtheorem{defi}{Definition}[subsection]
\newtheorem*{claim*}{Claim}
\newtheorem{lemma}{Lemma}[subsection]
\newtheorem{prop}{Proposition}[subsection]
\newtheorem{cor}{Corollary}[subsection]

\newcommand{\Gr}{\operatorname{\Gamma}}
\newcommand{\parti}{\operatorname{\Pi_{\mathsf{gr}}}}
\newcommand{\Path}{\operatorname{\mathsf{P}}}
\newcommand{\K}{\operatorname{\mathsf{K}}}
\newcommand{\Cyc}{\operatorname{\mathsf{C}}}
\newcommand{\St}{\operatorname{\mathsf{St}}}
\DeclareMathOperator{\Ver}{Vert}
\DeclareMathOperator{\Edge}{Edge}
\DeclareMathOperator{\Leav}{Leav}
\DeclareMathOperator{\In}{In}
\newcommand{\Aut}{\operatorname{\mathsf{Aut}}}
\DeclareMathOperator{\Id}{Id}
\DeclareMathOperator{\Hom}{Hom}
\newcommand{\GrCol}{\operatorname{\mathsf{GrCol}}}
\newcommand{\Y}{\operatorname{\overline{\mathsf{Y}}}}
\newcommand{\M}{\operatorname{\mathcal{M}}}
\newcommand{\bM}{\operatorname{\overline{\mathcal{M}}}}
\newcommand{\Pro}{\operatorname{\mathbb{P}}}
\DeclareMathOperator{\Conf}{Conf}

\newcommand{\D}{\operatorname{\mathcal{D}}}
\newcommand{\Disc}{\operatorname{\mathbb{D}}}
\newcommand{\La}{\operatorname{\mathcal{L}}}
\DeclareMathOperator{\rk}{rk}
\newcommand{\G}{\operatorname{\mathcal{G}}}
\newcommand{\N}{\operatorname{\mathcal{N}}}
\newcommand{\B}{\operatorname{\mathcal{B}}}
\newcommand{\A}{\operatorname{\mathcal{A}}}
\newcommand{\Orb}{\operatorname{\mathcal{O}}}
\newcommand{\F}{\operatorname{\mathcal{F}}}
\newcommand{\Q}{\operatorname{\mathcal{Q}}}
\newcommand{\I}{\operatorname{\mathcal{I}}}
\newcommand{\E}{\operatorname{\mathcal{E}}}
\newcommand{\R}{\operatorname{\mathcal{R}}}
\newcommand{\Ho}{\operatorname{\mathcal{H}}}
\newcommand{\T}{\operatorname{\mathbb{T}}}
\newcommand{\Pop}{\operatorname{\mathcal{P}}}
\newcommand{\C}{\operatorname{\mathcal{C}}}
\DeclareMathOperator{\Tw}{Tw}
\DeclareMathOperator{\Der}{Der}
\DeclareMathOperator{\Coder}{Coder}
\usepackage[OT2,OT1]{fontenc}
\newcommand\cyrillic[1]{{\fontencoding{OT2}\fontfamily{wncyr}\selectfont #1}}
\newcommand\mathcyr[1]{\text{\cyrillic{#1}}}
\newcommand\Sha{\textnormal{\mathcyr{Sh}}}
\newcommand{\forget}{\operatorname{\mathrm{f}}}
\newcommand{\LT}{\operatorname{\mathsf{LT}}}
\newcommand{\grpermlex}{\operatorname{\mathsf{graphpermlex}}}
\DeclareMathOperator{\ind}{ind}
\DeclareMathOperator{\res}{res}
\DeclareMathOperator{\Pa}{Path}
\newcommand{\Com}{\operatorname{\mathsf{gcCom}}}
\DeclareMathOperator{\End}{End}
\newcommand{\Susp}{\operatorname{\mathcal{S}}}
\newcommand{\Lie}{\operatorname{\mathsf{gcLie}}}
\newcommand{\Ass}{\operatorname{\mathsf{gcAss}}}
\newcommand{\OS}{\operatorname{\mathsf{OS}}}
\newcommand{\Gerst}{\operatorname{\mathsf{gcGerst}}}
\newcommand{\Tree}{\operatorname{\mathsf{Tree}}}
\newcommand{\ST}{\operatorname{\mathsf{ST}}}
\newcommand{\RST}{\operatorname{\mathsf{RST}}}

\title{
A generalization of operads based on subgraph contractions
}
\author{Denis Lyskov}

\address{National Research University Higher School of Economics, 20 Myasnitskaya street, Moscow 101000, Russia}

\email{ddl2001@yandex.ru}
\date{}
\begin{document}
\begin{abstract}
    We introduce a 
    generalization of the notion of operad that we call a \textit{contractad}, whose set of operations is indexed by connected graphs and whose composition rules are numbered by contractions of connected subgraphs.
    
    We show that many classical operads, such as the operad of commutative algebras, Lie algebras, associative algebras, pre-Lie algebras, the little disks operad, and the operad of moduli spaces of stable curves $\bM_{0,n+1}$ admit generalizations to contractads. 
    We explain that standard tools like Koszul duality and the machinery of Gr\"obner bases can be easily generalized to contractads. 
    We verify the Koszul property of the commutative, Lie, associative, and Gerstenhaber contractads.
\end{abstract}
\maketitle
\setcounter{tocdepth}{1}
\tableofcontents

\section{Introduction}

In a recent paper~\cite{dotsenko2024reconnectads}, the author, alongside Dotsenko and Keilthy, introduced a natural extension of the notion of operad to a construction defined on graphs, with composition defined in terms of the reconnected complement of a graph with respect to a connected subgraph. Another essential combinatorial construction in the study of finite simple graphs is that of contraction: collapsing a connected subgraph to a single vertex. Many graph invariants make use of contraction in their construction, and it is a powerful tool for proving results recursively. As such, it seemed natural to consider a graphical operad structure with components indexed by connected finite simple graphs, inputs labeled by vertices, and composition defined in terms of contraction. We call this structure a \textit{contractad} and find many natural examples of contractads in areas across mathematics.

Let us briefly define contractads, for details see Section~\ref{subsec:contractads}. A contractad with values in a symmetric monoidal category $\C$ is a contravariant functor 
\[\Pop\colon \mathsf{CGr}^{\mathrm{op}}\rightarrow \C\]
from the groupoid of connected finite simple graphs equipped with a collection of infinitesimal compositions defined as follows. For each connected graph $\Gr$ and collection of vertices $G$ inducing a connected subgraph $\Gr|_G$, we have a map
\[
\circ^{\Gr}_G\colon \Pop(\Gr/G)\otimes \Pop(\Gr|_G) \rightarrow \Pop(\Gr),
\] 
where $\Gr/G$ is the graph obtained from $\Gr$ by contracting $G$ to a single vertex. The collection of such maps satisfies natural associativity and equivariance conditions, for details see Definition~\ref{partdefi}.

We show that a contractad naturally generalizes the existing operad-like structures (Section~\ref{subsec:operads}). For example, the restriction of a contractad to the family of complete graphs recovers the structure of a symmetric operad. Similarly, the restriction to the family of path graphs recovers the structure of a non-symmetric operad (with an additional $\mathbb{Z}_2$-action on each component).

In this article, we discuss several equivalent definitions of contractad (Section~\ref{subsec:contractads}, \ref{subsec:admissibletrees}), and establish the necessary algebraic theory in order to introduce a bar-cobar construction (Section~\ref{subsec:bar}), Koszul duality (Section~\ref{subsec:koszul}), and Gr{\"o}bner bases (Section~\ref{sec:grobner}). We also introduce many examples arising from algebraic, geometric, topological, and combinatorial structures, which are used throughout to illustrate the theory we develop. These examples naturally generalize a number of familiar operads. We list some of them here.\\

The first example of a contractad comes naturally from algebraic topology. In Section~\ref{subsec:disks}, for $n\geq 1$, we introduce the \textit{contractad of little $n$-disks} $\D_n$. Each component of this contractad $\D_n(\Gr)$ consists of configurations of $n$-dimensional disks in the unit disk labeled by the vertex set of a graph, such that interiors of disks corresponding to adjacent vertices do not intersect. We obtain the structure of a contractad on $\D_n$ by substitution of disk configurations. By restricting to complete graphs, we recover the usual symmetric little disks operad~\cite{cohen1976homology}.

By taking the homology of this contractad, we obtain the contractad in $\mathbb{Z}$-modules with many results about the corresponding operad holding in the contractad setting~\cite{getzler1994operads}. 
\begin{thm*}[Theorem~\ref{thm:ass}, Theorem~\ref{thm:en}]
For $n\geq1$, the homology contractad of the little $n$-disks contractad $H_{\bullet}(\D_n)$ is quadratic and Koszul.
\end{thm*}
\noindent We additionally determine an explicit presentation for this homology contractad, allowing us to view it as a graphical generalization of $e_n$-operad~\cite{getzler1994operads}.

Similar to the little disks contractad, there is a natural notion of \textit{graphical configuration space}. Such spaces have been well studied, for example~\cite{eastwood2007euler, baranovsky2012graph, wiltshire2018configuration}. The classical little disks operad provide efficient tools for studying rational homotopy types of the configuration spaces of manifolds~\cite{campos2016model}. We expect that the little disks contractad would play the same role for graphical configuration spaces, enriching and enhancing the study of such spaces.\\

The second example of a contractad comes from algebraic geometry.  In Section~\ref{subsec:wonderful}, we introduce the \textit{Wonderful contractad} $\bM$ that generalizes the Deligne-Mumford operad of stable pointed curves of genus zero~\cite{getzler1995operads}. Specifically, we extend the construction of $\bM_{0,n}$ via wonderful compactifications, introduced by De Concini and Processi~\cite{de1995wonderful}, by considering the wonderful compactifications associated with the so-called \textit{graphical building sets}(page~\pageref{grapharrang}), with a contractad structure arising naturally from the combinatorics of building sets. We obtain a contractad whose restriction to complete graphs recovers the symmetric operad structure on the collection $\{\bM_{0,n}\}$, whose restriction to path graphs recovers the non-symmetric brick operad~\cite{dotsenko2019toric}, and whose restriction to stellar graphs recovers the twisted associative algebra of Losev-Manin moduli spaces $\overline{L}_{n}$~\cite{losev2000new}. It is worth mentioning that this contractad may be viewed as a special case of a general operad structure on wonderful compactifications developed by Coron~\cite{coron2022matroids} for the full generality of geometric lattices and their building sets. We expect the homology of this contractad to define a graphical generalization of the hypercommutative operad~\cite{getzler1995operads}, and the non-commutative hypercommutative operad~ \cite{dotsenko2019toric}. However, we leave the full exploration of this structure for a future article.\\

The third example of a contractad comes naturally from the combinatorics of graphs and spanning trees.  In Section~\ref{subsec:spantrees}, we introduce the \textit{contractad of rooted spanning trees} $\RST^{\vee}$. Each component of this contractad $\RST^{\vee}(\Gr)$ is generated by rooted spanning trees of the underlying graph, and a contractad structure comes from certain gluings of rooted trees. This defines a graphical counterpart of the rooted tree operad, first described by Chapoton and Livernet in~\cite{chapoton2001pre}. In this article, the authors proved that this operad is isomorphic to the operad $\mathsf{preLie}$ of pre-Lie algebras. We formulate and prove a similar result for $\RST^{\vee}$.  Moreover, we show that the property of $\mathsf{preLie}$ to be Koszul does not remain true for its graphical analogue.

\begin{thm*}[Proposition~\ref{rst}, Proposition~\ref{prop:nonkoszul}]
The contractad of rooted spanning trees $\RST^{\vee}$ is quadratic but not Koszul.
\end{thm*}

Finally, we introduce graphical analogues of the operads encoding classical algebras, such as the operads $\mathsf{Com}$ of commutative algebras, $\mathsf{Lie}$ of Lie algebras, $\mathsf{Ass}$ of associative algebras, and $\mathsf{Gerst}$ of Gerstenhaber algebras. Moreover, we show that these graphical replacements exhibit most algebraic properties that their classical analogues possess.
\begin{thm*}
    The contractads $\Com,\Lie,\Ass,$ and $\Gerst$ are quadratic and Koszul.
\end{thm*}

\subsection*{Organisation}  In Section~\ref{sec:contractads}, we give several equivalent definitions of a contractad.  In Section~\ref{sec:examples}, we give the first examples of contractads arising in algebra, topology, geometry, and combinatorics. In Section~\ref{sec:koszul}, we develop a Koszul duality theory of contractads and give examples of (non)Koszul contractads. In Section~\ref{sec:grobner}, we develop a theory of Gr\"obner basis for contractads. In Section~\ref{sec:discs}, we study the homology contractads of the little $n$-disks contractads.

\section{Contractads}\label{sec:contractads}
In this section, we introduce the notion of contractad. First, we introduce the necessary definitions and constructions around graphs. Next, we provide several definitions of contractads. At the end, we explain the relations of contractads to classical algebraic structures.  We refer the reader to Section~\ref{sec:examples} for examples. 
\subsection{Graphs, partitions, and contractions}
In this paper, we define a \textit{graph} as a finite undirected graph $\Gr=(V_{\Gr},E_{\Gr})$ without loops and  multiple edges, where $V_{\Gr}$ is a set of vertices, and $E_{\Gr}$ is a set of edges. Two vertices $v,w \in V_{\Gr}$ are \textit{adjacent} if they form an edge in the graph. A graph is \textit{connected} if there is at least one path connecting any two vertices. Let us consider some particular examples of graphs:
\begin{itemize}
\label{typesofgraphs}
     \item the path graph $\Path_n$ on the vertex set $\{1,\cdots, n\}$ with edges $\{(i,i+1)| 1\leq i \leq n-1 \}$,    \item the complete graph $\K_n$ on the vertex set $\{1,\cdots, n\}$ and the edges $\{(i,j)|i\neq j\}$,
    \item the cycle graph $\Cyc_n$ on the vertex set $\{1,\cdots, n\}$ with edges $\{(i,i+1)| 1 \leq i \leq n-1\}\cup \{(n,1)\}$,
    \item the stellar graph $\St_n$ on the vertex set $\{0,1,\cdots, n\}$ with edges $\{(0,i)| 1\leq i \leq n \}$. The vertex ``$0$'' adjacent to all vertices is called the ``core''.
\end{itemize}
For a graph $\Gr$ and a subset of vertices $S$,  the \textit{induced subgraph} is the graph $\Gr|_S$ with vertex set $S$ and edges coming from the original graph.
\begin{defi}
\begin{enumerate}
    \item A tube of a graph $\Gr$ is a non-empty subset $G$ of vertices such that the induced subgraph $\Gr|_G$ is connected.  If the tube consists of one vertex, we call it trivial.
    \item A \textit{partition of a graph} $\Gr$ is a partition of the vertex set whose blocks are tubes. We denote by $\parti(\Gr)$ the set of partitions of the graph $\Gr$.
\end{enumerate}
\end{defi}
\noindent The partition set $\parti(\Gr)$ admits a partial order by \textit{refinement}. More explicitly, for a pair of partitions $I,J$, we have $I<J$ if each block from the left partition is contained in some other block from the right. This partially ordered set (poset) has a maximal partition $\hat{1}=\{V_1,V_2,\cdots,V_k\}$ that is made up of connected components and a minimal one $\hat{0}=\{\{v\}\}_{v \in V_{\Gr}}$ made up of one-vertex blocks.
\begin{figure}[ht]
\def\gra{
\begin{tikzpicture}
    \draw (0,0)--(0.87,0.5)--(0.87,-0.5)--cycle;
   \draw (-0.7,0)--(0,0);
   \fill (0,0) circle (2pt);
   \fill (-0.7,0) circle (2pt);
    \fill (0.87,0.5) circle (2pt);
    \fill (0.87,-0.5) circle (2pt);
\end{tikzpicture}}
\def\top{
\begin{tikzpicture}
    \draw (0,0)--(0.87,0.5)--(0.87,-0.5)--cycle;
    \draw (-0.7,0)--(0,0);
    \fill (0,0) circle (2pt);
    \fill (-0.7,0) circle (2pt);
    \fill (0.87,0.5) circle (2pt);
    \fill (0.87,-0.5) circle (2pt);
    \draw[dashed] (-1.2,0)[rounded corners=15pt]--(1.1,0.8)--(1.1,-0.8)--cycle;
\end{tikzpicture}}
\def\midl{
\begin{tikzpicture}
    \draw (0,0)--(0.87,0.5)--(0.87,-0.5)--cycle;
   \draw (-0.7,0)--(0,0);
   \fill (0,0) circle (2pt);
   \fill (-0.7,0) circle (2pt);
    \fill (0.87,0.5) circle (2pt);
    \fill (0.87,-0.5) circle (2pt);
    \draw[dashed] (0.87,-0.5) circle (5pt);
    \draw[dashed] (-0.9,-0.2)[rounded corners=5pt]--(0.1,-0.2)--(1.1,0.4)--(0.9,0.7)--(0,0.2)--(-0.9,0.2)--cycle;
\end{tikzpicture}}
\def\midc{
\begin{tikzpicture}
    \draw (0,0)--(0.87,0.5)--(0.87,-0.5)--cycle;
   \draw (-0.7,0)--(0,0);
   \fill (0,0) circle (2pt);
   \fill (-0.7,0) circle (2pt);
    \fill (0.87,0.5) circle (2pt);
    \fill (0.87,-0.5) circle (2pt);
    \draw[dashed] (0.87,0.5) circle (5pt);
    \draw[dashed] (-0.9,0.2)[rounded corners=5pt]--(0.1,0.2)--(1.1,-0.4)--(0.9,-0.7)--(0,-0.2)--(-0.9,-0.2)--cycle;
\end{tikzpicture}}
\def\midr{
\begin{tikzpicture}
    \draw (0,0)--(0.87,0.5)--(0.87,-0.5)--cycle;
   \draw (-0.7,0)--(0,0);
   \fill (0,0) circle (2pt);
   \fill (-0.7,0) circle (2pt);
    \fill (0.87,0.5) circle (2pt);
    \fill (0.87,-0.5) circle (2pt);
    \draw[dashed] (-0.7,0) circle (5pt);
    \draw[dashed] (-0.3,0)[rounded corners=10pt]--(1.05,0.8)--(1.05,-0.8)--cycle;
\end{tikzpicture}}
\def\lmidl{
\begin{tikzpicture}
    \draw (0,0)--(0.87,0.5)--(0.87,-0.5)--cycle;
   \draw (-0.7,0)--(0,0);
   \fill (0,0) circle (2pt);
   \fill (-0.7,0) circle (2pt);
    \fill (0.87,0.5) circle (2pt);
    \fill (0.87,-0.5) circle (2pt);
    \draw[dashed] (0.87,-0.5) circle (5pt);
    \draw[dashed] (0.87,0.5) circle (5pt);
    \draw[dashed, rounded corners=5pt] (-0.9,-0.2) rectangle ++(1.1,0.4);
\end{tikzpicture}}
\def\lmidr{
\begin{tikzpicture}
    \draw (0,0)--(0.87,0.5)--(0.87,-0.5)--cycle;
   \draw (-0.7,0)--(0,0);
   \fill (0,0) circle (2pt);
   \fill (-0.7,0) circle (2pt);
    \fill (0.87,0.5) circle (2pt);
    \fill (0.87,-0.5) circle (2pt);
    \draw[dashed] (-0.7,0) circle (5pt);
    \draw[dashed] (0.87,-0.5) circle (5pt);
    \draw[dashed, rounded corners=5pt,rotate=30] (-0.2,-0.2) rectangle ++(1.4,0.4);
\end{tikzpicture}}
\def\rmidl{
\begin{tikzpicture}
    \draw (0,0)--(0.87,0.5)--(0.87,-0.5)--cycle;
   \draw (-0.7,0)--(0,0);
   \fill (0,0) circle (2pt);
   \fill (-0.7,0) circle (2pt);
    \fill (0.87,0.5) circle (2pt);
    \fill (0.87,-0.5) circle (2pt);
    \draw[dashed] (-0.7,0) circle (5pt);
    \draw[dashed] (0.87,0.5) circle (5pt);
    \draw[dashed, rounded corners=5pt,rotate=-30] (-0.2,-0.2) rectangle ++(1.4,0.4);
\end{tikzpicture}}
\def\rmidr{
\begin{tikzpicture}
    \draw (0,0)--(0.87,0.5)--(0.87,-0.5)--cycle;
   \draw (-0.7,0)--(0,0);
   \fill (0,0) circle (2pt);
   \fill (-0.7,0) circle (2pt);
    \fill (0.87,0.5) circle (2pt);
    \fill (0.87,-0.5) circle (2pt);
    \draw[dashed] (-0.7,0) circle (5pt);
    \draw[dashed] (0,0) circle (5pt);
    \draw[dashed, rounded corners=5pt] (0.67,-0.7) rectangle ++(0.4,1.4);
\end{tikzpicture}}
\def\bot{
\begin{tikzpicture}
    \draw (0,0)--(0.87,0.5)--(0.87,-0.5)--cycle;
   \draw (-0.7,0)--(0,0);
   \fill (0,0) circle (2pt);
   \fill (-0.7,0) circle (2pt);
    \fill (0.87,0.5) circle (2pt);
    \fill (0.87,-0.5) circle (2pt);
    \draw[dashed] (-0.7,0) circle (5pt);
    \draw[dashed] (0,0) circle (5pt);
    \draw[dashed] (0.87,-0.5) circle (5pt);
    \draw[dashed] (0.87,0.5) circle (5pt);
\end{tikzpicture}}
    \centering
    \begin{tikzpicture}
        \node (l1) at (0,0) {\top};
        \node (l21) at (-3,-2.5) {\midl};
        \node (l22) at (0,-2.5) {\midc};
        \node (l23) at (3,-2.5) {\midr};
        \node (l31) at (-4.5,-5) {\lmidl};
        \node (l32) at (-1.5,-5) {\lmidr};
        \node (l33) at (1.5,-5) {\rmidl};
        \node (l34) at (4.5,-5) {\rmidr};
        \node (l4) at (0,-7.5) {\bot};
        \draw (l1)--(l21);
        \draw (l1)--(l22);
        \draw (l1)--(l23);
        \draw (l21)--(l31);
        \draw (l21)--(l32);
        \draw (l22)--(l31);
        \draw (l22)--(l33);
        \draw (l23)--(l32);
        \draw (l23)--(l33);
        \draw (l23)--(l34);
        \draw (l31)--(l4);
        \draw (l32)--(l4);
        \draw (l33)--(l4);
        \draw (l34)--(l4);
    \end{tikzpicture}
    \caption{Hasse diagram for partition poset of the graph. Partitions increase from bottom to top.}
\end{figure}
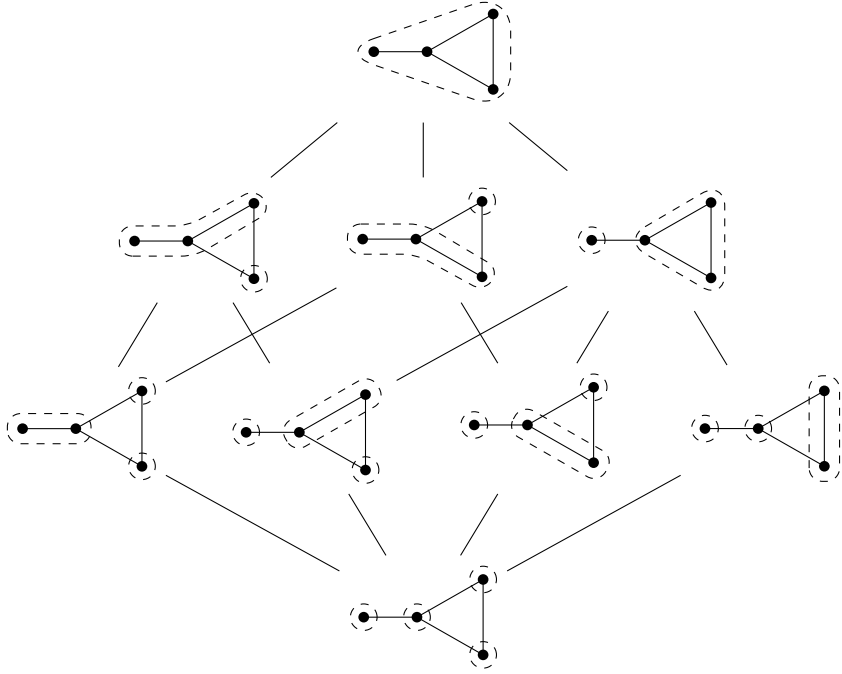

\begin{defi}[Graph contraction]
For a partition $I$ of a graph $\Gr$, the contracted graph, denoted $\Gr/I$, is the graph obtained from $\Gr$  by contracting each block of $I$ to a single vertex: specifically, vertices of $\Gr/I$ are partition blocks and edges are pairs $\{G\},\{H\}$ of blocks such that their union $G\cup H$ is a tube of $\Gr$.
\end{defi}

\begin{figure}[ht]
  \centering
  \begin{gather*}
  \vcenter{\hbox{\begin{tikzpicture}[scale=0.7]
    \fill (0,0) circle (2pt);
    \node at (0,0.6) {1};
    \fill (1,0) circle (2pt);
    \node at (1,0.6) {2};
    \fill (2,0) circle (2pt);
    \node at (2,0.6) {3};
    \fill (3,0) circle (2pt);
    \node at (3,0.6) {4};
    \fill (4,0) circle (2pt);
    \node at (4,0.6) {5};
    \draw (0,0)--(1,0)--(2,0)--(3,0)--(4,0);
    \draw[dashed, rounded corners=5pt] (-0.25,-0.25) rectangle ++(1.5,0.5);
    \draw[dashed, rounded corners=5pt] (2.75,-0.25) rectangle ++(1.5,0.5);
    \draw[dashed] (2,0) circle (7pt);
    \end{tikzpicture}}}
    \quad
    \longrightarrow
    \quad
  \vcenter{\hbox{\begin{tikzpicture}[scale=0.7]
    \fill (0,0) circle (2pt);
    \node at (0,0.5) {\{1,2\}};
    \fill (1.5,0) circle (2pt);
    \node at (1.5,0.5) {\{3\}};
    \fill (3,0) circle (2pt);
    \node at (3,0.5) {\{4,5\}};
    \draw (0,0)--(1.5,0)--(3,0);    
    \end{tikzpicture}}}
    \\
    \\
  \vcenter{\hbox{\begin{tikzpicture}[scale=0.7]
    \fill (0,0) circle (2pt);
    \node at (-0.4,-0.3) {1};
    \fill (2,0) circle (2pt);
    \node at (2.4,-0.3) {4};
    \fill (0,2) circle (2pt);
    \node at (-0.4,2.3) {2};
    \fill (1,1) circle (2pt);
    \node at (1,0.5) {5};
    \fill (2,2) circle (2pt);
    \node at (2.4,2.3) {3};
    \draw (0,0)--(2,0)--(2,2)--(0,2)--cycle;
    \draw (0,0)--(1,1)--(2,2);
    \draw (2,0)--(1,1)--(0,2);
    \draw[dashed] (0,0) circle (7pt);
    \draw[dashed] (1,1) circle (7pt);
    \draw[dashed] (2,0) circle (7pt);
    \draw[dashed, rounded corners=5pt] (-0.25,1.75) rectangle ++(2.5,0.5);
    \end{tikzpicture}}}
    \quad
    \longrightarrow
    \vcenter{\hbox{\begin{tikzpicture}[scale=0.75]
    \fill (0,0) circle (2pt);
    \node at (-0.4,-0.3) {\{1\}};
    \fill (1,1) circle (2pt);
    \node at (1,0.5) {\{5\}};
    \fill (1,2) circle (2pt);
    \node at (1,2.4) {\{2,3\}};
    \fill (2,0) circle (2pt);
    \node at (2.4,-0.3) {\{4\}};
    \draw (0,0)--(2,0)--(1,2)-- cycle;
    \draw (0,0)--(1,1)--(1,2);
    \draw (2,0)--(1,1);
    \end{tikzpicture}}}
    \quad
    \quad
    \quad
    \vcenter{\hbox{\begin{tikzpicture}[scale=0.7]
    \fill (-0.63,1.075)  circle (2pt);
    \node at (-0.9,1.4) {1};
    \fill (0.63,1.075)  circle (2pt);
    \draw[dashed] (0.63,1.075) circle  (7pt);
    \node at (0.9,1.4) {2};
    \fill (-1.22,0) circle (2pt);
    \node at (-1.6,0) {6};
    \fill (1.22,0) circle (2pt);
    \node at (1.6,0) {3};
    \draw[dashed] (1.22,0) circle  (7pt);
    \fill (-0.63,-1.075)  circle (2pt);
    \node at (-0.9,-1.4) {5};
    \fill (0.63,-1.075)  circle (2pt);
    \node at (0.9,-1.4) {4};
    \draw[dashed] (0.63,-1.075) circle  (7pt);
    \draw (-1.22,0)--(-0.63,1.075)--(0.63,1.075)--(1.22,0)--(0.63,-1.075)--(-0.63,-1.075)--cycle;
    \draw[dashed] (-1.6,0)[rounded corners=15pt]--(-0.4,1.7)[rounded corners=15pt]--(-0.4,-1.7)[rounded corners=12pt]--cycle;
    \end{tikzpicture}}}
    \quad\longrightarrow\quad
    \vcenter{\hbox{\begin{tikzpicture}[scale=0.7]
     \fill (0.63,1.075)  circle (2pt);
     \node at (0.9,1.5) {\{2\}};
    \fill (-1.22,0) circle (2pt);
    \node at (-2.2,0) {\{1,5,6\}};
    \fill (1.22,0) circle (2pt);
    \node at (1.75,0) {\{3\}};
    \fill (0.63,-1.075)  circle (2pt);
    \node at (0.9,-1.5) {\{4\}};
    \draw (-1.22,0)--(0.63,1.075)--(1.22,0)--(0.63,-1.075)--cycle;
    \end{tikzpicture}}}
  \end{gather*}
  \caption{Examples of contractions.}
  \label{contrpic}
\end{figure}
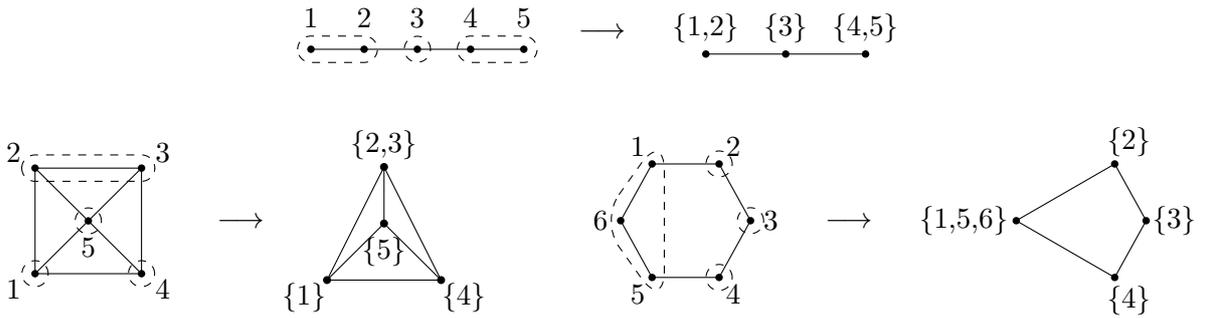

When a partition has the form $I=\{G\}\cup \{\{v\}|v \not\in G\}$ for some tube $G$, we shall denote the associated contracted graph by  $\Gr/G$. For a pair of comparable partitions $I_1\leq I_2$, let $[I_1,I_2]$ be the closed interval of intermediate partitions $I_1 \leq J \leq I_2$. The following lemma is proved by direct inspection.
\begin{lemma}
\label{intervals}
For a graph $\Gr$ with partition $I$, we have isomorphisms of posets
\begin{gather*}
    [\hat{0},I]\cong \prod_{G \in I} \parti(\Gr|_G),
    \\
    [I, \hat{1}] \cong \parti(\Gr/I).
\end{gather*}
\end{lemma}

\subsection{Contractads}\label{subsec:contractads}

Let $\C=(\C, \otimes, 1_{\C})$ be a symmetric monoidal category. In most cases, by $\C$ we mean the category of topological spaces $(\mathsf{Top},\times)$ or the category of differential graded vector spaces $(\mathsf{dgVect}, \otimes)$ (with the Koszul signs rule). Let us recall the definition of graphical collections from~ \cite{dotsenko2024reconnectads}. Consider the \textit{groupoid of connected graphs} $\mathsf{CGr}$ whose objects are non-empty connected simple graphs and whose morphisms are isomorphisms of graphs.
\begin{defi}
A graphical collection with values in $\C$ is a contravariant functor $\mathsf{CGr}^{\mathrm{op}}\rightarrow \C$. All graphical collections with values in $\C$ with natural transformations form a category $\GrCol_{\C}$.
\end{defi}

Let $\Orb$ be a graphical collection. For functorial reasons, for each component $\Orb(\Gr)$, there is a right action of the graph automorphisms group $\Aut(\Gr)$. This observation gives us a naive parallel with reduced $\Sigma$-modules (contravariant functors from the groupoid of non-empty finite sets). Recall that the category of $\Sigma$-modules has the "composition product", and monoids associated to this product are symmetric operads~\cite{loday2012algebraic}. A similar story holds for graphical collections.

\begin{defi}
The contraction product of two graphical collections $\Pop$ and $\Q$ is the graphical collection $\Pop\circ_{\mathsf{C}}\Q$ defined by the formula
\[
    \Pop \circ_{\mathsf{C}} \Q (\Gr) := \bigoplus_{I \in \parti(\Gr)} \Pop(\Gr/I) \otimes \bigotimes_{G \in I} \Q(\Gr|_G),
\]  where the sum ranges over all partitions of $\Gr$.
\end{defi}

\noindent Define the unit graphical collection $\mathbb{1}$ by putting
\[
\mathbb{1}(\Gr):= \begin{cases}
1_{\C}, \text{ for } \Gr \cong \Path_1,
\\
0, \text{ otherwise}.
\end{cases}
\]

\begin{prop}
The triple $\GrCol=(\GrCol_{\C}, \circ, \mathbb{1})$ forms a monoidal category.
\end{prop}
\begin{proof}
Let us extend each graphical collection to non-connected graphs by putting  $\F(\Gr\coprod \Gr')=\F(\Gr)\otimes \F(\Gr')$. Using this convention, the product of graphical collections is written in a more compact way
\[
    \Pop\circ\Q(\Gr)=\bigoplus_{I \in \parti(\Gr)} \Pop(\Gr/I) \otimes \Q(\Gr|_I),
\] where $\Gr|_I=\coprod_{G \in I} \Gr|_G$. For a triple of graphical collections, we have
\[
    ((\Pop\circ\Q)\circ \Ho)(\Gr)=\bigoplus_{I \in \parti(\Gr)}(\bigoplus_{J \in \parti(\Gr/I)} \Pop((\Gr/I)/J)\otimes \Q((\Gr/I)|_J))\otimes \Ho(\Gr|_I).
\] By Lemma~\ref{intervals}, each partition $J$ of the contracted graph $\Gr/I$ corresponds to a unique partition of the original graph from the interval $[I,\hat{1}]$. Hence, we have
\[
 \bigoplus_{I \in \parti(\Gr)}(\bigoplus_{J \in [I,\hat{1}]} \Pop(\Gr/J)\otimes \Q((\Gr|_J)/_I))\otimes \Ho(\Gr|_I) \cong \bigoplus_{J\geq I} \Pop(\Gr/J)\otimes \Q((\Gr|_J)/_I)\otimes \Ho(\Gr|_I),
\] where the sum on the right hand side ranges over all possible pairs of comparable partitions $J \geq I$. Note that the right hand side is exactly the evaluation of the graphical collection $\Pop\circ(\Q\circ\Ho)$ on $\Gr$
\[
    \bigoplus_{J\geq I} \Pop(\Gr/J)\otimes \Q((\Gr|_J)/_I)\otimes \Ho(\Gr|_I) = \bigoplus_{J \in \parti(\Gr)} \Pop(\Gr/J)\otimes (\Q\circ\Ho)(\Gr|_J)=(\Pop\circ(\Q\circ\Ho))(\Gr).
\] Moreover, similarly to the operad case~\cite{loday2012algebraic}, the associativity isomorphism $(\Pop\circ \Q)\circ \Ho \cong \Pop \circ (\Q \circ \Ho)$ satisfies the axioms of a monoidal category. Finally, for the maximal partition $\hat{1}=\{V_{\Gr}\}$ and the minimal one $\{\{v\}|v \in V_{\Gr}\}$, we have the obvious identities $\Gr|_{\{V_{\Gr}\}}=\Gr$ and $\Gr|_{\{\{v\}\}}=\coprod_{v \in V_{\Gr}} \Gr|_{v} \cong \coprod_{v \in V_{\Gr}} \Path_1$. These assertions immediately imply  the following isomorphisms
\begin{equation*}
    \Pop\circ\text{ }\mathbb{1} \cong \Pop\cong \mathbb{1} \circ \Pop,
\end{equation*} which verify the unit axiom of a monoidal category.
\end{proof}

\begin{defi}[Monoidal definition of contractads]\label{def:monoidal}
A contractad is a monoid in the monoidal category of graphical collections equipped with the contraction product $\circ_{\mathsf{C}}$.
\end{defi}

Let $\Pop$ be a $\C$-valued contractad. According to the definition provided above, we have the unit $\eta\colon \mathbb{1} \rightarrow \Pop$ and the product map $\gamma\colon \Pop\circ\Pop \rightarrow \Pop$ which satisfy axioms of monoids. More explicitly, the unit is given by the morphism $u\colon 1_{\C} \rightarrow \Pop(\Path_1)$ from the unit of the category to the one-vertex component of the underlying graphical collection. In the set-theoretical case, we shall denote by $\Id$ the unique element arising from the latter map. The product map $\gamma$ is given by the collection of morphisms
\[
    \gamma_I^{\Gr}\colon \Pop(\Gr/I)\otimes \bigotimes_{G \in I} \Pop(\Gr|_G)\to \Pop(\Gr),
\] ranging over all graphs and all their partitions. In a dual fashion, we define a \textit{cocontractad} as a comonoid in the category of graphical collections. We shall denote by $\triangle_I^{\Gr}$ the composition of the coproduct map $\Q \rightarrow \Q\circ\Q$ with the projection to the corresponding summand
\[
    \triangle_I^{\Gr}\colon \Q(\Gr) \to \Q(\Gr/I)\otimes \bigotimes_{G \in I} \Q(\Gr|_G). 
\]

Recall that a symmetric operad can be defined by a collection of maps $\circ_i\colon \Orb(n)\otimes \Orb(m) \rightarrow \Orb(n+m-1)$ called infinitesimal compositions. The same idea applies to contractads as follows. Recall that for a graph $\Gr$, each tube $G$ defines the partition $I=\{ G\}\cup \{ \{v\}| v \not\in G\}$, and the related contracted graph is denoted by  $\Gr/G$. Define the infinitesimal composition $\circ^{\Gr}_G\colon \Pop(\Gr/G)\otimes \Pop(\Gr|_G) \rightarrow \Pop(\Gr)$ by the substitution
\[
    \Pop(\Gr/G)\otimes\Pop(\Gr|_G) \cong \Pop(\Gr/G)\otimes\Pop(\Gr|_G)\otimes \bigotimes_{v \not\in G} 1_{\C} \overset{\Id \otimes u^{\otimes}}{\hookrightarrow} \Pop(\Gr/G)\otimes\Pop(\Gr|_G)\otimes \bigotimes_{v \not\in G}\Pop(\Gr|_{\{v\}}) \overset{\gamma}{\rightarrow} \Pop(\Gr).
\] In the set-theoretical case, these compositions can be written in the form $\mu\circ_G^{\Gr}\nu=\gamma(\mu;\nu,\Id,...,\Id)$. As in the case of operads, the collection of infinitesimal compositions recovers a contractad structure. The following proposition is proved by direct inspection.
\begin{prop}["Partial" definition of contractads] \label{partdefi}
A contractad structure on a graphical collection $\Pop$ is equivalent to the datum of morphisms $\circ_G^{\Gr}\colon \Pop(\Gr/G)\otimes\Pop(\Gr|_G) \to \Pop(\Gr)$, ranging over all pairs $(\Gr,G)$ of graph $\Gr$ and tube $G$, satisfying the following axioms:
\begin{itemize}
    \item \textnormal{Unit:} We have a morphism $u\colon 1_{\C} \to \Pop(\Path_1)$ such that
\begin{equation}  
    \begin{tikzcd}
	& \Pop(\Gr)\otimes 1_{\C} && \Pop(\Gr)\otimes\Pop(\Gr|_{\{v\}}) \\
	\Pop(\Gr) &&&& \Pop(\Gr) \\
	&  1_{\C}\otimes\Pop(\Gr) && \Pop(\Gr/\{V_{\Gr}\})\otimes \Pop(\Gr)
	\arrow["\cong", from=2-1, to=1-2]
	\arrow["\cong"', from=2-1, to=3-2]
	\arrow["\Id \otimes u", from=1-2, to=1-4]
	\arrow["\circ^{\Gr}_{\{v\}}", from=1-4, to=2-5]
	\arrow["u\otimes \Id"', from=3-2, to=3-4]
	\arrow["\circ^{\Gr}_{\{V_{\Gr}\}}"', from=3-4, to=2-5]
	\arrow["\Id", from=2-1, to=2-5]
\end{tikzcd}
\end{equation}
    \item \textnormal{Parallel:} For any pair $G_1,G_2$ of disjoint tubes, the following diagram commutes:
    \begin{equation}
    \begin{CD}
    \Pop((\Gr/G_1)/G_2) \bigotimes \Pop(\Gr|_{G_2}) \bigotimes \Pop(\Gr|_{G_1}) @> \circ_{G_2}^{\Gr/G_1} \otimes 1  >>\Pop(\Gr/G_1) \bigotimes \Pop(\Gr|_{G_1})\\
    @VV \circ_{G_1}^{\Gr/G_2}\otimes 1 V @VV \circ_{G_1}^{\Gr} V\\
    \Pop(\Gr/G_2) \bigotimes \Pop(\Gr|_{G_2})  @>\circ_{G_2}^{\Gr}>>\Pop(\Gr)
    \end{CD}
    \end{equation}

    \item \textnormal{Associativity:} For any pair of included tubes $G \subset H$, the following diagram commutes:
    \begin{equation}
    \begin{CD}
    \Pop(\Gr/H) \bigotimes \Pop((\Gr|_H)/G) \bigotimes \Pop(\Gr|_G) @> 1 \otimes \circ_G^H >>\Pop(\Gr/H) \bigotimes \Pop(\Gr|_H)\\
    @VV \circ_{H/G}^{\Gr/G}\otimes 1 V @VV \circ_H^{\Gr} V\\
    \Pop(\Gr/G) \bigotimes \Pop(\Gr|_G)  @>\circ_G^{\Gr}>>\Pop(\Gr)
    \end{CD}
    \end{equation}
    \item \textnormal{Equivariance:} For any tube $G$ and automorphism $\tau\in \Aut(\Gr)$, the following diagram commutes:
    \begin{equation}
    \begin{CD}
     \Pop(\Gr/\tau(G)) \bigotimes \Pop(\Gr|_{\tau(G)}) @>\circ_{\tau(G)}^{\Gr} >> \Pop(\Gr)\\
    @VV \tau/_G \otimes \tau|_{G} V @VV \tau V\\
    \Pop(\Gr/G) \bigotimes \Pop(\Gr|_G)  @>\circ_{G}^{\Gr}>>\Pop(\Gr).
    \end{CD}
    \end{equation}
\end{itemize}
\end{prop}

\subsection{Graph admissible trees and free contractads}\label{subsec:admissibletrees}
In this subsection, we give a combinatorial definition of a contractad based on the notion of \textit{admissible} rooted trees.  A \textit{rooted tree} is a connected directed tree $T$ in which each vertex has at least one input edge and exactly one output edge. Some edges of a tree might be bounded by a vertex at one end only. Such edges are called external. This tree should have exactly one external outgoing edge, output. The endpoint of this edge is called the \textit{root}. The endpoints of incoming external edges that are not vertices are called \textit{leaves}. A tree with a single vertex is called a \textit{corolla}. For a rooted tree $T$ and edge $e\in\Edge(T)$, let $T_e$ be the subtree of $T$  with the root at $e$, and let $T^e$ be the subtree obtained from $T$ by removing $T_e$.

\begin{defi}
For a connected graph $\Gr$, a $\Gr$-\textit{admissible} rooted tree is a rooted tree $T$ with leaves labeled by the vertex set $V_{\Gr}$ of the given graph such that, for each edge $e$ of the tree, the leaves of subtree $T_e$ form a tube of $\Gr$.
\end{defi} 

\begin{figure}[ht] 
  \centering
\[
\vcenter{\hbox{\begin{tikzpicture}[scale=0.6]
    \fill (0,0) circle (2pt);
    \fill (0,1.5) circle (2pt);
    \fill (1.5,0) circle (2pt);
    \fill (1.5,1.5) circle (2pt);
    \draw (0,0)--(1.5,0)--(1.5,1.5)--(0,1.5)-- cycle;
    \node at (-0.25,1.75) {$1$};
    \node at (1.75,1.75) {$2$};
    \node at (1.75,-0.25) {$3$};
    \node at (-0.25,-0.25) {$4$};
  \end{tikzpicture}}}
\qquad
\vcenter{\hbox{\begin{tikzpicture}[scale=0.75]
        \draw (0,0)--(0,1);
        \draw (0,1)--(1,2);
        \draw (0,1)--(-1,2);
        \draw (1,2)--(1.75,2.75);
        \draw (1,2)--(0.25,2.75);
        \draw (-1,2)--(-1.75,2.75);
        \draw (-1,2)--(-0.25,2.75);
        \node at (1.75,3.1) {$4$};
        \node at (0.25,3.1) {$3$};
        \node at (-1.75,3.1) {$1$};
        \node at (-0.25,3.1) {$2$};
  \end{tikzpicture}}}
\qquad 
\vcenter{\hbox{\begin{tikzpicture}[scale=0.75]
        \draw (0,0)--(0,1);
        \draw (0,1)--(1,2);
        \draw (0,1)--(-1,2);
        \draw (0,1)--(0,2);
        \draw (0,2)--(0.75,2.75);
        \draw (0,2)--(-0.75,2.75);
        \node at (-1,2.25) {$1$};
        \node at (-0.75,3.1) {$2$};
        \node at (0.75,3.1) {$3$};
        \node at (1,2.25) {$4$};
  \end{tikzpicture}}}
\qquad 
\vcenter{\hbox{\begin{tikzpicture}[scale=0.75]
        \draw (0,0)--(0,1);
        \draw (0,1)--(1,2);
        \draw (0,1)--(-1,2);
        \draw (-1,2)--(-1.75,2.75);
        \draw (-1,2)--(-1,2.75);
        \draw (-1,2)--(-0.25,2.75);
        \node at (-1.75,3.1) {$3$};
        \node at (-1,3.1) {$4$};
        \node at (-0.25,3.1) {$1$};
        \node at (1,2.25) {$2$};
  \end{tikzpicture}}}
\qquad 
\vcenter{\hbox{\begin{tikzpicture}[scale=0.75]
        \draw[thick] (-1.2,3)--(1.2,0);
        \draw[thick] (-1.2,0)--(1.2,3);
        \draw (0,0)--(0,1);
        \draw (0,1)--(1,2);
        \draw (0,1)--(-1,2);
        \draw (0,1)--(0,2);
        \draw (0,2)--(0.75,2.75);
        \draw (0,2)--(-0.75,2.75);
        \node at (-1,2.25) {$1$};
        \node at (-0.75,3.1) {$2$};
        \node at (0.75,3.1) {$4$};
        \node at (1,2.25) {$3$};
  \end{tikzpicture}}}
\]
\caption{Graph $\Cyc_4$ (on the left side) and examples of $\Cyc_4$-admissible trees. The first three are $\Cyc_4$-admissible, but the fourth is not, since leaves 2,4 do not form a tube.}
\label{roottrees}
\end{figure}
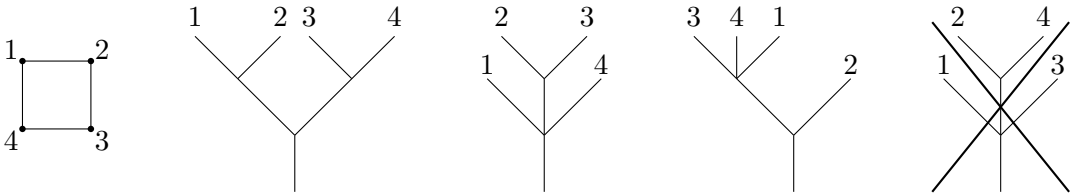

 We denote by $\Tree(\Gr)$ the set of all $\Gr$-admissible rooted trees. Note that a corolla with leaves labeled by the vertex set is always $\Gr$-admissible. Let us describe explicitly admissible trees for particular types of graphs.
 \begin{itemize}
 \label{grroottrees}
     \item For  paths, $\Path_n$-admissible trees are those that can be embedded in the plane such that leaves $[n]=\{1,2,3,...,n\}$ are arranged in increasing order. Indeed, this follows from the  fact  that tubes of $\Path_n$ are ordered intervals of $[n]$.
     \item For  cycles, $\Cyc_n$-admissible trees are those that can be embedded in the plane such that leaves are arranged in cyclic order, as in Figure~\ref{roottrees}.
     \item For complete graphs, $\K_n$-admissible trees are ordinary rooted trees since each vertex subset of a complete graph is a tube.
 \end{itemize}  
Note that subtrees of admissible trees are also admissible. Indeed, for each edge $e$ of a $\Gr$-admissible tree $T$, the subtree $T_e$ is a $\Gr|_{L_e}$-admissible tree, where $L_e$ is the set of leaves of $T_e$, and the subtree $T^e$ is a $\Gr/L_e$-admissible tree. For a partition $I$ of the graph $\Gr$, we define the substitution map
\[
    \Tree(\Gr/I)\times\prod_{G \in I} \Tree(\Gr|_G)\to \Tree(\Gr),
\] which joins roots of $\Gr|_G$-admissible trees to corresponding leaves of $\Gr/I$-admissible trees, as in Figure~\ref{substitution}.

\begin{figure}[ht]
  \centering
  \[
  \vcenter{\hbox{\begin{tikzpicture}[scale=0.7]
    \fill (-0.63,1.075)  circle (2pt);
    \node at (-0.9,1.4) {1};
    \fill (0.63,1.075)  circle (2pt);
    \node at (0.9,1.4) {2};
    \fill (1.22,0) circle (2pt);
    \node at (1.6,0) {3};
    \fill (0.63,-1.075)  circle (2pt);
    \node at (0.9,-1.4) {4};
    \fill (-0.63,-1.075)  circle (2pt);
    \node at (-0.9,-1.4) {5};
    \fill (-1.22,0) circle (2pt);
    \node at (-1.6,0) {6};
    \draw (-1.22,0)--(-0.63,1.075)--(0.63,1.075)--(1.22,0)--(0.63,-1.075)--(-0.63,-1.075)--cycle;
    \draw[dashed] (-1.6,0)[rounded corners=15pt]--(-0.4,1.7)[rounded corners=15pt]--(-0.4,-1.7)[rounded corners=12pt]--cycle;
    \draw[dashed] (1.6,0)[rounded corners=15pt]--(0.4,1.7)[rounded corners=15pt]--(0.4,-1.7)[rounded corners=12pt]--cycle;
    \end{tikzpicture}}}
    \quad\quad
    \text{\huge(} \vcenter{\hbox{\begin{tikzpicture}[scale=0.6]
        \draw (0,0)--(0,1);
        \draw (0,1)--(1,2);
        \draw (0,1)--(-1,2);
        \node at (-1.2,2.3) {\small $\{1,5,6\}$};
        \node at (1,2.3) {\small $\{2,3,4\}$};
     \end{tikzpicture}}};
     \vcenter{\hbox{\begin{tikzpicture}[scale=0.6]
        \draw (0,0)--(0,1);
        \draw (0,1)--(1,2);
        \draw (0,1)--(-1,2);
        \draw (-1,2)--(-1.75,2.75);
        \draw (-1,2)--(-0.25,2.75);
        \node at (-1.75,3.1) {$1$};
        \node at (-0.25,3.1) {$6$};
        \node at (1,2.3) {$5$};
     \end{tikzpicture}}},
     \vcenter{\hbox{\begin{tikzpicture}[scale=0.6]
        \draw (0,0)--(0,1);
        \draw (0,1)--(1,2);
        \draw (0,1)--(-1,2);
        \draw (1,2)--(1.75,2.75);
        \draw (1,2)--(0.25,2.75);
        \node at (-1,2.3) {$2$};
        \node at (0.25,3.1) {$3$};
        \node at (1.75,3.1) {$4$};
    \end{tikzpicture}}}\text{\huge)}
\quad \longrightarrow
    \vcenter{\hbox{\begin{tikzpicture}[scale=0.6]
        \draw (0,0)--(0,1);
        \draw (0,1)--(1,2);
        \draw (0,1)--(-1,2);
        \draw (1,2)--(1.75,2.75);
        \draw (1,2)--(0.25,2.75);
        \draw (1.75,2.75)--(2.5,3.5);
        \draw (1.75,2.75)--(1,3.5);
        \draw (-1,2)--(-1.75,2.75);
        \draw (-1,2)--(-0.25,2.75);
        \draw (-1.75,2.75)--(-2.5,3.5);
        \draw (-1.75,2.75)--(-1,3.5);
        \node at (-2.5,3.85) {\small$1$};
        \node at (-1,3.85) {\small$6$};
        \node at (-0.25,3.1) {\small$5$};
        \node at (0.25,3.1) {\small$2$};
        \node at (1,3.85) {\small$3$};
        \node at (2.5,3.85) {\small$4$};
        
  \end{tikzpicture}}}
  \]
  \caption{Example of substitution.}
  \label{substitution}
\end{figure}
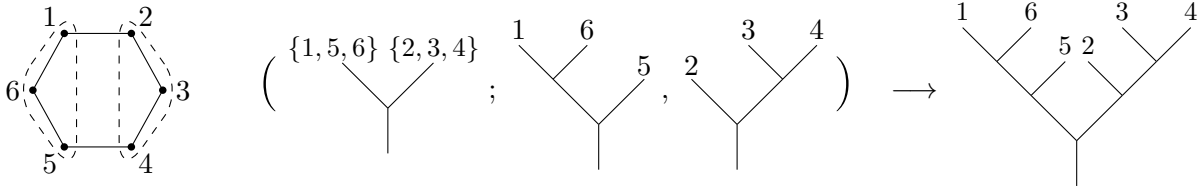

Recall that an operad can be defined as an algebra over the monad of rooted trees \cite[Ch.~5]{loday2012algebraic}. A similar story applies to contractads if we replace ordinary trees with admissible ones. For a $\Gr$-admissible tree $T$ and a vertex $v$, let $e$ be the output edge and $e_1,e_2,...,e_k$ be the input edges. Note that the collection of leaf sets $\{L_{e_1},L_{e_2},...,L_{e_k}\}$ forms a partition of the induced subgraph $\Gr|_{L_e}$. Define the \textit{input graph of the vertex} by $\In(v):=(\Gr|_{L_e})/\{L_{e_1},...,L_{e_k}\}$.
\begin{defi}
    The admissible rooted trees endofunctor is the functor
    \[
    \T\colon \GrCol \to \GrCol
    \] defined by the rule
    \[
    \T(\Pop)(\Gr)=\bigoplus_{T \in \Tree(\Gr)} \bigotimes_{v \in \Ver(T)} \Pop(\In(v)).
    \]
\end{defi}
\noindent In the linear case, each element in the component $\T(\F)(\Gr)$ can be expressed as a sum of $\Gr$-admissible trees whose vertices $v$ are labeled by elements of $\F(\In(v))$. We endow this endofunctor with a monad structure as follows. The unit transformation $\Id_{\GrCol} \Rightarrow \T$ is the natural inclusion $\F \to \T(\F)$ which corresponds to the corolla in each component. The natural transformation $\T\circ\T \Rightarrow \T$  is induced by the substitution of graph-admissible trees $\T(\T(\F)) \rightarrow \T(\F)$. Checking that this data defines a monad structure proceeds \textit{mutatis mutandis} in the same way as for the operad case~\cite[Lem.~5.5.2]{loday2012algebraic}. 

\begin{defi}[Monadic definition of contractads]\label{def:monadic}
A contractad is an algebra over the admissible tree monad. In other words, it is a graphical collection equipped with a structure map
\begin{equation*}
    \T(\Pop) \to \Pop,
\end{equation*}
compatible with the monad structure on $\mathbb{T}$.
\end{defi}

\noindent The monadic definition of contractads allows us to define a free contractad as a free $\T$-algebra.
\begin{defi}
The free contractad on a graphical collection $\E$ is the contractad $\T(\E)$ with the structure map $\T(\T(\E)) \to \T(\E)$ arising from the monad structure.
\end{defi}

\noindent Moreover, we can discuss presentations of contractads in terms of generators and relations as follows. An \textit{ideal} of a contractad $\Pop$ is a graphical subcollection $\mathcal{I} \subset \Pop$ invariant under the product map on both sides. Like in the case of algebras/groups/operads, the quotient of a contractad $\Pop/\I$ by an ideal is also a contractad.

\begin{defi}
The contractad presented by generators $\E$ and relations $\R \subset \T(\E)$ is the quotient contractad
\[
\T(\E)/\langle \R \rangle,
\] where $\langle \R \rangle$ is the minimal ideal containing the subcollection $\R$.
\end{defi}

\subsection{Relations to classical algebraic structures}\label{subsec:operads}

We can explicitly relate contractads to classical operad-like structures, such as (non)symmetric operads and twisted associative algebras. The key idea is to consider special families of graphs that are closed under induced and contracted graphs.

First, consider the family of complete graphs $\{\K_n\}_{n\geq 1}$. We have $\Aut(\K_n)\cong \Sigma_n$ since automorphisms of complete graphs are just permutations of vertices. Therefore, the restriction of a graphical collection to complete graphs defines a $\Sigma$-module by the rule
\[
\K_{*}(\Orb)(n):=\Orb(\K_n).
\]
\begin{prop}
For every contractad $\Pop$, the $\Sigma$-module $\K_*(\Pop)$ has a natural structure of a symmetric operad obtained from that of $\Pop$.
\end{prop}
\begin{proof}
Recall that a symmetric operad is an algebra over the monad of rooted trees $\T_{\Sigma}$ \cite{loday2012algebraic}. As we have mentioned before in \eqref{grroottrees}, $\K_n$-admissible rooted trees are just ordinary rooted trees. In other words, the restriction to complete graphs preserves monads
\[
\K_{*}(\T_{\GrCol}(\Orb))\cong \T_{\Sigma}( \K_{*}(\Orb)).
\]So, the functor $\K_{*}$ sends $\T_{\GrCol}$-algebras to $\T_{\Sigma}$-algebras.
\end{proof}

Next, consider the family of paths $\{\Path_n\}_{n\geq 1}$. Note that each path $\Path_n$ has only one non-trivial automorphism which relabels vertices in the reverse order.  Therefore, the restriction of a graphical collection to paths defines a non-symmetric collection
\[
\Path_{*}(\Orb)(n):=\Orb(\Path_n),
\] with an additional involution $(-)^*\colon \Path_{*}(\Orb) \to \Path_{*}(\Orb)$ arising  from the relabeling of vertices in paths. Recall that a non-symmetric operad is an algebra over the monad of planar rooted trees $\T_{\mathrm{pl}}$. Referring to \cite{dotsenko2024reconnectads}, a \textit{mirrored ns operad} is an ns operad  $\mathcal{Q}$ with an involution $(-)^*\colon \mathcal{Q} \to \mathcal{Q}$ such that 
\[
(\mu \circ_i \nu)^*=\mu^*\circ_{n-i+1} \nu^*, \quad \mu \in \mathcal{Q}(n), \nu \in \mathcal{Q}(m).
\] 
\begin{prop}
For every contractad $\Pop$, the non-symmetric collection $\Path_*(\Pop)$ has a natural structure of a mirrored ns operad obtained from that of $\Pop$.
\end{prop}
\begin{proof}
As we have mentioned before~\ref{grroottrees}, $\Path_n$-admissible rooted trees are precisely planar rooted trees. In other words, restriction to paths preserves monads
\[
\Path_{*}(\T_{\GrCol}(\Orb))\cong \T_{\mathrm{pl}}( \Path_{*}(\Orb)).
\] So, the restriction of a contractad is a non-symmetric operad. The mirrored structure follows from Equivariance Axiom~\ref{partdefi} applied to paths. 
\end{proof}

Finally, consider the family of stellar graphs $\{\St_n\}_{n\geq 0}$. In this setting, we shall consider only \textit{connected} contractads: $\Pop(\Path_1)=1_{\C}$. For the graph $\St_n$, we have $\Aut(\St_n)\cong \Sigma_n$ since automorphisms of a stellar graph are vertex permutations that stabilize the core.  Therefore, the restriction of a graphical collection to stellar graphs defines a $\Sigma$-module by the rule
\[
\St_*(\Orb)(I):=\Orb(\St_{I}),
\] where $\St_{I}$ is the stellar graph on the  vertex set $I\cup \{*\}$ with core $*$. 
Recall that a \textit{twisted associative algebra} is a monoid in the category of $\Sigma$-modules equipped with the \textit{Cauchy product}~\cite[Ch.~4]{bremner2016algebraic}
\[
(\mathcal{A} \cdot \B)(I)=\bigoplus_{I=J\sqcup K} \mathcal{A}(J)\otimes \B(K).
\]
\begin{prop}
For every connected contractad $\Pop$, the $\Sigma$-module $\St_*(\Pop)$ has a natural structure of twisted associative algebra obtained from that of $\Pop$.
\end{prop}
\begin{proof}
Tubes of the stellar graph $\St_I$ are singletons or vertex subsets containing the core. Hence, a partition of the stellar graph is a partition of the vertex set that is made up of one non-trivial block $J\cup \{0\}$ containing core. Moreover, the resulting contracted graph $\St_{I}/_{J\cup \{0\}} \cong \St_{I\setminus J}$ and induced subgraph $\St_{I}|_{J\cup\{0\}}\cong \St_{J}$ are also stellar graphs.
 \[
  \vcenter
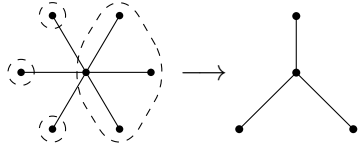
{\hbox{\begin{tikzpicture}[scale=0.7]
    \fill (0,0)  circle (2pt);
    \fill (-0.63,1.075)  circle (2pt);
    \fill (0.63,1.075)  circle (2pt);
    \fill (1.22,0) circle (2pt);
    \fill (0.63,-1.075)  circle (2pt);
    \fill (-0.63,-1.075)  circle (2pt);
    \fill (-1.22,0) circle (2pt);
    \draw (0,0)--(-0.63,1.075);
    \draw (0,0)--(0.63,1.075);
    \draw (0,0)--(-0.63,-1.075);
    \draw (0,0)--(0.63,-1.075);
    \draw (0,0)--(1.22,0);
    \draw (0,0)--(-1.22,0);
    \draw[dashed] (-1.22,0) circle  (7pt);
    \draw[dashed] (-0.63,1.075) circle  (7pt);
    \draw[dashed] (-0.63,-1.075) circle  (7pt);
    \draw[dashed] (1.6,0)[rounded corners=15pt]--(0.63,1.7)[rounded corners=15pt]--(-0.4,0)[rounded corners=15pt]--(0.63,-1.7)[rounded corners=12pt]--cycle;
    \end{tikzpicture}}}
     \longrightarrow \vcenter{\hbox{\begin{tikzpicture}[scale=0.75]
    \fill (0,0) circle (2pt);
    \fill (1,1) circle (2pt);
    \fill (1,2) circle (2pt);
    \fill (2,0) circle (2pt);
    \draw (0,0)--(1,1)--(1,2);
    \draw (2,0)--(1,1);
    \end{tikzpicture}}}
    \]These assertions imply that the restriction to stellar graphs sends the contraction product of connected graphical collections to the Cauchy product of related $\Sigma$-modules
\[
\St_{*}(\Pop \circ \mathcal{Q})(I) \cong \bigoplus_{J \subset I} \Pop(\St_{I\setminus J}) \otimes \mathcal{Q}(\St_J) \cong \St_{*}(\Pop)\cdot \St_{*}(\mathcal{Q})(I).
\]

\end{proof}

\section{Examples of contractads}\label{sec:examples}
In this section, we discuss examples of contractads arising in algebra, combinatorics, topology, and geometry.  Most of them can be viewed as graphical counterparts of familiar operads.
\subsection{Commutative contractad}
Consider the simplest example of a contractad. Let $\mathsf{k}$ be an arbitrary field.
\begin{defi}
The commutative contractad $\Com$ is the $\mathsf{k}$-linear contractad whose underlying graphical collection is given by the rule
\[
\Com(\Gr)=\mathsf{k},
\] with the infinitesimal compositions of the form $\mathsf{k}\otimes\mathsf{k} \overset{\cong}{\rightarrow}\mathsf{k}$.
\end{defi}
Note that this contractad can be defined in any symmetric monoidal category $\C$ if we replace the field $\mathsf{k}$ with a unit. If we restrict this contractad to paths, the resulting non-symmetric operad $\Path_{*}(\Com)$ coincides with the ns operad $\mathsf{As}$ of associative algebras. Similarly, the restriction to complete graphs gives us $\K_{*}(\Com)\cong \mathsf{Com}$ the symmetric operad of commutative algebras. These observations explain the name of this contractad. The letters "gc" in the name stand for "graphical contractad". Let us list some properties of this contractad:
\begin{itemize}
    \item The contractad $\Com$ has a quadratic presentation (Proposition~\ref{quadcom}).
    \item This contractad is Koszul (Theorem~\ref{thm:comkoszul}).
    \item This contractad admits a quadratic Gr\"obner basis (Proposition~\ref{quadcom}).
\end{itemize}
A quadratic presentation of this contractad is presented in the following proposition. We leave the proof until Section~\ref{sec:grobner}.
\begin{prop}
\label{compres}
The contractad $\Com$ is generated by a symmetric generator $m$ in the component $\Path_2$, satisfying the relations
\begin{gather}
    m \circ_{\{1,2\}}^{\mathsf{P_3}} m = m \circ_{\{2,3\}}^{\mathsf{P_3}} m
    \\
    m \circ_{\{1,2\}}^{\mathsf{K_3}} m = m \circ_{\{2,3\}}^{\mathsf{K_3}} m
\end{gather}
\end{prop}

\noindent\textbf{Remark:} Note that we do not explicitly mention the relations produced by graph automorphisms. For example, the action of the transposition $(12)$ on the second relation results in the following relation:
\[
 (m \circ_{\{1,2\}}^{\mathsf{K_3}} m - m \circ_{\{2,3\}}^{\mathsf{K_3}} m)^{(12)}=m \circ_{\{1,2\}}^{\mathsf{K_3}} m - m \circ_{\{1,3\}}^{\mathsf{K_3}} m.
\]
\subsection{Lie contractad}
Let us give a graphical counterpart of the operad $\mathsf{Lie}$ of Lie algebras.
\begin{defi}
The \textit{Lie contractad} $\Lie$ is the contractad generated by an anti-symmetric generator $b$ in the component $\Path_2$, satisfying the relations
\begin{gather}
    b \circ_{\{1,2\}}^{\Path_3} b =  b \circ_{\{2,3\}}^{\Path_3} b,
    \\
    b\circ_{\{1,2\}}^{\K_3} b + (b\circ_{\{1,2\}}^{\K_3} b)^{(123)} + (b\circ_{\{1,2\}}^{\K_3} b)^{(321)}=0.
\end{gather}
\end{defi}
\noindent Here, the choice of name $\Lie$ is motivated by the observation that the restriction to complete graphs $\K_{*}(\Lie)\cong \mathsf{Lie}$ gives us the operad of Lie algebras. In the next section, we explain the choice of associative relation in the component $\Path_3$.  Let us list some properties of this contractad:
\begin{itemize}
    \item The dimension of each component is given by the formula $\dim \Lie(\Gr)=|\mu_{\parti(\Gr)}(\hat{0},\hat{1})|$, where $\mu_{\parti(\Gr)}$ is the M\"obius function of the poset of graph-partitions (Corollary~\ref{cor:liekoszul}). For simplicity, we denote the number $\mu_{\parti(\Gr)}(\hat{0},\hat{1})$ by $\mu(\Gr)$. Also, this number appears as the first non-zero coefficient of the chromatic polynomial $\chi_{\Gr}(t)=\mu(\Gr)t+O(t^2)$ (Corollary~\ref{sled:hilbertseries}). For particular types of graphs, we have
    \[
    |\mu(T)|=1, \text{ if } T \text{ is a tree};\quad |\mu(\Cyc_n)|=n-1; \quad |\mu(\K_n)|=(n-1)!.
    \]
    \item This contractad is Koszul (Corollary~\ref{cor:liekoszul}).
    \item This contractad admits a quadratic Gr\"obner basis (Corollary~\ref{liegrob}).
\end{itemize}

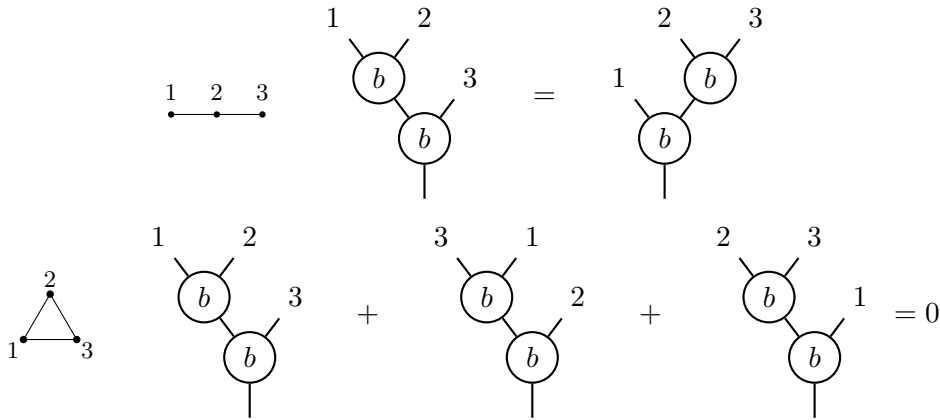
\begin{figure}[ht]
    \centering
    \caption{Relations in $\Lie$ in terms of admissible trees.}
    \begin{gather*}
        \vcenter{\hbox{\begin{tikzpicture}[scale=0.6]
        \fill (0,0) circle (2pt);
        \node at (0,0.5) {\footnotesize$1$};
        \fill (1,0) circle (2pt);
        \node at (1,0.5) {\footnotesize$2$};
        \fill (2,0) circle (2pt);
        \node at (2,0.5) {\footnotesize$3$};
        \draw (0,0)--(1,0)--(2,0);    
        \end{tikzpicture}}}
        \quad
        \vcenter{\hbox{\begin{tikzpicture}[
        scale=0.8,
        vert/.style={circle,  draw=black!30!black, thick, minimum size=1mm},
        leaf/.style={circle, thick, minimum size=1mm},
        edge/.style={-,black!30!black, thick},
        ]
        \node[leaf] (l1) at (-1.5,3) {$1$};
        \node[leaf] (l2) at (0,3) {$2$};
        \node[leaf] (l3) at (0.75,2) {$3$};
        \node[vert] (1) at (0,1) {$b$};
        \node[vert] (2) at (-0.75,2) {$b$};
        \draw[edge] (0,0)--(1);
        \draw[edge] (1)--(2);
        \draw[edge] (2)--(l1);
        \draw[edge] (2)--(l2);
        \draw[edge] (1)--(l3);
    \end{tikzpicture}}}
    \quad
    =
    \quad
    \vcenter{\hbox{\begin{tikzpicture}[
        scale=0.8,
        vert/.style={circle,  draw=black!30!black, thick, minimum size=1mm},
        leaf/.style={circle, thick, minimum size=1mm},
        edge/.style={-,black!30!black, thick},
        ]
        \node[leaf] (l1) at (-0.75,2) {$1$};
        \node[leaf] (l2) at (0,3) {$2$};
        \node[leaf] (l3) at (1.5,3) {$3$};
        \node[vert] (1) at (0,1) {$b$};
        \node[vert] (2) at (0.75,2) {$b$};
        \draw[edge] (0,0)--(1);
        \draw[edge] (1)--(2);
        \draw[edge] (1)--(l1);
        \draw[edge] (2)--(l2);
        \draw[edge] (2)--(l3);
    \end{tikzpicture}}}
    \\
    \vcenter{\hbox{\begin{tikzpicture}[scale=0.7]
    \fill (-0.5,0) circle (2pt);
    \node at (-0.7,-0.2) {\footnotesize$1$};
    \fill (0.5,0) circle (2pt);
    \node at (0.7,-0.2) {\footnotesize$3$};
    \fill (0,0.86) circle (2pt);
    \node at (0,1.14) {\footnotesize$2$};
    \draw (-0.5,0)--(0,0.86)--(0.5,0)--cycle;    
    \end{tikzpicture}}}
    \quad
    \vcenter{\hbox{\begin{tikzpicture}[
        scale=0.8,
        vert/.style={circle,  draw=black!30!black, thick, minimum size=1mm},
        leaf/.style={circle, thick, minimum size=1mm},
        edge/.style={-,black!30!black, thick},
        ]
        \node[leaf] (l1) at (-1.5,3) {$1$};
        \node[leaf] (l2) at (0,3) {$2$};
        \node[leaf] (l3) at (0.75,2) {$3$};
        \node[vert] (1) at (0,1) {$b$};
        \node[vert] (2) at (-0.75,2) {$b$};
        \draw[edge] (0,0)--(1);
        \draw[edge] (1)--(2);
        \draw[edge] (2)--(l1);
        \draw[edge] (2)--(l2);
        \draw[edge] (1)--(l3);
    \end{tikzpicture}}}
    \quad
    +
    \quad
        \vcenter{\hbox{\begin{tikzpicture}[
        scale=0.8,
        vert/.style={circle,  draw=black!30!black, thick, minimum size=1mm},
        leaf/.style={circle, thick, minimum size=1mm},
        edge/.style={-,black!30!black, thick},
        ]
        \node[leaf] (l1) at (-1.5,3) {$3$};
        \node[leaf] (l2) at (0,3) {$1$};
        \node[leaf] (l3) at (0.75,2) {$2$};
        \node[vert] (1) at (0,1) {$b$};
        \node[vert] (2) at (-0.75,2) {$b$};
        \draw[edge] (0,0)--(1);
        \draw[edge] (1)--(2);
        \draw[edge] (2)--(l1);
        \draw[edge] (2)--(l2);
        \draw[edge] (1)--(l3);
    \end{tikzpicture}}}
    \quad
    +
    \quad
        \vcenter{\hbox{\begin{tikzpicture}[
        scale=0.8,
        vert/.style={circle,  draw=black!30!black, thick, minimum size=1mm},
        leaf/.style={circle, thick, minimum size=1mm},
        edge/.style={-,black!30!black, thick},
        ]
        \node[leaf] (l1) at (-1.5,3) {$2$};
        \node[leaf] (l2) at (0,3) {$3$};
        \node[leaf] (l3) at (0.75,2) {$1$};
        \node[vert] (1) at (0,1) {$b$};
        \node[vert] (2) at (-0.75,2) {$b$};
        \draw[edge] (0,0)--(1);
        \draw[edge] (1)--(2);
        \draw[edge] (2)--(l1);
        \draw[edge] (2)--(l2);
        \draw[edge] (1)--(l3);
    \end{tikzpicture}}}=0
    \end{gather*}
\end{figure}

\subsection{Endomorphism contractad}

For a vector space $V$, we can associate the endomorphism operad $\End_V$ encoding multi-ary operations on $V$. Let us construct a graphical counterpart of such operads. For a symmetric monoidal category $\C$ and object $A$, we define the $\C$-valued \textit{endomorphism contractad} $\End_A$ by putting
\begin{equation*}
    \End_A(\Gr)=\Hom_{\C}(A^{\otimes V_{\Gr}}, A),
\end{equation*} and the product map is given by the usual composition
\begin{gather*}
    \gamma^{\Gr}_I\colon \End_A(\Gr/I)\otimes\bigotimes_{G\in I} \End_A(\Gr|_G) \to \End_A(\Gr)
    \\
    \gamma^{\Gr}_I(g;f_1,f_2,...,f_k)=g\circ(f_1\otimes f_2\otimes...\otimes f_k).
\end{gather*}
In the case when $A=\mathsf{k}$ is a field, the corresponding endomorphism contractad $\End_{\mathsf{k}} \cong \Com$ coincides with the commutative contractad.\\

\noindent\textbf{Remark:} Similarly to operads, for a contractad $\Pop$, we could define a $\Pop$-algebra structure on $A$ as a morphism of contractads $\Pop\rightarrow \End_A$. Unfortunately, the studying of algebras over contractads is not as interesting as in the case of operads. The reason is that, unlike operads, identities in algebras over an arbitrary contractad do not recover the contractad itself unambiguously. It follows from the observation that relations in a contractad "stick together" when we consider algebras. For example, if we take the contractad $\Lie$, the associated algebras are precisely Lie algebras with the additional identity $[x,[y,z]]=0$. Indeed, from the relations in this contractad, we conclude that the bracket $[-,-]$ corresponding to generator $b$ must simultaneously satisfy both the Jacobi and associative identities, hence the composition of two brackets is always zero.

\subsection{Contractad of (rooted) spanning trees}\label{subsec:spantrees}

In this subsection, we define combinatorial examples of contractads based on the spanning trees of graphs. More specifically, we define the contractad of spanning trees $\ST^{\vee}$ and the contractad of rooted spanning trees $\RST^{\vee}$. These contractads provide graphical counterparts of operads described in \cite{chapoton2001pre},\cite{aval2019graph}. The last example $\RST^{\vee}$ is a graphical counterpart of the operad $\mathsf{preLie}$ of pre-Lie algebras.

Let $\Gr$ be a connected graph. A spanning tree of $\Gr$ is a subgraph (not induced) $T$ of $\Gr$ that is a tree on the same vertex set as $\Gr$. We denote by $\ST(\Gr)$ the vector space generated by spanning trees. There is the right action on spanning trees $\Aut(\Gr) \curvearrowright \ST(\Gr)$ induced by the permutations of edges $\Aut(\Gr) \curvearrowright E_{\Gr}$. We endow the resulting graphical collection of spanning trees $\ST$ with a cocontractad structure as follows. For a spanning tree $T$ of a graph $\Gr$ we have an inclusion of posets $\parti(T)\subset\parti(\Gr)$ since each partition of the spanning tree is also a partition of the underlying graph. Moreover, for each partition $I$ of $T$, the contracted subgraph $T/I \subset \Gr/I$ is also a spanning tree. These observations allow us to define the coproduct map $\triangle\colon \ST \rightarrow \ST\circ \ST$ by the rule
\begin{gather*}
\triangle^{\Gr}_I\colon \ST(\Gr) \to \ST(\Gr/I)\otimes\bigotimes_{G \in I} \ST(\Gr|_G)
\\
T \mapsto \begin{cases}
(T/I;T|_{G_1},\cdots, T|_{G_k}) \text{, if } I \in \parti(T)\subset \parti(\Gr)
\\
0\text{, otherwise}.
\end{cases}
\end{gather*}
It is easy to check that this map endows $\ST$ with a cocontractad structure.

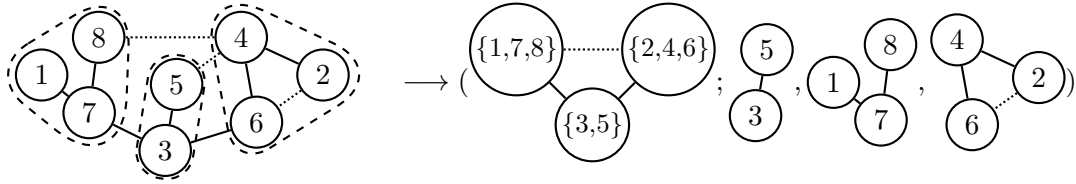
\begin{figure}[ht]
    \centering
    \[\vcenter{\hbox{\begin{tikzpicture}[
        scale=0.5,
        vert/.style={circle,  draw=black!30!black, thick, minimum size=1mm},
        edge/.style={-,black!30!black, thick},
        gedge/.style={-, densely dotted, thick},
        ]
        \node[vert] (3) at (0,0) {$3$};
        \node[vert] (5) at (0.3,1.75) {$5$};
        \draw[edge] (3)--(5);
        \draw[dashed, thick, rounded corners=10pt,rotate=-9] (-0.75,-0.75) rectangle ++(1.5,3.3);
        \node[vert] (1) at (-3.25,2) {$1$};
        \node[vert] (7) at (-2,1) {$7$};
        \node[vert] (8) at (-1.75,3) {$8$};
        \draw[edge] (7)--(8);
        \draw[edge] (1)--(7);
        \draw [dashed, thick] (-4.8,2.1)[rounded corners=25pt] - - (-0.9,4.5)[rounded corners=33pt] - - (-1.1,-0.9)[rounded corners=20pt] --cycle;
        \node[vert] (6) at (2.4,0.75) {$6$};
        \node[vert] (4) at (2,3) {$4$};
        \node[vert] (2) at (4.15,2) {$2$};
        \draw[edge] (6)--(4);
        \draw[edge] (2)--(4);
        \draw[gedge]  (2)--(6);
        \draw [dashed, thick]  (1.85,-0.5)[rounded corners=25pt]  - - (6,2.1)[rounded corners=27pt] - - (0.9,4.5)[rounded corners=15pt] --cycle;
        \draw[edge] (3)--(7); 
        \draw[edge] (3)--(6);
        \draw[gedge] (4)--(5);
        \draw[gedge] (4)--(8);
    \end{tikzpicture}}}
    \longrightarrow (
    \vcenter{\hbox{\begin{tikzpicture}[
        scale=0.5,
        vert/.style={inner sep=1pt, circle,  draw, thick},
        edge/.style={-,black!30!black, thick},
        gedge/.style={-, densely dotted, thick},
        ]
        \node[vert] (2) at (-2,2) {\small \{1,7,8\}};
        \node[vert] (1) at (0,0) {\small\{3,5\}};
        \node[vert] (3) at (2,2) {\small\{2,4,6\}};
        \draw[edge] (1)--(2);
        \draw[edge] (1)--(3);
        \draw[gedge] (2)--(3);
    \end{tikzpicture}}};
    \vcenter{\hbox{\begin{tikzpicture}[
        scale=0.5,
        vert/.style={circle,  draw=black!30!black, thick, minimum size=1mm},
        edge/.style={-,black!30!black, thick},
        gedge/.style={-, densely dotted, thick},
        ]
        \node[vert] (3) at (0,0) {$3$};
        \node[vert] (5) at (0.3,1.75) {$5$};
        \draw[edge] (3)--(5);  
    \end{tikzpicture}}},
    \vcenter{\hbox{\begin{tikzpicture}[
        scale=0.5,
        vert/.style={circle,  draw=black!30!black, thick, minimum size=1mm},
        edge/.style={-,black!30!black, thick},
        gedge/.style={-, densely dotted, thick},
        ]
        \node[vert] (7) at (-2,1) {$7$};
        \node[vert] (8) at (-1.75,3) {$8$};
        \node[vert] (1) at (-3.25,2) {$1$};
        \draw[edge] (7)--(8);
        \draw[edge] (1)--(7); 
    \end{tikzpicture}}},
    \vcenter{\hbox{\begin{tikzpicture}[
        scale=0.5,
        vert/.style={circle,  draw=black!30!black, thick, minimum size=1mm},
        edge/.style={-,black!30!black, thick},
        gedge/.style={-, densely dotted, thick},
        ]
        \node[vert] (6) at (2,0.75) {$6$};
        \node[vert] (4) at (1.6,3) {$4$};
        \node[vert] (2) at (3.75,2) {$2$};
        \draw[edge] (6)--(4);
        \draw[edge] (2)--(4);
        \draw[gedge] (2)--(6);      
    \end{tikzpicture}}})
    \]
    \caption{Example of the cocontractad map. The edges out of spanning trees are dotted.}
\end{figure} 

\noindent If we replace each component with its dual, we get the dual contractad $\ST^{\vee}$ of spanning trees with product map $\gamma=\triangle^{\vee}$. If we restrict this contractad to complete graphs, the resulting operad coincides with the operad of labeled trees described in \cite{aval2019graph}. The description of the latter operad in terms of generators and relations is a complicated task since it has infinitely many generators.

If we replace spanning trees with rooted ones we get the rooted spanning trees cocontractad, denoted $\RST$. Each component $\RST(\Gr)$ is the vector space generated by spanning trees with a marked vertex: the root. The coproduct map $\RST\rightarrow\RST\circ\RST$ is given by the rule
\begin{gather*}
\triangle^{\Gr}_I\colon \RST(\Gr) \to \RST(\Gr/I)\otimes\bigotimes_{G \in I} \RST(\Gr|_G)
\\
(T,r) \mapsto \begin{cases}
((T/I,r');(T|_{G_1},r_1),\cdots, (T|_{G_k},r_k) \text{, if } I \in \parti(T)
\\
0\text{, otherwise}.
\end{cases},
\end{gather*} where the choice of roots on the right hand side is defined as follows: for the contracted subtree $T/I$, the root $r'$ is the block of the partition containing $r$; for the block $G_i$, the root $r_i$ is the nearest vertex to $r$ from this block.

These two examples of cocontractads are connected by the cocontractad morphism
\[
\forget\colon \RST \twoheadrightarrow \ST,
\] induced by forgetting roots. The dual morphism $\forget^{\vee}\colon \ST^{\vee} \hookrightarrow \RST^{\vee}$ defines an embedding. If we restrict the contractad $\RST^{\vee}$ to the complete graphs, the  resulting symmetric operad $\K_{*}(\RST^{\vee})$ coincides with the rooted trees operad introduced by Chapoton and Livernet \cite{chapoton2001pre}. They proved that this operad is isomorphic  to the operad of pre-Lie algebras. Their  result can be generalized in the following way.

Let $\Gr$ be a graph and $T_1,T_2$ two disjoint rooted subtrees, such that their vertex sets form a partition $V_{\Gr}=V_{T_1}\coprod V_{T_2}$. We define their $\star$-product by the rule
\[
T_1\star T_2:=\gamma^{\Gr}_{\{V_{T_1},V_{T_2}\}}(\mu,T_1,T_2),
\] where $\mu$ is the rooted spanning tree of $\Gr/\{V_{T_2},V_{T_1}\}\cong \Path_2$ with the root $\{V_{T_1}\}$. More explicitly, the product $T_1\star T_2$ consists of grafting the root of $T_2$ on every adjacent vertex of $T_1$.
\begin{lemma}\label{lemma:prelieidentity}
For a graph $\Gr$ and a triple of non-intersecting rooted subtrees $T_1,T_2,T_3$, we have
\[
 (T_1\star T_2)\star T_3-T_1\star (T_2\star T_3)=(T_1\star T_3)\star T_2-T_1\star (T_3\star T_2)
\]
\end{lemma}
\begin{proof} By direct computation, we have
\[
(T_1\star T_2)\star T_3-T_1\star (T_2\star T_3)=(\sum \vcenter{\hbox{\begin{tikzpicture}[
        scale=0.5,
        root/.style={inner sep=2pt, circle,draw, dashed, thick},
        vert/.style={inner sep=2pt, circle,draw, thick},
        leaf/.style={inner sep=2pt, rectangle, thick},
        edge/.style={-,black!30!black, thick},
        dedge/.style={-,black!30!black,densely dotted, thick},
        ]
        \node[root] (1) at (0,0) {\scriptsize$T_1$};
        \node[vert] (2) at (0,1.5) {\scriptsize$T_2$};
        \node[vert] (3) at (0,3) {\scriptsize$T_3$};
        \draw[edge] (1)--(2)--(3);
    \end{tikzpicture}}}-\sum \vcenter{\hbox{\begin{tikzpicture}[
        scale=0.5,
        root/.style={inner sep=2pt, circle,draw, dashed, thick},
        vert/.style={inner sep=2pt, circle,draw, thick},
        leaf/.style={inner sep=2pt, rectangle, thick},
        edge/.style={-,black!30!black, thick},
        dedge/.style={-,black!30!black,densely dotted, thick},
        ]
        \node[root] (1) at (0,0) {\scriptsize$T_1$};
        \node[vert] (2) at (-1.131,1.3) {\scriptsize$T_2$};
        \node[vert] (3) at (1.131,1.3) {\scriptsize$T_3$};
        \draw[edge] (1)--(2);
        \draw[edge] (1)--(3);
    \end{tikzpicture}}})-\sum \vcenter{\hbox{\begin{tikzpicture}[
        scale=0.5,
        root/.style={inner sep=2pt, circle,draw, dashed, thick},
        vert/.style={inner sep=2pt, circle,draw, thick},
        leaf/.style={inner sep=2pt, rectangle, thick},
        edge/.style={-,black!30!black, thick},
        dedge/.style={-,black!30!black,densely dotted, thick},
        ]
        \node[root] (1) at (0,0) {\scriptsize$T_1$};
        \node[vert] (2) at (0,1.5) {\scriptsize$T_2$};
        \node[vert] (3) at (0,3) {\scriptsize$T_3$};
        \draw[edge] (1)--(2)--(3);
    \end{tikzpicture}}}=\sum \vcenter{\hbox{\begin{tikzpicture}[
        scale=0.5,
        root/.style={inner sep=2pt, circle,draw, dashed, thick},
        vert/.style={inner sep=2pt, circle,draw, thick},
        leaf/.style={inner sep=2pt, rectangle, thick},
        edge/.style={-,black!30!black, thick},
        dedge/.style={-,black!30!black,densely dotted, thick},
        ]
        \node[root] (1) at (0,0) {\scriptsize$T_1$};
        \node[vert] (2) at (-1.131,1.3) {\scriptsize$T_2$};
        \node[vert] (3) at (1.131,1.3) {\scriptsize$T_3$};
        \draw[edge] (1)--(2);
        \draw[edge] (1)--(3);
    \end{tikzpicture}}}
\] and the right part is invariant under the permutation $T_2\leftrightarrow T_3$.
\end{proof}

\begin{prop}\label{rst}
 The rooted spanning trees contractad $\RST^{\vee}$ is generated by a generator $\mu$ in the component $\Path_2$, satisfying the relations
 \begin{gather}
     \mu \circ^{\Path_3}_{\{1,2\}} \mu= \mu \circ^{\Path_3}_{\{2,3\}} \mu
     \\
     \mu \circ^{\Path_3}_{\{1,2\}} \mu^{(12)} = \mu^{(12)} \circ^{\Path_3}_{\{2,3\}} \mu
     \\
     \mu^{(12)} \circ^{\Path_3}_{\{1,2\}} \mu=0
     \\
     \mu \circ^{\K_3}_{\{1,2\}} \mu- \mu \circ^{\K_3}_{\{2,3\}} \mu= (\mu \circ^{\K_3}_{\{1,2\}} \mu- \mu \circ^{\K_3}_{\{2,3\}} \mu)^{(23)}
 \end{gather}
 \end{prop}
\begin{proof} The proof mimics that in \cite[Th.~1.9]{chapoton2001pre}.
Let $\mu:=\vcenter{\hbox{\begin{tikzpicture}[
        scale=0.5,
        root/.style={inner sep=2pt, circle,draw, dashed, thick},
        vert/.style={inner sep=2pt, circle,draw, thick},
        leaf/.style={inner sep=2pt, rectangle, thick},
        edge/.style={-,black!30!black, thick},
        ]
        \node[root] (1) at (0,0) {\scriptsize$1$};
        \node[vert] (2) at (1.5,0) {\scriptsize$2$};
        \draw[edge] (1)--(2);
    \end{tikzpicture}}}$ be the spanning tree of $\Path_2$ with root $1$. Let us examine the relations above. For the path $\Path_3$, we have
\begin{gather*}
    \mu \circ^{\Path_3}_{\{1,2\}} \mu= \vcenter{\hbox{\begin{tikzpicture}[
        scale=0.5,
        root/.style={inner sep=2pt, circle,draw, dashed, thick},
        vert/.style={inner sep=2pt, circle,draw, thick},
        edge/.style={-,black!30!black, thick},
        ]
        \node[root] (1) at (0,0) {\scriptsize$1$};
        \node[vert] (2) at (1.5,0) {\scriptsize$2$};
        \draw[edge] (1)--(2);
    \end{tikzpicture}}}\star \vcenter{\hbox{\begin{tikzpicture}[
        scale=0.5,
        root/.style={inner sep=2pt, circle,draw, dashed, thick},
        vert/.style={inner sep=2pt, circle,draw, thick},
        edge/.style={-,black!30!black, thick},
        ]
        \node[root] (1) at (0,0) {\scriptsize$3$};
    \end{tikzpicture}}}= \vcenter{\hbox{\begin{tikzpicture}[
        scale=0.5,
        root/.style={inner sep=2pt, circle,draw, dashed, thick},
        vert/.style={inner sep=2pt, circle,draw, thick},
        leaf/.style={inner sep=2pt, rectangle, thick},
        edge/.style={-,black!30!black, thick},
        ]
        \node[root] (1) at (0,0) {\scriptsize$1$};
        \node[vert] (2) at (1.5,0) {\scriptsize$2$};
        \node[vert] (3) at (3,0) {\scriptsize$3$};
        \draw[edge] (1)--(2)--(3);
    \end{tikzpicture}}}=\vcenter{\hbox{\begin{tikzpicture}[
        scale=0.5,
        root/.style={inner sep=2pt, circle,draw, dashed, thick},
        vert/.style={inner sep=2pt, circle,draw, thick},
        leaf/.style={inner sep=2pt, rectangle, thick},
        edge/.style={-,black!30!black, thick},
        ]
        \node[root] (1) at (0,0) {\scriptsize$1$};
    \end{tikzpicture}}}\star \vcenter{\hbox{\begin{tikzpicture}[
        scale=0.5,
        root/.style={inner sep=2pt, circle,draw, dashed, thick},
        vert/.style={inner sep=2pt, circle,draw, thick},
        leaf/.style={inner sep=2pt, rectangle, thick},
        edge/.style={-,black!30!black, thick},
        ]
        \node[root] (1) at (0,0) {\scriptsize$2$};
        \node[vert] (2) at (1.5,0) {\scriptsize$3$};
        \draw[edge] (1)--(2);
    \end{tikzpicture}}}=\mu \circ^{\Path_3}_{\{2,3\}} \mu
    \\
    \mu \circ^{\Path_3}_{\{1,2\}} \mu^{(12)}= \vcenter{\hbox{\begin{tikzpicture}[
        scale=0.5,
        root/.style={inner sep=2pt, circle,draw, dashed, thick},
        vert/.style={inner sep=2pt, circle,draw, thick},
        leaf/.style={inner sep=2pt, rectangle, thick},
        edge/.style={-,black!30!black, thick},
        ]
        \node[vert] (1) at (0,0) {\scriptsize$1$};
        \node[root] (2) at (1.5,0) {\scriptsize$2$};
        \draw[edge] (1)--(2);
    \end{tikzpicture}}}\star \vcenter{\hbox{\begin{tikzpicture}[
        scale=0.5,
        root/.style={inner sep=2pt, circle,draw, dashed, thick},
        vert/.style={inner sep=2pt, circle,draw, thick},
        leaf/.style={inner sep=2pt, rectangle, thick},
        edge/.style={-,black!30!black, thick},
        ]
        \node[root] (1) at (0,0) {\scriptsize$3$};
    \end{tikzpicture}}}= \vcenter{\hbox{\begin{tikzpicture}[
        scale=0.5,
        root/.style={inner sep=2pt, circle,draw, dashed, thick},
        vert/.style={inner sep=2pt, circle,draw, thick},
        leaf/.style={inner sep=2pt, rectangle, thick},
        edge/.style={-,black!30!black, thick},
        ]
        \node[vert] (1) at (0,0) {\scriptsize$1$};
        \node[root] (2) at (1.5,0) {\scriptsize$2$};
        \node[vert] (3) at (3,0) {\scriptsize$3$};
        \draw[edge] (1)--(2)--(3);
    \end{tikzpicture}}}=\vcenter{\hbox{\begin{tikzpicture}[
        scale=0.5,
        root/.style={inner sep=2pt, circle,draw, dashed, thick},
        vert/.style={inner sep=2pt, circle,draw, thick},
        leaf/.style={inner sep=2pt, rectangle, thick},
        edge/.style={-,black!30!black, thick},
        ]
        \node[root] (1) at (0,0) {\scriptsize$2$};
        \node[vert] (2) at (1.5,0) {\scriptsize$3$};
        \draw[edge] (1)--(2);
    \end{tikzpicture}}}\star \vcenter{\hbox{\begin{tikzpicture}[
        scale=0.5,
        root/.style={inner sep=2pt, circle,draw, dashed, thick},
        vert/.style={inner sep=2pt, circle,draw, thick},
        leaf/.style={inner sep=2pt, rectangle, thick},
        edge/.style={-,black!30!black, thick},
        ]
        \node[root] (1) at (0,0) {\scriptsize$1$};
    \end{tikzpicture}}}=\mu^{(12)} \circ^{\Path_3}_{\{2,3\}} \mu
    \\
    \mu^{(12)} \circ^{\Path_3}_{\{1,2\}} \mu=\vcenter{\hbox{\begin{tikzpicture}[
        scale=0.5,
        root/.style={inner sep=2pt, circle,draw, dashed, thick},
        vert/.style={inner sep=2pt, circle,draw, thick},
        leaf/.style={inner sep=2pt, rectangle, thick},
        edge/.style={-,black!30!black, thick},
        ]
        \node[root] (1) at (0,0) {\scriptsize$3$};
    \end{tikzpicture}}}\star\vcenter{\hbox{\begin{tikzpicture}[
        scale=0.5,
        root/.style={inner sep=2pt, circle,draw, dashed, thick},
        vert/.style={inner sep=2pt, circle,draw, thick},
        leaf/.style={inner sep=2pt, rectangle, thick},
        edge/.style={-,black!30!black, thick},
        ]
        \node[root] (1) at (0,0) {\scriptsize$1$};
        \node[vert] (2) at (1.5,0) {\scriptsize$2$};
        \draw[edge] (1)--(2);
    \end{tikzpicture}}}=0, \text{ since vertices }\vcenter{\hbox{\begin{tikzpicture}[
        scale=0.5,
        vert/.style={inner sep=2pt, circle,draw, thick},
        ]
        \node[vert] (1) at (0,0) {\scriptsize$1$};
    \end{tikzpicture}}}, \vcenter{\hbox{\begin{tikzpicture}[
        scale=0.5,
        vert/.style={inner sep=2pt, circle,draw, thick},
        ]
        \node[vert] (1) at (0,0) {\scriptsize$3$};.
    \end{tikzpicture}}} \text{ are not adjacent}.
\end{gather*}
For the graph $\K_3$, we have
\[
\mu \circ^{\K_3}_{\{1,2\}} \mu- \mu \circ^{\K_3}_{\{2,3\}} \mu=(\vcenter{\hbox{\begin{tikzpicture}[
        scale=0.5,
        root/.style={inner sep=2pt, circle,draw, dashed, thick},
        vert/.style={inner sep=2pt, circle,draw, thick},
        leaf/.style={inner sep=2pt, rectangle, thick},
        edge/.style={-,black!30!black, thick},
        ]
        \node[root] (1) at (0,0) {\scriptsize$1$};
        \node[vert] (2) at (1.5,0) {\scriptsize$2$};
        \draw[edge] (1)--(2);
    \end{tikzpicture}}}\star \vcenter{\hbox{\begin{tikzpicture}[
        scale=0.5,
        root/.style={inner sep=2pt, circle,draw, dashed, thick},
        vert/.style={inner sep=2pt, circle,draw, thick},
        leaf/.style={inner sep=2pt, rectangle, thick},
        edge/.style={-,black!30!black, thick},
        ]
        \node[root] (1) at (0,0) {\scriptsize$3$};
    \end{tikzpicture}}})-(\vcenter{\hbox{\begin{tikzpicture}[
        scale=0.5,
        root/.style={inner sep=2pt, circle,draw, dashed, thick},
        vert/.style={inner sep=2pt, circle,draw, thick},
        leaf/.style={inner sep=2pt, rectangle, thick},
        edge/.style={-,black!30!black, thick},
        ]
        \node[root] (1) at (0,0) {\scriptsize$1$};
    \end{tikzpicture}}}\star \vcenter{\hbox{\begin{tikzpicture}[
        scale=0.5,
        root/.style={inner sep=2pt, circle,draw, dashed, thick},
        vert/.style={inner sep=2pt, circle,draw, thick},
        leaf/.style={inner sep=2pt, rectangle, thick},
        edge/.style={-,black!30!black, thick},
        ]
        \node[root] (1) at (0,0) {\scriptsize$2$};
        \node[vert] (2) at (1.5,0) {\scriptsize$3$};
        \draw[edge] (1)--(2);
    \end{tikzpicture}}})=(\vcenter{\hbox{\begin{tikzpicture}[
        scale=0.5,
        root/.style={inner sep=2pt, circle,draw, dashed, thick},
        vert/.style={inner sep=2pt, circle,draw, thick},
        leaf/.style={inner sep=2pt, rectangle, thick},
        edge/.style={-,black!30!black, thick},
        dedge/.style={-,black!30!black,densely dotted, thick},
        ]
        \node[root] (1) at (0,0) {\scriptsize$1$};
        \node[vert] (2) at (-0.87,1) {\scriptsize$2$};
        \node[vert] (3) at (0.87,1) {\scriptsize$3$};
        \draw[edge] (1)--(2);
        \draw[edge] (1)--(3);
        \draw[dedge] (2)--(3);
    \end{tikzpicture}}}+\vcenter{\hbox{\begin{tikzpicture}[
        scale=0.5,
        root/.style={inner sep=2pt, circle,draw, dashed, thick},
        vert/.style={inner sep=2pt, circle,draw, thick},
        leaf/.style={inner sep=2pt, rectangle, thick},
        edge/.style={-,black!30!black, thick},
        dedge/.style={-,black!30!black,densely dotted, thick},
        ]
        \node[root] (1) at (0,0) {\scriptsize$1$};
        \node[vert] (2) at (-0.87,1) {\scriptsize$2$};
        \node[vert] (3) at (0.87,1) {\scriptsize$3$};
        \draw[edge] (1)--(2)--(3);
        \draw[dedge] (1)--(3);
    \end{tikzpicture}}})-\vcenter{\hbox{\begin{tikzpicture}[
        scale=0.5,
        root/.style={inner sep=2pt, circle,draw, dashed, thick},
        vert/.style={inner sep=2pt, circle,draw, thick},
        leaf/.style={inner sep=2pt, rectangle, thick},
        edge/.style={-,black!30!black, thick},
        dedge/.style={-,black!30!black,densely dotted, thick},
        ]
        \node[root] (1) at (0,0) {\scriptsize$1$};
        \node[vert] (2) at (-0.87,1) {\scriptsize$2$};
        \node[vert] (3) at (0.87,1) {\scriptsize$3$};
        \draw[edge] (1)--(2)--(3);
        \draw[dedge] (1)--(3);
    \end{tikzpicture}}}=\vcenter{\hbox{\begin{tikzpicture}[
        scale=0.5,
        root/.style={inner sep=2pt, circle,draw, dashed, thick},
        vert/.style={inner sep=2pt, circle,draw, thick},
        leaf/.style={inner sep=2pt, rectangle, thick},
        edge/.style={-,black!30!black, thick},
        dedge/.style={-,black!30!black,densely dotted, thick},
        ]
        \node[root] (1) at (0,0) {\scriptsize$1$};
        \node[vert] (2) at (-0.87,1) {\scriptsize$2$};
        \node[vert] (3) at (0.87,1) {\scriptsize$3$};
        \draw[edge] (1)--(2);
        \draw[edge] (1)--(3);
        \draw[dedge] (2)--(3);
    \end{tikzpicture}}},
\] and the resulting rooted spanning tree is invariant under the transposition $(23)$. Hence, we have the well-defined morphism of contractads $\pi\colon\Pop\rightarrow \RST^{\vee}$, where $\Pop$ is the quadratic contractad generated by the generators and relations above.

For a graph $\Gr$, each rooted tree can be written in the following way. For a vertex $v$ and collection of non-intersecting rooted subtrees $\{T_1,\cdots, T_k\}$ whose roots are adjacent to $v$, we define the rooted spanning tree $B(v,T_1,T_2,\cdots, T_k)$ of $\Gr$ by the rule
\[
B(v,T_1,T_2,\cdots, T_k)=\vcenter{\hbox{\begin{tikzpicture}[
        scale=0.7,
        root/.style={inner sep=2pt, circle,draw, dashed, thick},
        vert/.style={inner sep=2pt, circle,draw, thick},
        leaf/.style={inner sep=2pt, rectangle, thick},
        edge/.style={-,black!30!black, thick},
        dedge/.style={-,black!30!black,densely dotted, thick},
        ]
        \node[root] (1) at (0,0) {\scriptsize$v$};
        \node[vert] (T1) at (-2,1.1){\scriptsize$T_1$};
        \node[vert] (T2) at (-1,1.1) {\scriptsize$T_2$};
        \node[leaf] (d) at (0,1.1) {\scriptsize$\cdots$};
        \node[leaf] (T3) at (1,1.1){\scriptsize$\cdots$};
        \node[vert] (T4) at (2,1.1) {\scriptsize$T_k$};
        \draw[edge] (1)--(T1);
        \draw[edge] (1)--(T2);
        \draw[edge] (1)--(d);
        \draw[edge] (1)--(T3);
        \draw[edge] (1)--(T4);
    \end{tikzpicture}}}.
\] Note that the definition does not depend on the order of subtrees $T_i$. For $k=1$, we have $B(v,T_1)=\vcenter{\hbox{\begin{tikzpicture}[
        scale=0.5,
        root/.style={inner sep=2pt, circle,draw, dashed, thick},
        ]
        \node[root] (1) at (0,0) {\scriptsize$v$};
    \end{tikzpicture}}}\star T_1$, and for larger $k$ we have the recurrence relation
\[
B(v,T_1,\cdots,T_k)=B(v,T_2,\cdots,T_k)\star T_1-\sum^k_{i=2} B(v,T_1,\cdots,T_i\star T_1,\cdots, T_k).
\] Since each rooted tree is obtained by means of the $\star$-product of smaller trees, we conclude that $\mu$ generates $\RST^{\vee}$. Hence the morphism $\pi: \Pop\rightarrow\RST^{\vee}$ is onto. 

To complete the proof, it suffices to construct an inverse map $\iota: \RST^{\vee}\rightarrow \Pop$. We define the section $\iota$ inductively by the rule
\[
\iota(B(v,T_1,\cdots,T_k))=\iota(B(v,T_2,\cdots,T_k))\star_{\Pop} \iota(T_1)-\sum^k_{i=2} \iota(B(v,T_1,\cdots,T_i\star T_1,\cdots, T_k)),
\] where $\star_{\Pop}$ stands for the star product in $\Pop$, defined in a similar way.  It remains to verify that this correspondence does not depend on the order of $T_i$. This part is proved analogously to \cite[Th.~1.9]{chapoton2001pre} using Lemma~\ref{lemma:prelieidentity}. By construction, we see that the map $\iota$ is a morphism of contractads. Since compositions $\iota\circ \pi$ and $\pi\circ\iota$ are the identity maps on the generator  $\mu$, we conclude that  $\pi$ and $\iota$ are inverse to each other.
\end{proof}
\subsection{Little disks contractads}\label{subsec:disks}

In this subsection, we give topological examples of contractads. More explicitly, we define the little $n$-disks contractads $\D_n$. These contractads  are graphical counterparts of the little disks operads, first described in~\cite{cohen1976homology}.

Let $\Disc \subset \mathbb{R}^n$ be the $n$-dimensional unit disk. For a graph $\Gr$, a \textit{graphical disk configuration} is a continuous map $i: \coprod_{v \in V_{\Gr}} \Disc_{v} \rightarrow \Disc$ from the disjoint union of labeled unit disks to the unit disk, such that the restriction to each connected component $i|_{\Disc_v}$ is the composition of a dilation and a translation, and for each pair of adjacent vertices $v,w \in V_{\Gr}$, the related interiors of disks $\mathring{\Disc}_v$ and $\mathring{\Disc}_w$ don't intersect in the image. Let us denote by $\D_n(\Gr)$ the space of graphical disk configurations. This space is topologized as a subspace of $((0,1)\times \mathring{\Disc})^{V_{\Gr}}$
\[
\D_n(\Gr):=\{(r_v,x_v)_{v \in V_{\Gr}}| (v,w)\in E_{\Gr} \Rightarrow (r_v\mathring{\Disc}+x_v)\cap (r_w\mathring{\Disc}+x_w)=\varnothing\} \subset ((0,1)\times \mathring{\Disc})^{V_{\Gr}}.
\]

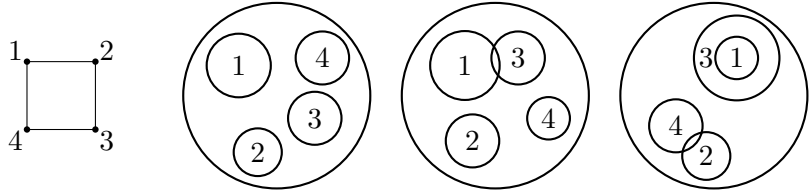
\begin{figure}[ht]
    \centering
    \[
    \vcenter{\hbox{\begin{tikzpicture}[scale=0.6]
    \fill (0,0) circle (2pt);
    \fill (0,1.5) circle (2pt);
    \fill (1.5,0) circle (2pt);
    \fill (1.5,1.5) circle (2pt);
    \draw (0,0)--(1.5,0)--(1.5,1.5)--(0,1.5)-- cycle;
    \node at (-0.25,1.75) {$1$};
    \node at (1.75,1.75) {$2$};
    \node at (1.75,-0.25) {$3$};
    \node at (-0.25,-0.25) {$4$};
    \end{tikzpicture}}}
    \quad
    \quad
    \vcenter{\hbox{\begin{tikzpicture}
     \draw[thick] (0,0) circle [radius=35pt];
     \draw[thick] (-0.5,0.4) circle [radius=12pt];
     \node at (-0.5,0.4) {1};
     \draw[thick] (-0.25,-0.75) circle [radius=9pt];
     \node at (-0.25,-0.75) {2};
     \draw[thick] (0.5,-0.3) circle [radius=10pt];
     \node at (0.5,-0.3) {3};
     \draw[thick] (0.6,0.5) circle [radius=10pt];
     \node at (0.6,0.5) {4};
    \end{tikzpicture}}}
    \quad
    \vcenter{\hbox{\begin{tikzpicture}
     \draw[thick] (0,0) circle [radius=35pt];
     \draw[thick] (-0.4,0.4) circle [radius=13pt];
     \node at (-0.4,0.4) {1};
     \draw[thick] (-0.3,-0.6) circle [radius=10pt];
     \node at (-0.3,-0.6) {2};
     \draw[thick] (0.3,0.5) circle [radius=10pt];
     \node at (0.3,0.5) {3};
     \draw[thick] (0.7,-0.3) circle [radius=8pt];
     \node at (0.7,-0.3) {4};    
    \end{tikzpicture}}}
    \quad
    \vcenter{\hbox{\begin{tikzpicture}
     \draw[thick] (0,0) circle [radius=35pt];
     \draw[thick] (0.3,0.5) circle [radius=16pt];
     \node at (-0.1,0.5) {3};
     \draw[thick] (-0.1,-0.8) circle [radius=9pt];
     \node at (-0.1,-0.8) {2};
     \draw[thick] (-0.5,-0.4) circle [radius=10pt];
     \node at (-0.5,-0.4) {4};
     \draw[thick] (0.3,0.5) circle [radius=8pt];
     \node at (0.3,0.5) {1};
    \end{tikzpicture}}}\]
    \caption{Example of configurations in $\D_2(\Cyc_4)$.}
\end{figure}
The resulting graphical collection of disk configurations $\D_n$ is endowed with a contractad structure as follows. For a graph $\Gr$ and a tube $G$, we define the infinitesimal composition $\circ_G^{\Gr}\colon \D_n(\Gr/G)\times \D_n(\Gr|_G) \to \D_n(\Gr)$ by compositions of disk configurations as in Figure~\ref{fig:disccomp}. It is easy to check that this datum of maps endows $\D_n$ with a structure of contractad. The restrictions of these contractads to complete graphs $\K_{*}(\D_n)$ are the classical little $n$-disks operads.
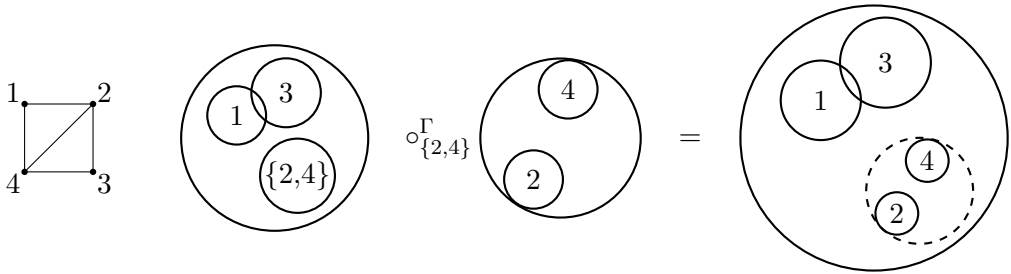
\begin{figure}[ht]
    \centering
    \[
      \vcenter{\hbox{\begin{tikzpicture}[scale=0.6]
    \fill (0,0) circle (2pt);
    \fill (0,1.5) circle (2pt);
    \fill (1.5,0) circle (2pt);
    \fill (1.5,1.5) circle (2pt);
    \draw (0,0)--(1.5,0)--(1.5,1.5)--(0,1.5)-- cycle;
    \draw (0,0)--(1.5,1.5);
    \node at (-0.25,1.75) {$1$};
    \node at (1.75,1.75) {$2$};
    \node at (1.75,-0.25) {$3$};
    \node at (-0.25,-0.25) {$4$};
    \end{tikzpicture}}}
    \quad
    \quad
     \vcenter{\hbox{\begin{tikzpicture}
     \draw[thick] (0,0) circle [radius=35pt];
     \draw[thick] (-0.5,0.3) circle [radius=11pt];
     \node at (-0.5,0.3) {1};
     \draw[thick] (0.3,-0.5) circle [radius=14pt];
     \node at (0.3,-0.5) {\{2,4\}};
     \draw[thick] (0.15,0.6) circle [radius=13pt];
     \node at (0.15,0.6) {3};
    \end{tikzpicture}}}
    \quad
    \circ^{\Gr}_{\{2,4\}}
     \vcenter{\hbox{\begin{tikzpicture}
     \draw[thick] (0,0) circle [radius=30pt];
     \draw[thick] (-0.355,-0.55) circle [radius=11pt];
     \node at (-0.355,-0.55) {2};
     \draw[thick] (0.1,0.645) circle [radius=11pt];
     \node at (0.1,0.645) {4};    
    \end{tikzpicture}}}
    \quad
    =
    \quad
    \vcenter{\hbox{\begin{tikzpicture}
     \draw[thick] (0,0) circle [radius=50pt];
     \draw[thick] (-0.7,0.5) circle [radius=15pt];
     \node at (-0.7,0.5) {1};
     \draw[dashed, thick] (0.6,-0.7) circle [radius=20pt];
     \draw[thick] (0.7,-0.3) circle [radius=8pt];
     \node at (0.7,-0.3) {4};
     \draw[thick] (0.3,-1) circle [radius=8pt];
     \node at (0.3,-1) {2};
     \draw[thick] (0.15,1) circle [radius=17pt];
     \node at (0.15,1) {3};
    \end{tikzpicture}}}
\]
    \caption{Disks composition in $\D_2$}
    \label{fig:disccomp}
\end{figure}

Let us explore the relationship between little disks contractads and configuration spaces. For a graph $\Gr$ and topological space $X$, define the \textit{graphical configuration space} $\Conf_{\Gr}(X)$ as the space of functions $f\colon V_{\Gr} \rightarrow X$ from the vertex set to the given space such that images of adjacent vertices do not coincide. This space is topologized as a subspace of $X^{V_{\Gr}}$
\[
\Conf_{\Gr}(X)=\{(x_v) \in X^{V_{\Gr}}| (v,w) \in E_{\Gr} \Rightarrow x_v \neq x_w\}.
\] In the case of complete graphs, we get the classical configuration spaces of $n$-distinct points $\Conf_{\K_n}(X):=\Conf_n(X)$.
\begin{prop}\label{discsconf}
For a graph $\Gr$, the continuous map
\begin{gather*}
    \pi\colon \D_n(\Gr) \to \Conf_{\Gr}(\mathbb{R}^n)
    \\
    (r_v\Disc+x_v)_{v \in V_{\Gr}} \mapsto (x_v)_{v \in V_{\Gr}}
\end{gather*} provides a homotopy equivalence.
\end{prop}
\begin{proof}
We can replace $\mathbb{R}^n$ with the homeomorphic open unit disk $\mathring{\Disc}$ since the related configuration spaces $\Conf_{\Gr}(\mathbb{R}^n)\cong \Conf_{\Gr}(\mathring{\Disc})$ are homeomorphic. Let us construct the section $s\colon \Conf_{\Gr}(\mathring{\Disc}) \to \D_n(\Gr)$ as follows. For each point in the configuration space $x=(x_v)$, we define the radius $r(x):=\frac{1}{3}\min\{|x_v-x_w| (v,w) \in E_{\Gr}\}$, which is non-zero by the construction of a graphical configuration space and depends continuously on $x$. The related map
\begin{gather*}
   s\colon \Conf_{\Gr}(\mathring{\Disc}) \rightarrow \D_n(\Gr)
   \\
   x \mapsto (r(x)\Disc+x_v)_{v \in V_{\Gr}}
\end{gather*} is a well-defined continuous section of $\pi$ and is the homotopy inverse. Indeed, we have the homotopy $s\circ\pi\sim\Id$ given by the rule
\[
H_t((r_v\Disc+x_v))= ((tr(x)+(1-t)r_v)\Disc+x_v).
\]
\end{proof}    

\noindent If we replace each component of the contractad $\D_n$ with its homology groups, we obtain a linear contractad $H_{\bullet}(\D_n)$. Let us list some properties of these contractads
\begin{itemize}
    \item For $n\geq 1$, the contractad $H_{\bullet}(\D_n)$ has a quadratic presentation (Proposition~\ref{thm:ass},\ref{thm:en}). 
    \item For $n=1$, the contractad $H_{\bullet}(\D_1)$ defines a graphical counterpart of the operad $\mathsf{Ass}$ of associative algebras (Theorem~\ref{thm:ass}). The dimension of each component $H_{\bullet}(\D_1(\Gr))$ is given by the formula $\dim H_{\bullet}(\D_1)=(-1)^{|V_{\Gr}|}\chi_{\Gr}(-1)$, where $\chi_{\Gr}(t)$ is the chromatic polynomial of the graph (Corrollary~\ref{sled:hilbertseries}). It is worth mentioning that the number  $(-1)^{|V_{\Gr}|}\chi_{\Gr}(-1)$ coincides with the number of acyclic edge orientation of the graph.
    \item For $n=2$, the contractad $H_{\bullet}(\D_2)$ defines a graphical counterpart of the operad $\mathsf{Gerst}$ of Gerstenhaber algebras (Theorem~\ref{thm:gerst}). The Hilbert series of each component of this contractad is given by the formula (Proposition~\ref{discring})
    \[
    H_{\D_2(\Gr)}(t)=t^{|V_{\Gr}|}\chi_{\Gr}(\frac{1}{t}).
    \] For particular types of graphs, we have
    \[
    H_{\D_2(\Path_n)}(t)=(1-t)^{n-1};\quad H_{\D_2(\Cyc_n)}(t)=(1-t)^{n}-(-t)^n-(-t)^{n-1}; \quad H_{\D_2(\K_n)}(t)=\prod^{n-1}_{k=1}(1-kt).
    \]
    \item For $n\geq 1$, the contractad $H_{\bullet}(\D_n)$ is Koszul and has a quadratic Gr\"obner basis\\ (Proposition~\ref{thm:en}).
\end{itemize}

\subsection{Wonderful contractad}\label{subsec:wonderful}
In this subsection, we give a geometric example of a contractad. More specifically, we define the contractad of graphical compactifications $\bM$. The restriction of this contractad to particular types of graphs gives us well-known examples of moduli spaces of stable curves of genus zero with certain conditions.\\

\noindent\textbf{Building sets and Wonderful models}  Let us briefly recall some facts about building sets and wonderful compactifications. For details, refer to \cite{de1995wonderful, rains2010homology, feichtner2004chow}.

Let $\A=\{H_1,H_2,...,H_n\}$ be an arrangement of complex linear hyperplanes in the complex vector space $V$. The \textit{intersection lattice} $\La(\A)$ is the set of all intersections of subsets of $\A$ partially ordered by reverse inclusion. It has maximal element $\hat{1}=\bigcap \A$ and minimal element $\hat{0}= V$.  The join(least upper bound) of a subset $S\subset \La(\A)$ is equal to the intersection $\vee S= \bigcap_{H \in S} H$. 

For an element $X \in \La(\A)$ and a subset $S$, we denote by $S|_X$ the intersection $S\cap [\hat{0},X]=\{y  \in S| y \leqslant X\}$. A \textit{building set} on an arrangement $\A$ is a subset $\G \subset \La(\A)\setminus \{\hat{0}\}$ such that, for each non-minimal element $X \in \La(\A)$ with $\max\G|_X=\{G_1,\cdots, G_k\}$, the join provides  a bijective map
\begin{gather*}
    \prod_{i=1}^k [\hat{0},G_i] \overset{\cong}{\longrightarrow} [\hat{0},X], \quad (X_1,\cdots, X_k) \mapsto \bigvee^k_{i=1} X_i.
\end{gather*}
A building set $\G$ is called \textit{connected} if it contains a maximal element of an intersection lattice. For a building set $\G$ on a hyperplane arrangement $\A$, consider the map from the projective complement of  the hyperplane arrangement $\mathcal{M}(\mathcal{A})=\Pro(V)\setminus \bigcup_{H \in \alpha} \Pro(H)$ to the product of projective spaces
\begin{equation*}
    \rho\colon \mathcal{M}(\mathcal{A}) \rightarrow \prod_{G \in \G} \Pro(V/G),
\end{equation*} where $\rho$ is the product of projections $\rho_G\colon \Pro(V)\setminus \bigcup_i \Pro(H_i) \rightarrow \Pro(V/G)$ induced by quotients $V \rightarrow V/G$. The \textit{projective wonderful compactification} $\Y_{\G}$ associated with building set $\G$ is the closure of the image of $\M(\A)$ under the map above. It is known that $\Y_{\G}$ is a smooth irreducible projective variety, and if $\G$ is connected, then the projection $\Y_{\G} \rightarrow \Pro(V)$ is an iterated sequence of blow-downs encoded by elements of the building set \cite{de1995wonderful}.

For an element $X \in \La(\A)$, the restriction $\A|_X=\{H \in \A| X \subset H\}$ forms a hyperplane arrangement in the quotient space $V/X$. Note that the related intersection lattice $\La(\A|_X)$ forms an  initial interval $[\hat{0},X]$ in the origin $\La(\A)$, and the restriction $\G|_X$ forms a building set on this arrangement. In a dual way, the intersections of hyperplanes $H_i \cap X$ that differ from $X$ defines the contracted arrangement $\A/X$ in the vector space $X$, and the corresponding intersection lattice forms a terminal interval $\La(\A/X)\cong [X,\hat{1}]\subset \La(\A)$.  It is easy to see that the collection  $\G/X:=\{G\cap X| G \in \G\cap (X, \hat{1}]\}$ forms a building set on the arrangement $\A/X$.

According to Rains \cite[Th.~2.5]{rains2010homology}, for any non-maximal element of the building set $G \in \G\setminus\{\hat{1}\}$, there is a closed embedding of wonderful compactifications
\begin{equation}
\label{rainsmap}
     \circ^{\G}_{H}\colon \Y_{\G/G}\times \Y_{\G|_G} \longrightarrow \Y_{\G}.
\end{equation} To define this morphism it suffices to specify $\rho_H\circ^{\Gr}_G$ for each $H\in \G$, where $\rho_H\colon \Y_{\G}\rightarrow \Pro(V/H)$. For $G\subset H$, we set $\rho_H\circ^{\Gr}_G=\rho_{H/G}$, projecting from $\Y_{\G|_G}$. Otherwise, we compose $\rho_{H\cap G}$ (projecting from $\Y_{\G/G}$) with the embedding $\Pro(G/H\cap G)\rightarrow \Pro(V/H)$. The image of this product is an irreducible divisor which we denote by $D_G$. It is a known fact that the complement $\Y_{\G}\setminus \M(\A)$ is a divisor with normal crossings whose irreducible components are divisors $D_G$ indexed by elements of $\G\setminus \{\hat{1}\}$~\cite{de1995wonderful}.

Intersections of these divisors are uniquely determined by the combinatorics of the building set. A non-empty subset of a building set $\Susp \subset \G$ is called a \textit{nested set} if for each subset $\{G_1,G_2,\cdots, G_k\} \subset \Susp$ of pairwise incomparable elements, we have $\bigvee^k_{i=1}G_i \not\in \G$. For a subset $\Susp \subset  \G\setminus\{\hat{1}\}$ the intersection $D_{\Susp}:=\bigcap_{G \in \Susp} D_G$ is non-empty iff $\Susp$ is nested. The collection of all nested sets forms a simplicial set $\hat{\N}(\La(\A), \G)$ on the vertex set $\G$. This complex is homeomorphic to a cone with apex $\{\hat{1}\}$
\[
\hat{\N}(\La(\A), \G) \cong \{\hat{1}\} * \N(\La(\A),\G),
\] whose base $\N(\La(\A),\G)$ is called a \textit{nested set complex} and consists  of all  nested sets which do not contain the maximal element. For our purposes, we consider the \textit{augmented nested set complex} obtained by adding to $\N(\La(\A),\G)$ one  $(-1)$-simplex $\varnothing$.  While the nested set complex encodes intersections of divisors, the augmented version encodes their interiors
\[
\mathring{D}_{\Susp}:=D_{\Susp}\setminus \bigcup_{\Susp \subsetneq \Susp'} D_{\Susp'}
\] with the notation $\mathring{D}_{\varnothing}=\M(\A)$, which form a locally-open stratification of $\Y_{\G}$.\\

\noindent\textbf{Wonderful contractad.} For a connected graph $\Gr$ with at least one edge, consider the \textit{graphic arrangement} $\B(\Gr)=\{\{x_v=x_w\}| (v,w)\in E_{\Gr}\}$ in $\mathbb{C}^{V_{\Gr}}$. For an edge $e=(v,w)$, we denote the related hyperplane $\{x_v=x_w\}$ by $H_e$ the . Note that the complete intersection of hyperplanes $\cap_{e\in E_{\Gr}} H_e=\langle \sum_{v\in V_{\Gr}} x_v\rangle$ is a one-dimensional space, hence, we replace the vector space $\mathbb{C}^{V_{\Gr}}$ with its quotient $V=\mathbb{C}^{V_{\Gr}}/\langle \sum_{v\in V_{\Gr}} x_v\rangle$ to make the arrangement $\B(\Gr)$ essential.
\begin{prop}\label{grapharrang}
For a graph $\Gr$, the intersection lattice of the graphic arrangement is isomorphic to the partition poset of $\Gr$
    \begin{equation*}
    \parti(\Gr) \cong \La(\mathcal{B}(\Gr)),
\end{equation*}
\end{prop}
\begin{proof}
For each non-trivial tube $G$, we associate the subspace $H_G:=\bigcap_{v,w \in G} H_{v,w}$. In the case when tube $G=\{v\}$ is trivial, we set $H_G=V$. The correspondence $I \mapsto \bigcap_{G \in I} H_G$ provides a well-defined map from the partition poset $\phi: \parti(\Gr) \rightarrow \La(\B(\Gr))$ to the related intersection lattice. It is easy to see that this correspondence is order-preserving and bijective.
\end{proof}

For a connected graph $\Gr$, the collection of subspaces $\G(\Gr):=\{ H_G \}_{G \subset V_{\Gr}}$ encoded by non-trivial tubes defines a connected building set on the graphic arrangement  $\B(\Gr)$. Indeed, via isomorphism  $\parti(\Gr) \cong \La(\B(\Gr))$, for each non-minimal partition $I$, we have an isomorphism
\[
\prod [\hat{0},G]\overset{\cong}{\rightarrow} [\hat{0},I],
\] where the product is taken over all non-trivial blocks of partition $I$. 

For simplicity of notation, we shall denote the projective complement $\Pro(V)\setminus \bigcup_{(i,j) \in E_{\Gr}} \Pro(H_e)$ by $\M(\Gr)$ and the corresponding projective wonderful compactification associated with $\G(\Gr)$ by $\bM(\Gr)$. For the one-vertex graph, we set $\bM(\Path_1)=\{pt\}$.

Note that the action of the automorphism group $\Aut(\Gr)$ on the lattice $\La(\B(\Gr))$ stabilizes the building set $\G(\Gr)$. Therefore, the action of $\Aut(\Gr)$ on $\M(\Gr)$ lifts to the action on its compactification $\bM(\Gr)$. The resulting graphical collection of wonderful compactifications $\bM$ is endowed with a contractad structure as follows.

\begin{prop}
The Rains maps \eqref{rainsmap} define a contractad structure on the graphical collection $\bM$ of graphical wonderful compactifications.
\end{prop}
\begin{proof}
Consider some graph $\Gr$ and the corresponding building set $\G(\Gr)$. By Proposition~\ref{intervals}, restrictions/contractions of graphs and lattices commute: 
\[\La(\B(\Gr)|_{H_G})\cong \La(\B(\Gr|_G)), \quad \La(\B(\Gr)/H_G)\cong \La(\B(\Gr/G)).\] This assertion remains true on the level of building sets: $\G(\Gr)|_{H_G} \cong \G(\Gr|_G)$, $\G(\Gr)/H_G \cong \G(\Gr/G)$. As a result, for each non-trivial tube $G$, Rains' map \eqref{rainsmap} has the form
\[
\bM(\Gr/G)\times \bM(\Gr|_G) \rightarrow \bM(\Gr).
\] The fact that this datum of morphisms satisfies contractad axioms follows from the properties of Rain's maps \eqref{rainsmap} described in \cite[Th.~2.5]{rains2010homology} or can be checked by hand.
\end{proof}

Let us consider some examples of  graphical compactifications for particular types of graphs.
\begin{itemize}
\item In the case of complete graphs, the graphic arrangements coincide with the braid arrangements $\B(\K_n):=\B_n$. The graphical building set $\G(\K_n)$ is a minimal building set on  this arrangement. It is a known fact that the resulting wonderful compactification $\bM(\K_n)$ coincides with the moduli space $\bM_{0,n+1}$ of   stable genus zero curves with $(n+1)$-marked points (with one marked  point singled out $\infty$ as a  special). The infinitesimal maps on the restriction $\K_{*}(\bM)=\{\bM_{0,n+1}\}$ are gluing maps
\[
\circ_i\colon \bM_{0,n+1} \times \bM_{0,m+1} \rightarrow \bM_{0,n+m},
\] obtained by gluing the special point $\infty$  of the curve  from the right factor to the $i$-th point of the curve from the left one. Moreover, the stratification of $\bM_{0,n+1}=\coprod \mathring{D}_{\Susp}$ encoded by elements of the augmented complex $\overline{\N}(\parti(\K_n),\G(\K_n))\cup \{\varnothing\}$ is equivalent to the canonical stratification $\bM_{0,n+1}=\coprod \M((T))$ by dual graphs, which are precisely $\K_n$-admissible rooted trees.

\item In the case of stellar graphs, the graphical building set $\G(\St_n)$ is a maximal building set on the arrangement $\B(\St_n)=\{\{x_0=x_i\}| i \in [n]\}$. It is a known fact that the resulting wonderful compactification $\bM(\St_n)$ coincides with the Losev-Manin moduli space $\overline{L}_n$ \cite{losev2000new}. This moduli space parametrizes chains of projective lines with two poles(labeled by $\infty$ and $0$) and $n$ marked points that may coincide. The associated structure of a twisted algebra on the restriction $\St_{*}(\bM)=\{\overline{L}_n\}$ is given by gluing of poles $\overline{L}_n \times \overline{L}_m \rightarrow \overline{L}_{n+m}$. Similarly to the previous example, the stratification coming from nested sets is equivalent to the stratification by dual graphs, which are $\St_n$-admissible rooted trees.
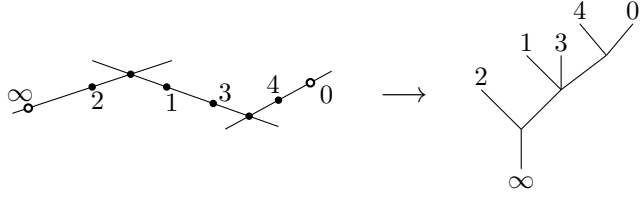
\begin{figure}[ht]
\caption{Dual graphs of stable chains are $\St$-admissible trees.}
    \[
    \vcenter{\hbox{\begin{tikzpicture}[scale=0.7]
    \draw[name path=1] (0,0)--(3,1);
    \draw[name path=2] (1.5,1)--(5,-0.25);
    \fill[name intersections={of=1 and 2}] (intersection-1) circle (2pt);
    \draw[name path=3] (4,-0.3)--(6,0.8);
    \fill[name intersections={of=2 and 3}] (intersection-1) circle (2pt);
    \draw[thick, fill=white] ($(0,0)!.1!(3,1)$) circle [radius=2pt];
    \node at (0.15,0.35) {\small$\infty$};
    \fill ($(0,0)!.5!(3,1)$) circle (2pt);
    \node at (1.6,0.25) {\small$2$};
    \fill ($(1.5,1)!.4!(5,-0.25)$) circle (2pt);
    \node at (3,0.2) {\small$1$};
    \fill ($(1.5,1)!.65!(5,-0.25)$) circle (2pt);
    \node at (4,0.4) {\small$3$};
    \fill ($(4,-0.3)!.5!(6,0.8)$) circle (2pt);
    \node at (4.9,0.55) {\small$4$};
    \draw[thick, fill=white] ($(4,-0.3)!.8!(6,0.8)$) circle [radius=2pt];
    \node at (5.9,0.35) {\small$0$};
    
    \end{tikzpicture}}}
    \quad
    \longrightarrow
    \quad
    \vcenter{\hbox{\begin{tikzpicture}[scale=0.7]
        \draw (0,0)--(0,0.75);
        \draw (0,0.75)--(-0.75,1.5);
        \draw (0,0.75)--(0.75,1.5);
        \draw (0.75,1.5)--(0.1,2.15);
        \draw (0.75,1.5)--(0.75,2.15);
        \draw (0.75,1.5)--(1.6,2.15);
        \draw (1.6,2.15)--(1.1,2.75);
        \draw (1.6,2.15)--(2.1,2.75);
        \node at (0,-0.25) {\small$\infty$};
        \node at (-0.75,1.75) {\small$2$};
        \node at (0.1,2.4) {\small$1$};
        \node at (0.75,2.4) {\small$3$};
        \node at (1.1,3) {\small$4$};
        \node at (2.1,3) {\small$0$};
  \end{tikzpicture}}}
\]
\end{figure}

\item Both of these examples are special cases of the Losev-Manin modular operad $\La\M_{0,S}$~\cite{losev2004extended}. These moduli spaces parametrize stable curves with marked points indexed by two-colored set $S$ (say “black” and “white”), where the points of type black are allowed to coincide and the points of type white are not. Let us define the graph $\K_{(1^m,n)}:=\K_m*\overline{\K}_n$ as the graph join of the complete graph on $m$ vertices with the graph on $n$ vertices without edges. It was shown in \cite[Th.~4.4]{coron2023supersolvability}, that the wonderful compactification $\bM(\K_{(1^n,m)})$ coincides with the component of the extended Losev-Manin operad encoding stable curves with $n$ black and $m$ white points respectively.

\item Finally, consider the case of paths. It was shown in \cite[Sec.~5]{dotsenko2019toric}, that the wonderful compactification $\bM(\Path_n)$ associated with the path is an $(n-2)$-dimensional toric projective variety. Its dual polytope is the Stasheff polytope $\mathcal{K}_{n-2}$. Recall that the faces of this polytope are encoded by planar rooted trees, i.e., $\Path_n$-admissible trees. Moreover, the locally open stratification that comes from nested sets is equivalent to the torus orbit decomposition $\bM(\Path_n)=\coprod_{T \in \Tree(\Path_n)} \Orb_T$.
\end{itemize}
For a graph $\Gr$, elements of the building set $\G(\Gr)$ correspond to non-trivial tubes. In terms of tubes, the nested condition on subsets of non-trivial tubes is as follows: a subset of non-trivial tubes $\Susp$ is nested if every two tubes $G,G'$ are comparable by inclusion $G\subset G'$ or do not intersect $G\cap G'=\varnothing$. As we have shown above, for some types of graphs, there is a one-to-one correspondence between elements of an augmented nested set complex $\N(\parti(\Gr),\G(\Gr))\cup \{\varnothing\}$ and \textit{stable} $\Gr$-admissible trees, i.e., trees whose vertices have at least 2 inputs. This observation remains true for all connected graphs.
\begin{prop}\label{nesttrees}
Let $\Gr$ be a connected graph. There is the natural one-to-one correspondence between stable $\Gr$-admissible rooted trees and elements of augmented nested set complex
\[
\Tree_{\mathrm{st}}(\Gr) \cong \N(\parti(\Gr),\G(\Gr))\cup \{\varnothing\}.
\]
\end{prop}
\begin{proof}
Recall that for a $\Gr$-admissible tree $T$ and edge $e \in \Edge(T)$, the set of leaves $L_e=\Leav(T_e)$ of the rooted subtree defines tube. It is easy to see that the collection of tubes $\Susp(T)=\{L_e| e \in \Edge(T)\}$ belongs to the augmented complex  $\N(\parti(\Gr),\G(\Gr))\cup \{\varnothing\}$. Indeed, for the corolla, the collection $\Susp(T)=\varnothing$ is the $-1$-cell; for other trees, $\Susp(T)$ is nested since, for each pair of edges of a rooted tree, the related subtrees are comparable by inclusion order or do not intersect at all. So, we have the map $\Susp\colon \Tree_{\mathrm{st}}(\Gr) \rightarrow \N(\parti(\Gr),\G(\Gr))\cup \{\varnothing\}$, which is bijective by the stability condition.
\end{proof}
 
\section{Koszul duality for contractads}\label{sec:koszul}
In this section, we construct a Bar-Cobar adjunction for contractads and Koszul duality theory for quadratic contractads. The main purpose of these constructions is finding minimal models for contractads. For Koszul contractads we construct these models explicitly. At the end, we give the first examples of Koszul and non-Koszul contractads respectively.  We outline the necessary statements; most of the proofs are \textit{mutatis mutandis} those of \cite[Chap.~6-7]{loday2012algebraic}.
\subsection{Bar-Cobar adjunction}\label{subsec:bar}

\begin{defi}
A differential graded contractad is a contractad in the category of differential-graded spaces. More explicitly, it is a pair of a linear graded contractad $\Pop$ and a differential $d_{\Pop}$ compatible with a contractad structure
\[
d_{\Pop}(\mu \circ^{\Gr}_{G}\nu)=d_{\Pop}(\mu)\circ^{\Gr}_{G}\nu+(-1)^{|\mu|}\mu\circ_G^{\Gr}d_{\Pop}(\nu).
\]
\end{defi}
\noindent The definition of differential graded cocontractad is defined in the dual way.\\

\noindent\textbf{Twisting morphisms.}
For graphical collections $\Pop, \Q$, define their \textit{infinitesimal product} by the rule 
\[
\Pop\circ'\Q(\Gr)=\bigoplus_{G} \Pop(\Gr/G)\otimes \Q(\Gr|_G),
\] where the sum is taken over all tubes $G$ of $\Gr$. Note that the infinitesimal product of $\Pop$ with itself is isomorphic to the weight 2 part of the free contractad on $\Pop$
\[
\Pop\circ'\Pop \cong \T^{(2)}(\Pop).
\]
Let $\Pop$ and $\Q$ be a pair of a dg contractad and a dg cocontractad respectively. Consider the space of homogeneous morphisms of graphical collections $\Hom_{\GrCol}(\Q,\Pop)$. This collection admits the structure of a dg complex with respect to the differential $\partial$
\[
\partial(f):=d_{\Pop}\circ f-(-1)^{|f|}f\circ d_{\Q}.
\] Similarly to the operad case, we define a convolution of two morphisms $f \star g$ by the composition
\[
\Q \overset{\Delta'}{\longrightarrow} \Q\circ'\Q \overset{f \circ' g}{\longrightarrow} \Pop\circ'\Pop \overset{\gamma'}{\longrightarrow} \Pop.
\] The partial axioms of contractads~\ref{partdefi} ensure that this product assembles  the complex $\Hom_{\GrCol}(\Q,\Pop)$ into a dg pre-Lie algebra:
\[
(f\star g)\star h-f\star(g\star h)=(-1)^{|g||h|}((f\star h)\star\ g-f\star(h\star g)).
\] 
The alternating sum $[f,g]=f\star g - (-1)^{|f||g|}g\star f$ defines the dg Lie algebra
\[
\Hom(\Q,\Pop):=(\Hom_{\GrCol}(\Q,\Pop), \partial, [,]),
\] which we refer to as the \textit{convolution Lie algebra}. A \textit{twisting morphism} is a degree $-1$ solution of the Maurer-Cartan Equation
\[
\partial(\alpha)+\frac{1}{2}[\alpha,\alpha]=\partial(\alpha)+\alpha\star\alpha=0.
\] Similarly to operads~\cite[Sec.~6.4.11]{loday2012algebraic}, such an element twists the complex $\Q\circ\Pop=(\Q\circ\Pop, d_{\Q}\circ \Id+d_{\Pop}\circ \Id)$ to a new one  $\Q \circ^{\alpha}\Pop=(\Q\circ\Pop, d_{\alpha})$ on the same graphical collection. More explicitly, the differential $d_{\alpha}$ have  the form 
\[
d_{\alpha}=d_{\Q}\circ \Id_{\Pop} + \Id_{\Q}\circ d_{\Pop}+d_{\alpha}^{r},
\] where $d_{\alpha}^r=(\Id_{\Q}\circ \gamma)([(\Id_{\Q}\circ'\alpha)\Delta']\circ \Id_{\Pop})$ is the twisting term. This map is square-zero  since, for any $(-1)$-degree element $\alpha$, we have $d^{2}_{\alpha}=d^r_{\partial(\alpha)+\alpha\star\alpha}$. We denote by $\Tw(\Q,\Pop)$ the set of twisting morphisms.\\

\noindent\textbf{Bar-Cobar construction.}
Let $\E$ be a graphical collection. By the Leibniz rule, any derivation of the free contractad $\T(\E)$ on $\E$ is uniquely determined  by the images of generators, and the freeness guarantees that any morphism of graphical collections $\E \rightarrow \T(\E)$ extends to a unique derivation. Dually, each coderivation of the free cocontractad $\T^c(\E)$ is determined by the projection on the cogenerators:
\begin{gather*}
    \Der(\T(\E))\cong \Hom_{\GrCol}(\E,\T(\E)),
    \\
    \Coder(\T^c(\E))\cong \Hom_{\GrCol}(\T^c(\E),\E).
\end{gather*}

Let $(\Pop,d_{\Pop})$ be a dg contractad with an augmentation $\varepsilon\colon \Pop \twoheadrightarrow \mathbb{1}$. Consider the free cocontractad generated by the suspension of the augmentation ideal $\T^c(s\overline{\Pop})$. We define the pair $d_1,d_2$ of differentials on this cocontractad arising from a dg contractad structure on $\Pop$ as follows. The differential $d_1$ is a coderivation induced from the morphism of graphical collections 
\[
 \T^c(s\overline{\Pop}) \twoheadrightarrow s\overline{\Pop} \overset{sd_{\Pop}}{\rightarrow} s\overline{\Pop}.
\] This coderivation is square-zero because the $d_{\Pop}$ is. The second coderivation $d_2$ arises from the infinitesimal product
\[
\T^c(s\overline{\Pop})\twoheadrightarrow \T^{(2)}(s\overline{\Pop})\cong s\overline{\Pop}\circ's\overline{\Pop}\cong (s\otimes s)\otimes \overline{\Pop}\circ'\overline{\Pop}\overset{\gamma_s\otimes \gamma'}{\longrightarrow} s\overline{\Pop},
\] where $\gamma_s\colon s\otimes s \mapsto s$ is the desuspension map and $\gamma'\colon \overline{\Pop}\circ'\overline{\Pop} \rightarrow \overline{\Pop}$ is the reduced infinitesimal product. Proposition~\ref{partdefi} and sign conventions ensure that coderivation $d_2$ is square-zero. The compatibility of the differential $d_{\Pop}$ with the infinitesimal compositions in $\Pop$ implies the identity $d_1d_2+d_2d_1=0$. So, the sum $d_1+d_2$ is a square-zero coderivation of the free cocontractad $\T^c(s\overline{\Pop})$.

\begin{defi}
The Bar construction of an augmented dg contractad $\Pop$ is the dg cocontractad
\[
\mathsf{B}\Pop:=(\T^c(s\overline{\Pop}), d_1+d_2).
\]
\end{defi}

Dually, we define the Cobar construction of a coaugmented cocontractad as follows. Let $(\Q,d_{\Q})$ be a dg cocontractad with a coaugmentation $\eta\colon \mathbb{1} \hookrightarrow \Q$. Consider the free contractad generated by the desuspension of the coaugmentation cokernel $\T(s^{-1}\overline{\Q})$. We define the pair $d_1,d_2$ of differentials on this contractad as follows. The differential $d_1$ is a square-zero derivation induced from the morphism of graphical collections 
\[
 s^{-1}\overline{\Q} \overset{s^{-1}d_{\Q}}{\rightarrow} s^{-1}\overline{\Q} \hookrightarrow \T(s^{-1}\overline{\Q})
\] The second derivation $d_2$ arises from
\[
s^{-1}\overline{\Q} \overset{\gamma_{s^{-1}}\otimes \triangle'}{\longrightarrow} (s^{-1}\otimes s^{-1})\otimes \overline{\Q}\circ'\overline{\Q}\cong  s^{-1}\overline{\Q}\circ's^{-1}\overline{\Q} \hookrightarrow \T(s^{-1}\overline{\Q}),
\] where $\gamma_{s^{-1}}\colon s^{-1} \mapsto s^{-1}\otimes s^{-1} $ is the desuspension map and $\triangle'\colon \overline{\Q} \rightarrow \overline{\Q}\circ' \overline{\Q}$ is the reduced infinitesimal coproduct. Similarly to the previous case, the sum $d_1+d_2$ is a square-zero derivation of $\T(s^{-1}\Q)$.

\begin{defi}
The Cobar construction of a coaugmented dg cocontractad $\Q$ is the dg contractad
\[
\mathsf{\Omega}\Q:=(\T(s^{-1}\overline{\Q}), d_1+d_2).
\]
\end{defi}
\noindent The following proposition explains the significance of these constructions.
\begin{prop}[Bar-Cobar Adjunction]
\label{barcobar}
The Bar $\mathsf{B}$ and Cobar construction $\mathsf{\Omega}$ form an adjoint pair. Moreover, we have
\[
\Hom_{\mathsf{dgCon}}(\mathsf{\Omega}\Q,\Pop)\cong \Tw(\Q,\Pop)\cong \Hom_{\mathsf{dgcoCon}}(\Q,\mathsf{B}\Pop).
\]
\end{prop}
\begin{proof}
Let us establish the first isomorphism, the second one is established analogously. Let $f\colon \mathsf{\Omega}\Q\rightarrow \Pop$ be a morphism of dg contractads. This morphism is characterized by its restriction $s^{-1}\overline{\Q} \rightarrow \Pop$ since the Cobar construction $\mathsf{\Omega}\Q=\T(s^{-1}\overline{\Q})$ is a free contractad. It is easy to see that the associated map $\Bar{f}\colon \Q\rightarrow \Pop$ is a twisted morphism if and only if $f\colon \mathsf{\Omega}\Q\rightarrow \Pop$ is a chain map.
\end{proof}
Let us examine when a twisting morphism $\alpha \in \Tw(\Q,\Pop)$ produces a quasi-isomorphism $\Omega(\Q)\overset{\simeq}{\rightarrow}\Pop$ of dg contractads.
\begin{defi}[Koszul morphisms]
A twisting morphism $\alpha\colon \Q \rightarrow \Pop$ is called a Koszul if the twisted contracted product $\Q\circ^{\alpha}\Pop$ is acyclic.
\end{defi}
 For example, for each contractad $\Pop$, there is the twisting morphism $\pi\colon \mathsf{B}\Pop \twoheadrightarrow \Pop$ given by projection and desuspension. Similarly to the operad case~\cite[Lem.~6.5.14]{loday2012algebraic}, we see that the contracted complex $\mathsf{B}\Pop \circ^{\pi} \Pop$ is acyclic. The following proposition is proved  analogously to \cite[Th.~6.6.2]{loday2012algebraic}.
 \begin{prop}
 For an augmented contractad $\Pop$ and coaugmented cocontractad $\Q$, and twisting morphism $\alpha: \Q \rightarrow \Pop$, the following assertions are equivalent:
 \begin{enumerate}
     \item $\alpha$ is Koszul,
     \item $f_{\alpha}\colon \Q \rightarrow \mathsf{B}\Pop$ is a quasi-isomorphism,
     \item $g_{\alpha}\colon \mathsf{\Omega}\Q \rightarrow \Pop$ is a quasi-isomorphism.
 \end{enumerate}
 \end{prop}

 \noindent As an immediate corollary we get, that for each dg contractad $\Pop$, the morphism $g_{\pi}\colon \mathsf{\Omega B}\Pop \twoheadrightarrow \Pop$ is a quasi-isomorphism. Dually, for each dg cocontractad, the unit $\Q\rightarrow\mathsf{B\Omega}\Q$ is a quasi-isomorphism.

\subsection{Koszul Duality}\label{subsec:koszul}
In this subsection, we define Koszul contractads and construct their minimal models. As we have mentioned before, for an arbitrary dg contractad $\Pop$, we have the free model given by the counit $\mathsf{\Omega B}\Pop \twoheadrightarrow \Pop$. Unfortunately, this model is too "large" and it does not give much useful information about the latter contractad. A good candidate for such a role is a \textit{minimal model}. A minimal model of a dg contractad $(\Pop, d_{\Pop})$ is a free dg contractad $\T(\E)$ with a surjective quasi-morphism $\T(\E) \twoheadrightarrow \Pop$ and whose differential is \textit{decomposable}  $d(\E) \subset \T^{(\geq 2)}(\E)$, i.e., images of generators have the weight of at least 2.

\begin{defi}
The suspension contractad is the endomorphism contractad $\Susp:=\End_{\mathsf{k}s}$ for the one-dimensional vector space $\mathsf{k}s$, $|s|=1$. Similarly, the desuspension contractad $\Susp^{-1}:=\End_{\mathsf{k}s^{-1}}$, where $|s^{-1}|=-1$.
\end{defi}
Each component of the former contractad has the form $\Susp(\Gr)=\Hom(\mathsf{k}s^{\otimes n},\mathsf{k}s)$, where $n$ is the number of vertices, is a one-dimensional vector space concentrated in degree $(1-n)$, since its generator sends $s^n$ to $s$. The automorphism group $\Aut(\Gr)$ acts on this space by the sign of the associated permutations
\[
(s^n\mapsto s)^{\sigma}=(-1)^{|\sigma|}(s^n\mapsto s).
\]

Contractads, themselves, form a symmetric monoidal category as follows. For a given pair of contractads $\Pop, \mathcal{Q}$, their \textit{Hadamard product} $\Pop \underset{\mathrm{H}}{\otimes}\mathcal{Q}$ is the contractad whose underlying graphical collection is given by the tensor product
\[
\Pop\underset{\mathrm{H}}{\otimes} \mathcal{Q}(\Gr):=\Pop(\Gr)\otimes \mathcal{Q}(\Gr),
\] with the natural contractad structure 
\[
(\Pop\underset{\mathrm{H}}{\otimes}\Q)(\Gr/I)\otimes\bigotimes_{G\in I} (\Pop\underset{\mathrm{H}}{\otimes}\Q)(\Gr|_G)\cong (\Pop(\Gr/I)\otimes\bigotimes_{G\in I} \Pop(\Gr|_G))\otimes(\Q(\Gr/I)\otimes\bigotimes_{G\in I} \Q(\Gr|_G))\overset{\gamma_{\Pop}\otimes\gamma_Q}{\longrightarrow}(\Pop\underset{\mathrm{H}}{\otimes}\Q)(\Gr).
\] Note that the commutative contractad $\Com$ is the unit with respect to the Hadamard product
\[
\Com\underset{\mathrm{H}}{\otimes}\Pop\cong 
\Pop\underset{\mathrm{H}}{\otimes}\Com\cong \Pop.
\] For a contractad $\Pop$, its (de)suspension is the contractad defined by the rule
\[
\Susp\Pop=\Susp\underset{\mathrm{H}}{\otimes}\Pop \quad \Susp^{-1}\Pop=\Susp^{-1}\underset{\mathrm{H}}{\otimes}\Pop
\]

\begin{defi}[Koszul dual contractad]
Let $\Pop=\Pop(\E, \R) $ be a contractad generated by generators $\E$ and quadratic relations $\R \subset \E\circ'\E$. The  Koszul dual cocontractad $\Pop^{\text{!`}}=\Q(s\E,s^2\R)$ is the cocontractad with cogenerators $s\E$ and corelations $s^2\R$. The Koszul dual contractad $\Pop^{!}$ is the contractad defined by
\[
\Pop^{!}=\mathcal{S}^{-1}(\Pop^{\text{!`}})^*
\]
\end{defi}
Similarly to operads~\cite[Pr.~7.2.4]{loday2012algebraic}, for a quadratic contractad, its Koszul dual is also a quadratic contractad. The following proposition describes the presentation of Koszul dual contractads. 
\begin{prop}
    Let $\Pop=\Pop(\E, \R)$ be a quadratic contractad, whose generators $\E$ are finite-dimensional in each component. Then the Koszul dual contractad admits the quadratic presentation
    \[
    \Pop^{!}=\Pop(s^{-1}\Susp^{-1}\E^*, \R^{\bot}),
    \] where $\R^{\bot}$ is the annihilator of $\R$ with respect to the natural pairing
    \[
    \langle-,-\rangle\colon (\E\circ'\E)\otimes (s^{-1}\Susp^{-1}\E^*\circ' s^{-1}\Susp^{-1}\E^*)\rightarrow \mathsf{k}.
    \]
    Furthermore, we have an isomorphism of contractads
    \[
    (\Pop^!)^!\cong \Pop.
    \]
\end{prop}

\noindent Let us consider some examples of Koszul dual contractads:
\begin{itemize}
    \item By Proposition~\ref{quadcom}, the commutative contractad $\Com$ is quadratic. By the previous proposition, its Koszul dual contractad coincides with the Lie contractad
    \[
    \Com^{!}\cong \Lie.
    \]
    \item By Proposition~\ref{rst}, the contractad $\RST^{\vee}$ of rooted spanning trees is quadratic. Its Koszul dual $\RST^!$ is the contractad generated by a generator $\nu=\mu^!$, satisfying the relations
    \begin{gather*}
    \nu \circ^{\Path_3}_{\{1,2\}} \nu= \nu \circ^{\Path_3}_{\{2,3\}} \nu,
     \\
     \nu \circ^{\Path_3}_{\{1,2\}} \nu^{(12)} = \nu^{(12)} \circ^{\Path_3}_{\{2,3\}} \nu,
     \\
     \nu \circ^{\K_3}_{\{1,2\}} \nu= \nu \circ^{\K_3}_{\{2,3\}} \nu=(\nu \circ^{\K_3}_{\{2,3\}}\nu)^{(23)}.
    \end{gather*} Note that this contractad is obtained from a set-theoretical contractad by linearization.
    
\end{itemize}
The bar construction $\mathsf{B}\Pop$ of a quadratic contractad $\Pop=\Pop(\E,\R)$ is equipped with the \textit{syzygy degree} as follows. Since quadratic relations $\R$ are homogenous with respect to generators $\E$, we have a well-defined weight grading of $\Pop$. The syzygy degree on $\mathsf{B}\Pop$ is defined by $\omega(s\alpha_1\otimes s\alpha_2\otimes\cdots\otimes s\alpha_k)=\omega(\alpha_1)+\cdots+\omega(\alpha_k)-k$, where $\omega(\alpha_i)$ are weights of elements from $\Pop$.  Since $\Pop$ has trivial internal differential, the differential on $\mathsf{B}\Pop$ reduces to $d_2$, which raises the syzygy degree by 1. Hence, $\mathsf{B}^{\bullet}\Pop$ forms a cochain complex with respect to the syzygy degree. By the definition, $\Pop$, itself, is concentrated  in the zero syzygy degree.

Similarly, the Cobar construction $\mathsf{\Omega}\Q$ over the quadratic cocontractad $\Q$ has syzygy degree. In contrast with bar construction, the differential decreases the syzygy degree by 1.  Hence, $\mathsf{\Omega}_{\bullet}\Q$ forms a chain complex with respect to the syzygy degree.

Let $\Pop=\Pop(\E,\R)$ be a quadratic contractad. We have the canonical twisting morphism $\tau\colon \Pop^{\text{!`}} \to \Pop$ given by the formula
\[
\Pop^{\text{!`}}\twoheadrightarrow s\E \overset{s^{-1}}{\rightarrow} \E \hookrightarrow \Pop.
\] By Proposition~\ref{barcobar}, this twisting morphism induces the morphisms $\Pop^{\text{!`}} \hookrightarrow \mathsf{B}\Pop$ and $\mathsf{\Omega}\Pop^{\text{!`}} \twoheadrightarrow \Pop$, which we shall refer to as the Koszul inclusion and projection respectively. Moreover, the differential in the Cobar construction $\mathsf{\Omega}\Pop^{\text{!`}}$  is decomposable. Similarly to operads \cite[Pr.~7.3.2]{loday2012algebraic}, the Koszul inclusion and projection induce isomorphisms in the zero degree (co)homology
\[
\Pop^{\text{!`}} \overset{\cong}{\rightarrow} H^0(\mathsf{B}^{\bullet}\Pop), \quad H_{0}(\mathsf{\Omega}_{\bullet}\Pop^{\text{!`}}) \overset{\cong}{\rightarrow} \Pop,
\] with respect to the syzygy degree.

\begin{defi}[Koszul contractads]
A quadratic contractad $\Pop$ is called Koszul if the twisting morphism $\tau\colon \Pop^{\text{!`}} \rightarrow \Pop$ is Koszul. Equivalently, the Koszul projection $\mathsf{\Omega}\Pop^{\text{!`}} \rightarrow \Pop$ is a minimal model of $\Pop$, and the Koszul inclusion $\Pop^{\text{!`}} \rightarrow \mathsf{B}\Pop$ is a minimal comodel for $\Pop^{\text{!`}}$.
\end{defi}
Let us illustrate the first example of a Koszul contractad.
\begin{theorem}\label{thm:comkoszul}
The commutative contractad $\Com$ is Koszul.
\end{theorem}
\begin{proof}
Consider the bar construction $\mathsf{B}\Com$ of the commutative contractad. On the level of cocontractads, we have $\mathsf{B}\Com=\T^c(s\overline{\Com})$. Each component of the underlying graphical collection has the form
\[
\mathsf{B}\Com(\Gr)=\bigoplus_{T\in \Tree_{\mathrm{st}}(\Gr)}\bigotimes_{v\in \Ver(T)} \mathsf{k}s,
\] Hence, the component $\mathsf{B}\Com(\Gr)$ is spanned by vertex-ordered stable $\Gr$-admissible rooted trees $(T,s_{v_1}\otimes s_{v_2}\otimes \cdots \otimes s_{v_n}\otimes s_{v_{root}})$, but with different orderings identified, up to the obvious sign factor. The differential of $\mathsf{B}\Com$ is given by edge contractions of rooted trees as follows
\begin{equation*}
    d(T,s_{v_1}\otimes \cdots \otimes s_{v_n}\otimes s_{v_{root}})=\sum^n_{i=1} (-1)^{i-1} (T/e_i,s_{v_1}\otimes \cdots \hat{s}_{v_i}\cdots\otimes s_{v_{root}})
\end{equation*} where $v_{root}$ is the root vertex, $e_i$ is the edge outgoing from the vertex $v_i$. Consider the total order $<$ on the vertex set $V_{\Gr}$ which refines the inclusion order. With respect to this order, we define a canonical vertex ordering for a stable $\Gr$-admissible rooted tree $T$ as follows. For a pair of vertices $v,w\in \Ver(T)$, we put $v\prec w$ if $L_{e_v}<L_{e_w}$, where $L_{e_v}$ is the leaf set of the subtree with the root at the edge $e_v$ outgoing from $v$.

Let us consider the map $\Susp\colon \Tree_{\mathrm{st}}(\Gr)\rightarrow \N(\parti(\Gr), \G(\Gr))\cup \{\varnothing\}$ constructed in Proposition~\ref{nesttrees}. The ordering of tubes with respect to $<$-order turns the nested set complex $\N(\parti(\Gr), \G(\Gr))$ into an ordered simplicial complex. With respect to this ordering, we have $\Susp(T/e_i)=\Susp(T)\setminus \{L_{e_i}\}=\partial_i(\Susp(T))$, where $\partial_i$ is the $i$-th degeneracy map. Hence, the linearisation of $\Susp$ preserves differentials
\[
\Susp(dT)=\sum^{n}_{i=1}(-1)^{i-1}\Susp(T/e_i)=\sum^{n}_{i=1}(-1)^{i-1}\partial_i(\Susp(T))=\partial\Susp(T),
\] where trees above are considered with canonical vertex-orderings. Since $\Susp$ is a bijective map, its linearisation defines an isomorphism of dg complexes (after suitable reordering of grading)
\[
\Susp\colon \mathsf{B}^{\bullet}(\Com)(\Gr)\overset{\cong}{\longrightarrow} \mathsf{C}_{(n-3)-\bullet}^{+}(\N(\parti(\Gr), \G(\Gr))),
\] where $\mathsf{C}^{+}(\N(\parti(\Gr), \G(\Gr)))$ is the augmented simplicial complex. It was proved in \cite[Cor.~3.4]{feichtner2005topology} that, for any atomic lattice $\La$ and connected building set $\G$, the complex of nested sets is homeomorphic to the reduced order complex of this lattice (simplicial complex of ordered chains from $\hat{0}$ to $\hat{1}$): $\N(\La,\G)\cong \Delta(\La)$. Combining with the previous isomorphism, we obtain
\begin{equation*}
    H^{\bullet}(\mathsf{B}\Com(\Gr)) \cong \Tilde{H}_{(n-3)-\bullet}(\N(\parti(\Gr),\G(\Gr)))\cong \Tilde{H}_{(n-3)-\bullet}(\parti(\Gr)),
\end{equation*} where $\Tilde{H}_{\bullet}(\parti(\Gr))$ are homology groups of the reduced ordered complex of $\parti(\Gr)$.  By Proposition~\ref{grapharrang}, for each graph, the partition poset $\parti(\Gr)$ is isomorphic to the intersection lattice of a hyperplane arrangement. It is known that the reduced homology groups of an intersection lattice are concentrated in the top degree~\cite{bjorner1980shellable}. Hence, the cohomology groups of the Bar-construction are concentrated in the zero degree. So, the Koszul embedding $\Com^{\text{!`}}\hookrightarrow \mathsf{B}\Com$ is a quasi-isomorphism.
\end{proof}
As an immediate consequence, we get
\begin{cor}\label{cor:liekoszul}
The Lie contractad $\Lie$ is Koszul. Moreover, the dimension of each component is given by the formula
\[
\dim \Lie(\Gr)=|\mu(\Gr)|,
\] where $\mu(\Gr):=\mu_{\parti(\Gr)}(\hat{0},\hat{1})$ and $\mu_{\parti(\Gr)}$ is the M\"obius function of the poset of graph-partitions.
\end{cor}
\begin{proof}
Recall that the contractad $\Lie$ is Koszul dual to $\Com$. The first assertion follows from the observation that the Koszul dual contractad of a Koszul contractad is also Koszul. The second one follows from the isomorphism $\Com^{\text{!`}}(\Gr)\cong H_{n-3}(\parti(\Gr))$ and the fact that the rank of this homology group is equal to the evaluation of the associated M\"obius function on the pair $(\hat{0},\hat{1})$~\cite{bjorner1980shellable}.
\end{proof}
Let us present an example of a non-Koszul quadratic contractad.
\begin{prop}\label{prop:nonkoszul}
The contractad of rooted spanning trees $\RST^{\vee}$  is not Koszul.
\end{prop}
\begin{proof}
Consider the Koszul complex $\RST^{\text{!`}}\circ^{\tau} \RST$. For the cycle on 4 vertices, the complex $\RST^{\text{!`}}\circ^{\tau} \RST(\Cyc_4)$ has the form:
\begin{equation*}
\scalebox{0.75}{$
0\to \RST^{\text{!`}}(\Cyc_4)\to(\RST^{\text{!`}}(\Cyc_3)\otimes \RST(\Path_2))^{\oplus 4}\to (\RST^{\text{!`}}(\Path_2)\otimes \RST(\Path_2)\otimes \RST(\Path_2))^{\oplus 2} \oplus (\RST^{\text{!`}}(\Path_2)\otimes \RST(\Path_3))^{\oplus 3}\to \RST(\Cyc_4)\to 0$}.
\end{equation*} By direct computation of the Euler characteristic, we have
\[
\chi(\RST^{\text{!`}}\circ \RST(\Cyc_4))=-\dim\RST^{\text{!`}}(\Cyc_4).
\]Note that the dimension of each component of $\RST^{\text{!`}}$ can not be zero since the Koszul dual contractad $\RST^{!}$ is obtained from the set-theoretical contractad by linearization. So, we conclude that the Euler characteristic $\chi(\RST^{\text{!`}}\circ^{\tau} \RST(\Cyc_4))$ is non-zero. Hence, the complex $\RST^{\text{!`}}\circ^{\tau} \RST$ is not acyclic.
\end{proof}

\section{Gr\"obner bases for contractads}\label{sec:grobner}
In this section, we develop the theory of Gr\"obner bases for contractads. This theory gives us efficient tools for finding bases in contractads. Moreover, Gr\"obner bases allow us to reduce questions about the presentation of contractads or koszulity to computational tasks. First, we define \textit{shuffle} contractads similarly to the operad case~\cite{dotsenko2010grobner}. Next, we define the notion of Gr\"obner bases and give the first examples of such bases. We then explain how this theory is applied to Koszul duality theory. As in the previous section, we outline the necessary statements; most of the proofs are \textit{mutatis mutandis} those of \cite{dotsenko2010grobner}.

\subsection{Shuffle contractads}
An \textit{ordered graph} is a simple graph $(\Gr,<)$ with a total order on the set of vertices $V_{\Gr}$. For a tube $G$ of the graph $(\Gr,<)$, the restriction of the order to the tube $G$ defines the induced  ordered subgraph $(\Gr|_G,<_{\res})$. For a partition $I=\{G_1,G_2,...,G_k\}$ of the graph $\Gr$, we define the order on the vertex set of the contracted graph $\Gr/I$ by comparing minimal vertices from blocks
\begin{equation*}
    \{G_i\}<^{\ind} \{G_j\}:= \min_{v\in G_i} v<\min_{v\in G_j} v.
\end{equation*}Consider the category $\mathsf{OCGr}$ of ordered connected simple graphs with order-preserving isomorphisms. Note that in this setting each ordered graph has no non-trivial automorphisms.

\begin{defi}
A non-symmetric graphical collection with values in $\C$ is a contravariant functor from the category $\mathsf{OCGr}$ to the category $\C$. 
\end{defi} 
\noindent We define the \textit{shuffle contraction product} of ns graphical collections by the rule:
\begin{equation*}
    (\Pop\circ_{\Sha}\Q)(\Gr,<)=\bigoplus_{I \in \parti(\Gr)} \Pop(\Gr/I,<^{\ind})\otimes\bigotimes_{\substack{I=\{G_1,G_2,...,G_k\}\\ \min G_1<\cdots<\min G_k}} \Q(\Gr|_{G_i},<_{\res}).
\end{equation*} Similarly with contractads, this product endows the category of ns graphical collections with a monoidal category structure.
\begin{defi}
A shuffle contractad is a monoid in the monoidal category of non-symmetric graphical collections equipped with the shuffle contraction product $\circ_{\Sha}$.
\end{defi} 
\noindent Analogously to Section~\ref{subsec:contractads}, we can define infinitesimal compositions for shuffle contractads and the monad $\T_{\Sha}$ of shuffle admissible rooted trees. Let us discuss the relationship between contractads and shuffle ones. For a graphical collection $\Orb$, we define its shuffle version $\Orb^{\forget}$ by 
\[
\Orb^{\forget}(\Gr,<)=\Orb(\Gr)
\]
The following proposition is proved by direct inspection.
\begin{prop}
The forgetful functor is monoidal
\[
    (\Pop\circ\Q)^{\forget}=\Pop^{\forget}\circ_{\Sha} \Q^{\forget}.
\]
\end{prop}
\noindent As an immediate consequence, we have the following results
\begin{cor}
\begin{enumerate}
\item The forgetful functor sends contractads to shuffle ones.
\item For every graphical collection $\E$, we have an isomorphism of shuffle contractads
 \[
\T(\E)^{\forget}\cong\T_{\Sha}(\E^{\forget}). 
 \]
\item For every graphical collection $\E$ and every graphical subcollection $\R\subset\T(E)$, we have an isomorphism of shuffle contractads
 \[
\T(\E)/\langle \R)^{\forget}\rangle \cong \T_{\Sha}(\E^{\forget})/\langle \R^{\forget}\rangle.
 \]
\item For every contractad $\Pop$, we have an isomorphism of dg shuffle contractads
 \[
\mathsf{B}(\Pop)^{\forget}\cong\mathsf{B}_{\Sha}(\Pop^{\forget}).
 \] 
\item A quadratic contractad $\Pop$ is Koszul if and only if the associated shuffle contractad $\Pop^{\mathrm{f}}$ is Koszul.
\end{enumerate}
\end{cor}
\subsection{Monomials and orders}

This section is a direct generalization of that of \cite[Sec.~3.1-3.3]{dotsenko2010grobner}. To define Gr\"obner basis, one needs a notion of "monomials" of the free contractad and ordering compatible with the contractad structure. In the case of shuffle operads, the role of monomials is played by decorated trees. The same idea applies to contractads.

Consider some  linear graphical collection $\E$ and some basis $\B \subset \E$, which we shall refer to as an alphabet. Then the free shuffle contractad $\T_{\Sha}(\E)$ is obtained from the set-theoretical free shuffle contractad $\T_{\Sha}(\B)$ by linearization. Each element of the shuffle contractad $\T_{\Sha}(\B)$ can be expressed as an admissible tree, whose vertices are decorated by elements of the alphabet $\B$. Hence why we denote this contractad by $\Tree_{\B}$ and call elements of this contractad \textit{tree monomials} of $\T_{\Sha}(\E)$.

Consider some tree monomial $T \in \Tree_{\B}(\Gr,<)$. For each edge $e$,  the rooted subtree $T_e \subset T$ with a root at $e$ defines a tree monomial in the component $\Tree_{\B}(\Gr|_{L_e}, <_{\res})$, where $L_e \subset V_{\Gr}$ is the set of leaves of $T_e$. Also, the subtree $T^e\subset T$ obtained by cutting the subtree $T_e$ from the original is a well-defined tree monomial in the  component $\Tree_{\B}(\Gr/L_e)$. So, each subtree of a tree monomial is also a tree monomial. We say that a monomial $T$ is \textit{divisible} by a monomial $T'$ if $T'$ forms a decorated subtree of $T$.

\begin{defi}
A monomial order of tree monomials of  $\T_{\Sha}(\E)$ is a collection of total well orders on each component of $\Tree_{\B}$ such that shuffle compositions are strictly increasing functions: For any tube $G$ of the ordered graph $(\Gr,<)$, and pairs $T_1, T_1' \in \Tree_{\B}(\Gr/G,<^{\ind})$, $T_2, T_2' \in \Tree_{\B}(\Gr|_G,<_{\res})$, we have:
    \begin{gather*}
    T_1\circ_G^{\Gr} T_2 \prec T_1'\circ_G^{\Gr} T_2  \text{, if } T_1\prec T'_1
    \\
    T_1\circ_G^{\Gr} T_2 \prec T_1\circ_G^{\Gr} T_2'  \text{, if } T_2\prec T'_2
    \end{gather*}
\end{defi}

In practice, it is more convenient to deal with planar trees. For a tree with labeled leaves, its canonical planar representative is defined as follows. Let $T$ be a rooted tree. For each vertex $v \in \Ver(T)$ let $T(v) \subset T$ be a corolla containing this vertex. An embedding of the tree $T$ in a plane is called \textit{canonical} if, for each vertex $v$,  the ordering of leaves of $T(v)$ coincides with the ordering given by the planar structure.
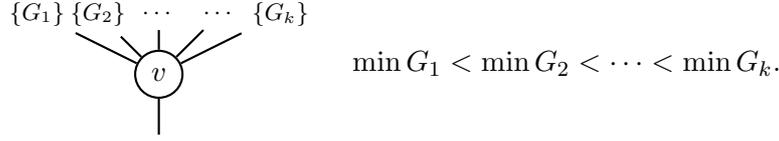
\[
\vcenter{\hbox{\begin{tikzpicture}[
        scale=0.8,
        vert/.style={circle,  draw=black!30!black, thick, minimum size=1mm},
        leaf/.style={rectangle, thick, minimum size=1mm},
        edge/.style={-,black!30!black, thick},
        ]
        \node[vert] (1) at (0,1) {$v$};
        \node[leaf] (l1) at (-2,2) {\footnotesize$\{G_1\}$};
        \node[leaf] (l2) at (-1,2) {\footnotesize$\{G_2\}$};
        \node[leaf] (lmid) at (0,2) {\footnotesize$\cdots$};
        \node[leaf] (l3) at (1,2) {\footnotesize$\cdots$};
        \node[leaf] (lk) at (2,2) {\footnotesize$\{G_k\}$};
        \draw[edge] (0,0)--(1);
        \draw[edge] (1)--(l1);
        \draw[edge] (1)--(l2);
        \draw[edge] (1)--(lmid);
        \draw[edge] (1)--(l3);
        \draw[edge] (1)--(lk);
    \end{tikzpicture}}} \quad \min G_1<\min G_2<\cdots <\min G_k.
\]\\

\noindent\textbf{$\grpermlex$-order.} Let us present the first example of a monomial order. Checking that this order is monomial proceeds \textit{mutatis mutandis} in the same way as for the operad case~\cite[Sec.~3.2.1]{dotsenko2010grobner}. Let $T$ be a tree monomial in $\Tree_{\B}(\Gr)$. We associate to $T$ a \textit{path-sequence} $\Pa(T)=(\alpha_{v_1}, \alpha_{v_2},\cdots, \alpha_{v_n})$ of words labeled by leaves in the increasing order in the alphabet $\B$,
and a \textit{leaf-permutation} $\sigma(T)$ as follows. For each leaf $v_i$ of the underlying tree, there exists a unique path
from the root to $v$. The word $\alpha_{v_i}$ is the word composed, from left to right of the labels of the vertices of this path, starting from the root vertex. The permutation $\sigma(T)$ lists the labels of leaves of the underlying tree in the order
determined by the planar structure (from left to right). It is known that the datum $(\Pa(T),\sigma(T))$ defines the tree monomial uniquely. 
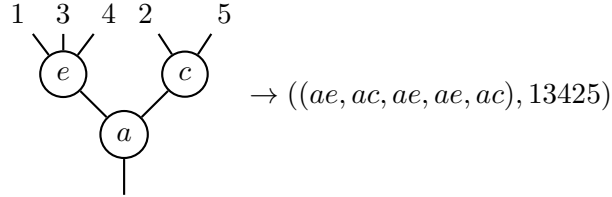
\[
\vcenter{\hbox{\begin{tikzpicture}[
        scale=0.8,
        vert/.style={circle,  draw=black!30!black, thick, minimum size=1mm},
        leaf/.style={rectangle, thick, minimum size=1mm},
        edge/.style={-,black!30!black, thick},
        ]
        \node[leaf] (l1) at (1.65,3) {$5$};
        \node[leaf] (l2) at (0.35,3) {$2$};
        \node[vert] (1) at (0,1) {$a$};
        \node[vert] (2) at (1,2) {$c$};
        \node[vert] (3) at (-1,2) {$e$};
        \node[leaf] (r1) at (-1.75,3) {$1$};
        \node[leaf] (r2) at (-1,3) {$3$};
        \node[leaf] (r3) at (-0.25,3) {$4$};
        \draw[edge] (0,0)--(1);
        \draw[edge] (1)--(2);
        \draw[edge] (1)--(3);
        \draw[edge] (3)--(r1);
        \draw[edge] (3)--(r2);
        \draw[edge] (3)--(r3);
        \draw[edge] (2)--(l1);
        \draw[edge] (2)--(l2);
    \end{tikzpicture}}}\rightarrow ((ae,ac,ae,ae,ac),13425)
\]
\begin{defi}
Let $\B$ be an alphabet with a monomial order on the words $\B^*$. The graphical permutation lexical order $\grpermlex$  on the shuffle trees $\Tree_{\B}$ is the order defined as follows:
\begin{itemize}
    \item For a pair of shuffle trees $T_1,T_2 \in \Tree_{\B}(\Gr)$ we compare the sequences $\Pa(T_1)$ and $\Pa(T_2)$ word by word, comparing words using the monomial order $\B^*$.
     \item If their path sequences coincide, we compare their leaf-permutations $\sigma(T_1),\sigma(T_2)$ by lexicographic order.
\end{itemize}
\end{defi}
\begin{figure}[ht]
    \centering

\[
        \vcenter{\hbox{\begin{tikzpicture}[
        scale=0.7,
        vert/.style={circle,  draw=black!30!black, thick, minimum size=1mm},
        leaf/.style={circle, thick, minimum size=1mm},
        edge/.style={-,black!30!black, thick},
        ]
        \node at (0,-0.5) {$((\alpha,\alpha\alpha,\alpha\alpha),123)$};
        \node[leaf] (l1) at (1.5,3) {\footnotesize$3$};
        \node[leaf] (l2) at (0,3) {\footnotesize$2$};
        \node[leaf] (l3) at (-0.75,2) {\footnotesize$1$};
        \node[vert] (1) at (0,1) {\footnotesize$\alpha$};
        \node[vert] (2) at (0.75,2) {\footnotesize$\alpha$};
        \draw[edge] (0,0)--(1);
        \draw[edge] (1)--(2);
        \draw[edge] (2)--(l1);
        \draw[edge] (2)--(l2);
        \draw[edge] (1)--(l3);
    \end{tikzpicture}}}
    \quad
    <
    \quad
        \vcenter{\hbox{\begin{tikzpicture}[
        scale=0.7,
        vert/.style={circle,  draw=black!30!black, thick, minimum size=1mm},
        leaf/.style={circle, thick, minimum size=1mm},
        edge/.style={-,black!30!black, thick},
        ]
        \node at (0,-0.5) {$((\alpha\alpha,\alpha,\alpha\alpha),132)$};
        \node[leaf] (l1) at (-1.5,3) {\footnotesize$1$};
        \node[leaf] (l2) at (0,3) {\footnotesize$3$};
        \node[leaf] (l3) at (0.75,2) {\footnotesize$2$};
        \node[vert] (1) at (0,1) {\footnotesize$\alpha$};
        \node[vert] (2) at (-0.75,2) {\footnotesize$\alpha$};
        \draw[edge] (0,0)--(1);
        \draw[edge] (1)--(2);
        \draw[edge] (2)--(l1);
        \draw[edge] (2)--(l2);
        \draw[edge] (1)--(l3);
    \end{tikzpicture}}}
    \quad
    <
    \quad
        \vcenter{\hbox{\begin{tikzpicture}[
        scale=0.7,
        vert/.style={circle,  draw=black!30!black, thick, minimum size=1mm},
        leaf/.style={circle, thick, minimum size=1mm},
        edge/.style={-,black!30!black, thick},
        ]
        \node at (0,-0.5) {$((\alpha\alpha,\alpha\alpha,\alpha),123)$};
        \node[leaf] (l1) at (-1.5,3) {\footnotesize$1$};
        \node[leaf] (l2) at (0,3) {\footnotesize$2$};
        \node[leaf] (l3) at (0.75,2) {\footnotesize$3$};
        \node[vert] (1) at (0,1) {\footnotesize$\alpha$};
        \node[vert] (2) at (-0.75,2) {\footnotesize$\alpha$};
        \draw[edge] (0,0)--(1);
        \draw[edge] (1)--(2);
        \draw[edge] (2)--(l1);
        \draw[edge] (2)--(l2);
        \draw[edge] (1)--(l3);
    \end{tikzpicture}}}
    \]
    \caption{$\grpermlex$-order.}
\end{figure}
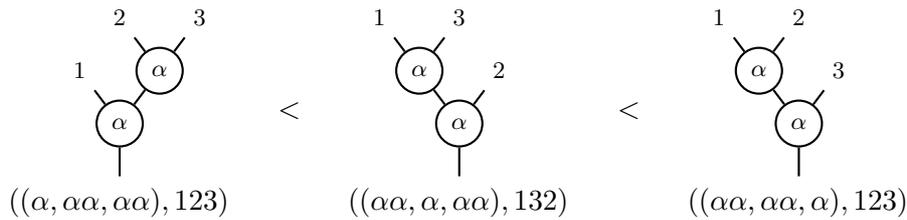
\subsection{Gr\"obner bases}

Consider some free shuffle contractad $\T_{\Sha}(\E)$ with a monomial order on tree monomials $\Tree_{\B}$. For an element $\N \in \T_{\Sha}(\E)(\Gr)$, we denote by $\LT(\N)$ the leading term in the monomial expansion $\N=\alpha\LT(\N)+\sum_{T_i<\LT(\N)}\alpha_i T_i$. It is easy to see that for an ideal $\mathcal{I} \subset \T_{\Sha}(\E)$ in the free contractad, the linear span of leading terms $\LT(\mathcal{I})$ of elements from the ideal forms an ideal, which we refer to as a \textit{leading term ideal}.

\begin{defi}
Let $\mathcal{I}$ be an ideal in the free shuffle operad $\T_{\Sha}(\E)$. A Gr\"obner basis of the ideal $\I$ is a non-symmetric graphical subcollection $\G \subset \mathcal{I}$ whose leading terms $\LT(\G)$ generate the leading term ideal:
\[
(\LT(\G))=\LT(\mathcal{I}).
\]
\end{defi}
  Using classical methods of Gr\"obner bases, it can be shown that the Gr\"obner basis of an ideal $\G \subset \mathcal{I}$ generates this ideal. For a given subcollection $\Ho \subset \T_{\Sha}(\E)$, \textit{normal monomials with respect to} $\Ho$ are monomials that are not divisible by leading terms $\LT(\Ho)$. The following proposition gives us an effective criterion for a set of elements in an ideal to be a Gr\"obner basis.
\begin{prop}[Monomial basis]
\label{coset}
For any subset $\G$ of an ideal $\I$, the set of normal monomials with respect to $\G$ spans the quotient $\T_{\Sha}(\E)/\I$. Moreover, $\G$ is a Gr\"obner basis of the ideal if and only if normal monomials form a basis of the quotient $\T_{\Sha}(\E)/\mathcal{I}$.
\end{prop}

So, now we are ready to prove quadraticity of the commutative contractad $\Com$.

\begin{prop}
\label{quadcom}
The commutative contractad $\Com$ is generated by one symmetric generator $m$ in the component $\Path_2$, satisfying the relations
\begin{gather}
    m \circ_{\{1,2\}}^{\mathsf{P_3}} m = m \circ_{\{2,3\}}^{\mathsf{P_3}} m
    \\
    m \circ_{\{1,2\}}^{\mathsf{K_3}} m = m \circ_{\{2,3\}}^{\mathsf{K_3}} m.
\end{gather}
\end{prop}
\begin{proof} By the construction, each  infinitesimal composition is an isomorphism of one-dimensional spaces. Hence, this contractad is generated by the one-dimensional component $\Com(\mathsf{P}_2)\cong \mathsf{k}\langle m\rangle$. For similar reasons, the relations above are satisfied. Hence, we have the surjective morphism of contractads $\Pop\twoheadrightarrow \Com$, where $\Pop$ is the quadratic contractad obtained from generators and relations above. Consider the shuffle version $\Pop^{\forget}=\T_{\Sha}(m)/\langle \R^{\forget}\rangle$ and the monomial order on $\Tree_{m}$ reverse to the $\grpermlex$-order(with usual $\mathsf{deglex}$-order on words $\{m\}^*$). The leading terms of the relations $\R^{\forget}$ have the form
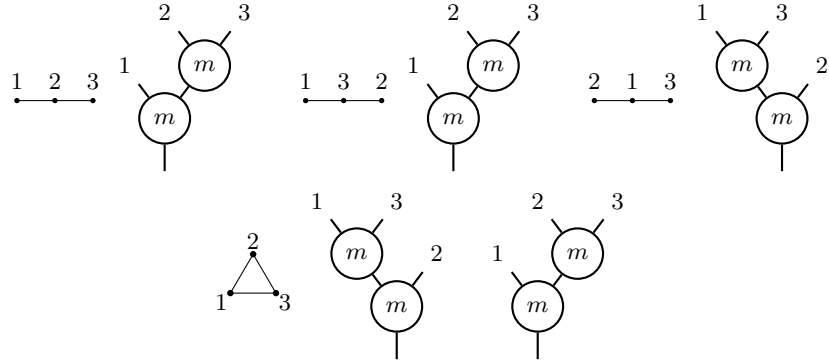
\begin{gather*}
\label{leadtermcom}
    \vcenter{\hbox{\begin{tikzpicture}[scale=0.5]
    \fill (0,0) circle (2pt);
    \node at (0,0.5) {\footnotesize$1$};
    \fill (1,0) circle (2pt);
    \node at (1,0.5) {\footnotesize$2$};
    \fill (2,0) circle (2pt);
    \node at (2,0.5) {\footnotesize$3$};
    \draw (0,0)--(1,0)--(2,0);    
    \end{tikzpicture}}}
    \vcenter{\hbox{\begin{tikzpicture}[
        scale=0.7,
        vert/.style={circle,  draw=black!30!black, thick, minimum size=1mm},
        leaf/.style={rectangle, thick, minimum size=1mm},
        edge/.style={-,black!30!black, thick},
        ]
        \node[vert] (1) at (0,1) {\footnotesize$m$};
        \node[leaf] (l1) at (-0.75,2) {\footnotesize$1$};
        \node[vert] (2) at (0.75,2) {\footnotesize$m$};
        \node[leaf] (l2) at (0,3) {\footnotesize$2$};
        \node[leaf] (3) at (1.5,3) {\footnotesize$3$};
        \draw[edge] (0,0)--(1);
        \draw[edge] (1)--(2)--(3);
        \draw[edge] (1)--(l1);
        \draw[edge] (2)--(l2);
    \end{tikzpicture}}}
    \quad
    \vcenter{\hbox{\begin{tikzpicture}[scale=0.5]
    \fill (0,0) circle (2pt);
    \node at (0,0.5) {\footnotesize$1$};
    \fill (1,0) circle (2pt);
    \node at (1,0.5) {\footnotesize$3$};
    \fill (2,0) circle (2pt);
    \node at (2,0.5) {\footnotesize$2$};
    \draw (0,0)--(1,0)--(2,0);    
    \end{tikzpicture}}}
    \vcenter{\hbox{\begin{tikzpicture}[
        scale=0.7,
        vert/.style={circle,  draw=black!30!black, thick, minimum size=1mm},
        leaf/.style={rectangle, thick, minimum size=1mm},
        edge/.style={-,black!30!black, thick},
        ]
        \node[vert] (1) at (0,1) {\footnotesize$m$};
        \node[leaf] (l1) at (-0.75,2) {\footnotesize$1$};
        \node[vert] (2) at (0.75,2) {\footnotesize$m$};
        \node[leaf] (l2) at (0,3) {\footnotesize$2$};
        \node[leaf] (3) at (1.5,3) {\footnotesize$3$};
        \draw[edge] (0,0)--(1);
        \draw[edge] (1)--(2)--(3);
        \draw[edge] (1)--(l1);
        \draw[edge] (2)--(l2);
    \end{tikzpicture}}}
    \quad
    \vcenter{\hbox{\begin{tikzpicture}[scale=0.5]
    \fill (0,0) circle (2pt);
    \node at (0,0.5) {\footnotesize$2$};
    \fill (1,0) circle (2pt);
    \node at (1,0.5) {\footnotesize$1$};
    \fill (2,0) circle (2pt);
    \node at (2,0.5) {\footnotesize$3$};
    \draw (0,0)--(1,0)--(2,0);    
    \end{tikzpicture}}}
    \vcenter{\hbox{\begin{tikzpicture}[
        scale=0.7,
        vert/.style={circle,  draw=black!30!black, thick, minimum size=1mm},
        leaf/.style={rectangle, thick, minimum size=1mm},
        edge/.style={-,black!30!black, thick},
        ]
        \node[vert] (1) at (0,1) {\footnotesize$m$};
        \node[leaf] (l1) at (0.75,2) {\footnotesize$2$};
        \node[vert] (2) at (-0.75,2) {\footnotesize$m$};
        \node[leaf] (l2) at (0,3) {\footnotesize$3$};
        \node[leaf] (3) at (-1.5,3) {\footnotesize$1$};
        \draw[edge] (0,0)--(1);
        \draw[edge] (1)--(2)--(3);
        \draw[edge] (1)--(l1);
        \draw[edge] (2)--(l2);
    \end{tikzpicture}}}
    \\
\vcenter{\hbox{\begin{tikzpicture}[scale=0.6]
    \fill (-0.5,0) circle (2pt);
    \node at (-0.7,-0.2) {\footnotesize$1$};
    \fill (0.5,0) circle (2pt);
    \node at (0.7,-0.2) {\footnotesize$3$};
    \fill (0,0.86) circle (2pt);
    \node at (0,1.14) {\footnotesize$2$};
    \draw (-0.5,0)--(0,0.86)--(0.5,0)--cycle;    
    \end{tikzpicture}}}
    \vcenter{\hbox{\begin{tikzpicture}[
        scale=0.7,
        vert/.style={circle,  draw=black!30!black, thick, minimum size=1mm},
        leaf/.style={rectangle, thick, minimum size=1mm},
        edge/.style={-,black!30!black, thick},
        ]
        \node[vert] (1) at (0,1) {\footnotesize$m$};
        \node[leaf] (l1) at (0.75,2) {\footnotesize$2$};
        \node[vert] (2) at (-0.75,2) {\footnotesize$m$};
        \node[leaf] (l2) at (0,3) {\footnotesize$3$};
        \node[leaf] (3) at (-1.5,3) {\footnotesize$1$};
        \draw[edge] (0,0)--(1);
        \draw[edge] (1)--(2)--(3);
        \draw[edge] (1)--(l1);
        \draw[edge] (2)--(l2);
    \end{tikzpicture}}}
    \quad
    \vcenter{\hbox{\begin{tikzpicture}[
        scale=0.7,
        vert/.style={circle,  draw=black!30!black, thick, minimum size=1mm},
        leaf/.style={rectangle, thick, minimum size=1mm},
        edge/.style={-,black!30!black, thick},
        ]
        \node[vert] (1) at (0,1) {\footnotesize$m$};
        \node[leaf] (l1) at (-0.75,2) {\footnotesize$1$};
        \node[vert] (2) at (0.75,2) {\footnotesize$m$};
        \node[leaf] (l2) at (0,3) {\footnotesize$2$};
        \node[leaf] (3) at (1.5,3) {\footnotesize$3$};
        \draw[edge] (0,0)--(1);
        \draw[edge] (1)--(2)--(3);
        \draw[edge] (1)--(l1);
        \draw[edge] (2)--(l2);
    \end{tikzpicture}}}
\end{gather*} 
Let us describe the normal monomials with respect to quadratic relations. The form of leading terms ensures us that the underlying trees of $\Gr$-input normal monomials have the form 
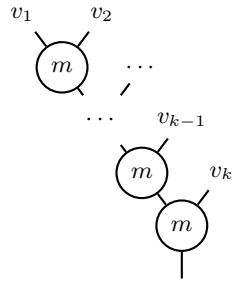
\[
\vcenter{\hbox{\begin{tikzpicture}[
        scale=0.7,
        vert/.style={circle,  draw=black!30!black, thick, minimum size=1mm},
        emptyvert/.style={circle, thick, minimum size=1mm},
        leaf/.style={rectangle, thick, minimum size=1mm},
        edge/.style={-,black!30!black, thick},
        ]
        \node[vert] (k) at (0.75,0) {\footnotesize$m$};
        \node[leaf] (lk) at (1.5,1) {\footnotesize$v_k$};
        \node[vert] (1) at (0,1) {\footnotesize$m$};
        \node[leaf] (l1) at (0.75,2) {\footnotesize$v_{k-1}$};
        \node[emptyvert] (2) at (-0.75,2) {\footnotesize$\cdots$};
        \node[leaf] (l2) at (0,3) {\footnotesize$\cdots$};
        \node[vert] (3) at (-1.5,3) {\footnotesize$m$};
        \node[leaf] (l3) at (-0.75,4) {\footnotesize$v_2$};
        \node[leaf] (l4) at (-2.25,4) {\footnotesize$v_1$};
        \draw[edge] (0.75,-1)--(k);
        \draw[edge] (k)--(1)--(2)--(3);
        \draw[edge] (k)--(lk);
        \draw[edge] (1)--(l1);
        \draw[edge] (2)--(l2);
        \draw[edge] (3)--(l3);
        \draw[edge] (3)--(l4);
    \end{tikzpicture}}}
\] with an additional condition on the labeling of leaves: if vertex $v_{i+1}$ is adjacent to tube $\{v_1,\cdots, v_{i-1}\}$, then $v_{i}<v_{i+1}$. Note that there is only one labeling of this type defined as follows: $v_1$ is the minimal vertex of $\Gr$, $v_2$ is the minimal vertex adjacent to $\{v_1\}$, $\cdots$, $v_i$ is the minimal vertex adjacent to tube $\{v_1,v_2,\cdots,v_{i-1}\}$. By Proposition~\ref{coset}, these monomials form a spanning set of $\Pop^{\forget}$. Hence, by dimension reasons, we conclude that the morphism of shuffle contractads $\Pop^{\forget}\to\Com^{\forget}$ is an isomorphism. Moreover, the set of quadratic relations $\R^{\forget}$ forms the Gr\"obner basis of the ideal $\langle \R^{\forget}\rangle$.
\end{proof}

The following result illustrates a connection between Gr\"obner bases of quadratic contractads and Koszul property. See the proof in Appendix~\ref{subsec:pbw}.

\begin{theorem}\label{quadgrob}
\begin{itemize}
    \item A shuffle contractad $\Pop$ which admits a quadratic Gr\"obner basis for some monomial order is Koszul.
    \item A shuffle contractad $\Pop$ has a quadratic Gr\"obner basis for some monomial order if and only if its Koszul dual $\Pop^{!}$ admits a quadratic Gr\"obner basis for the dual monomial order.
\end{itemize}
\end{theorem}

During the proof of Proposition~\ref{quadcom}, we showed that shuffle contractad $\Com^{\forget}$ has a  quadratic Gr\"obner basis with respect to reverse $\grpermlex$-order. By Theorem~\ref{quadgrob}, we get an alternative proof of Koszulity of $\Com$. Another consequence is related to the Koszul dual contractad $\Lie$.

\begin{cor}
\label{liegrob}
The shuffle contractad $\Lie^{\forget}$ has a quadratic Gr\"obner basis with respect to $\grpermlex$-order.
\end{cor}
Let us discuss the monomial basis of this contractad. With respect to the $\grpermlex$-order, the leading terms of relations $\R^{\forget}_{\Lie}$ have the form
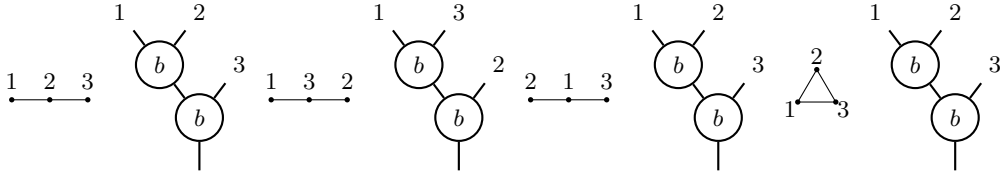
\[\label{leadtermlie}
    \vcenter{\hbox{\begin{tikzpicture}[scale=0.5]
    \fill (0,0) circle (2pt);
    \node at (0,0.5) {\footnotesize$1$};
    \fill (1,0) circle (2pt);
    \node at (1,0.5) {\footnotesize$2$};
    \fill (2,0) circle (2pt);
    \node at (2,0.5) {\footnotesize$3$};
    \draw (0,0)--(1,0)--(2,0);    
    \end{tikzpicture}}}
    \vcenter{\hbox{\begin{tikzpicture}[
        scale=0.7,
        vert/.style={circle,  draw=black!30!black, thick, minimum size=1mm},
        leaf/.style={rectangle, thick, minimum size=1mm},
        edge/.style={-,black!30!black, thick},
        ]
        \node[vert] (1) at (0,1) {\footnotesize$b$};
        \node[leaf] (l1) at (0.75,2) {\footnotesize$3$};
        \node[vert] (2) at (-0.75,2) {\footnotesize$b$};
        \node[leaf] (l2) at (0,3) {\footnotesize$2$};
        \node[leaf] (3) at (-1.5,3) {\footnotesize$1$};
        \draw[edge] (0,0)--(1);
        \draw[edge] (1)--(2)--(3);
        \draw[edge] (1)--(l1);
        \draw[edge] (2)--(l2);
    \end{tikzpicture}}}
    \vcenter{\hbox{\begin{tikzpicture}[scale=0.5]
    \fill (0,0) circle (2pt);
    \node at (0,0.5) {\footnotesize$1$};
    \fill (1,0) circle (2pt);
    \node at (1,0.5) {\footnotesize$3$};
    \fill (2,0) circle (2pt);
    \node at (2,0.5) {\footnotesize$2$};
    \draw (0,0)--(1,0)--(2,0);    
    \end{tikzpicture}}}
    \vcenter{\hbox{\begin{tikzpicture}[
        scale=0.7,
        vert/.style={circle,  draw=black!30!black, thick, minimum size=1mm},
        leaf/.style={rectangle, thick, minimum size=1mm},
        edge/.style={-,black!30!black, thick},
        ]
        \node[vert] (1) at (0,1) {\footnotesize$b$};
        \node[leaf] (l1) at (0.75,2) {\footnotesize$2$};
        \node[vert] (2) at (-0.75,2) {\footnotesize$b$};
        \node[leaf] (l2) at (0,3) {\footnotesize$3$};
        \node[leaf] (3) at (-1.5,3) {\footnotesize$1$};
        \draw[edge] (0,0)--(1);
        \draw[edge] (1)--(2)--(3);
        \draw[edge] (1)--(l1);
        \draw[edge] (2)--(l2);
    \end{tikzpicture}}}
    \vcenter{\hbox{\begin{tikzpicture}[scale=0.5]
    \fill (0,0) circle (2pt);
    \node at (0,0.5) {\footnotesize$2$};
    \fill (1,0) circle (2pt);
    \node at (1,0.5) {\footnotesize$1$};
    \fill (2,0) circle (2pt);
    \node at (2,0.5) {\footnotesize$3$};
    \draw (0,0)--(1,0)--(2,0);    
    \end{tikzpicture}}}
    \vcenter{\hbox{\begin{tikzpicture}[
        scale=0.7,
        vert/.style={circle,  draw=black!30!black, thick, minimum size=1mm},
        leaf/.style={rectangle, thick, minimum size=1mm},
        edge/.style={-,black!30!black, thick},
        ]
        \node[vert] (1) at (0,1) {\footnotesize$b$};
        \node[leaf] (l1) at (0.75,2) {\footnotesize$3$};
        \node[vert] (2) at (-0.75,2) {\footnotesize$b$};
        \node[leaf] (l2) at (0,3) {\footnotesize$2$};
        \node[leaf] (3) at (-1.5,3) {\footnotesize$1$};
        \draw[edge] (0,0)--(1);
        \draw[edge] (1)--(2)--(3);
        \draw[edge] (1)--(l1);
        \draw[edge] (2)--(l2);
    \end{tikzpicture}}}
    \vcenter{\hbox{\begin{tikzpicture}[scale=0.5]
    \fill (-0.5,0) circle (2pt);
    \node at (-0.7,-0.2) {\footnotesize$1$};
    \fill (0.5,0) circle (2pt);
    \node at (0.7,-0.2) {\footnotesize$3$};
    \fill (0,0.86) circle (2pt);
    \node at (0,1.2) {\footnotesize$2$};
    \draw (-0.5,0)--(0,0.86)--(0.5,0)--cycle;    
    \end{tikzpicture}}}
    \vcenter{\hbox{\begin{tikzpicture}[
        scale=0.7,
        vert/.style={circle,  draw=black!30!black, thick, minimum size=1mm},
        leaf/.style={rectangle, thick, minimum size=1mm},
        edge/.style={-,black!30!black, thick},
        ]
        \node[vert] (1) at (0,1) {\footnotesize$b$};
        \node[leaf] (l1) at (0.75,2) {\footnotesize$3$};
        \node[vert] (2) at (-0.75,2) {\footnotesize$b$};
        \node[leaf] (l2) at (0,3) {\footnotesize$2$};
        \node[leaf] (3) at (-1.5,3) {\footnotesize$1$};
        \draw[edge] (0,0)--(1);
        \draw[edge] (1)--(2)--(3);
        \draw[edge] (1)--(l1);
        \draw[edge] (2)--(l2);
    \end{tikzpicture}}}
\]
By direct inspection, we  conclude that a tree monomial $T$ is normal if, for each  subtree of the form $\vcenter
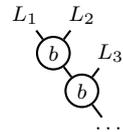
{\hbox{\begin{tikzpicture}[
        scale=0.5,
        vert/.style={inner sep=2pt, circle,draw, thick},
        leaf/.style={inner sep=2pt, rectangle, thick},
        edge/.style={-,black!30!black, thick},
        ]
        \node[vert] (1) at (0,1) {\scriptsize$b$};
        \node[leaf] (l1) at (0.75,2) {\scriptsize$L_3$};
        \node[vert] (2) at (-0.75,2) {\scriptsize$b$};
        \node[leaf] (l2) at (0,3) {\scriptsize$L_2$};
        \node[leaf] (3) at (-1.5,3) {\scriptsize$L_1$};
        \node[leaf] (dd) at (0.75,0) {\scriptsize$\cdots$};
        \draw[edge] (dd)--(1);
        \draw[edge] (1)--(2)--(3);
        \draw[edge] (1)--(l1);
        \draw[edge] (2)--(l2);
    \end{tikzpicture}}}$, the union $L_1 \cup L_3$ is a tube and $\min L_2 >\min L_3$. By Proposition~\ref{coset}, such monomials form the basis of $\Lie^{\forget}$. We construct an alternative basis of this contractad in Appendix~\ref{subsec:diamond}.

\begin{figure}[ht]
    \centering
    \caption{List of $\Lie$-monomials for classical ordering of the cycle $\Cyc_4$}
    \begin{gather*}
    \vcenter{\hbox{\begin{tikzpicture}[scale=0.5]
    \fill (0,0) circle (2pt);
    \fill (0,1) circle (2pt);
    \fill (1,0) circle (2pt);
    \fill (1,1) circle (2pt);
    \draw (0,0)--(1,0)--(1,1)--(0,1)-- cycle;
    \node at (-0.25,1.25) {\footnotesize$1$};
    \node at (1.25,1.25) {\footnotesize$2$};
    \node at (1.25,-0.25) {\footnotesize$3$};
    \node at (-0.25,-0.25) {\footnotesize$4$};
  \end{tikzpicture}}}
  \quad
  \vcenter{\hbox{\begin{tikzpicture}[
        scale=0.7,
        vert/.style={circle,  draw=black!30!black, thick, minimum size=1mm},
        leaf/.style={rectangle, thick, minimum size=1mm},
        edge/.style={-,black!30!black, thick},
        ]
        \node[vert] (1) at (0,1) {\footnotesize$b$};
        \node[leaf] (l1) at (-0.75,2) {\footnotesize$1$};
        \node[vert] (2) at (0.75,2) {\footnotesize$b$};
        \node[leaf] (l2) at (0,3) {\footnotesize$2$};
        \node[vert] (3) at (1.5,3) {\footnotesize$b$};
        \node[leaf] (l3) at (0.75,4) {\footnotesize$3$};
        \node[leaf] (l4) at (2.25,4) {\footnotesize$4$};
        \draw[edge] (0,0)--(1);
        \draw[edge] (1)--(2)--(3);
        \draw[edge] (1)--(l1);
        \draw[edge] (2)--(l2);
        \draw[edge] (3)--(l3);
        \draw[edge] (3)--(l4);
    \end{tikzpicture}}}
    \quad
    \vcenter{\hbox{\begin{tikzpicture}[
        scale=0.7,
        vert/.style={circle,  draw=black!30!black, thick, minimum size=1mm},
        leaf/.style={rectangle, thick, minimum size=1mm},
        edge/.style={-,black!30!black, thick},
        ]
        \node[leaf] (l1) at (1.65,3) {\footnotesize$3$};
        \node[leaf] (l2) at (0.35,3) {\footnotesize$2$};
        \node[vert] (1) at (0,1) {\footnotesize$b$};
        \node[vert] (2) at (1,2) {\footnotesize$b$};
        \node[vert] (3) at (-1,2) {\footnotesize$b$};
        \node[leaf] (r1) at (-1.65,3) {\footnotesize$1$};
        \node[leaf] (r3) at (-0.35,3) {\footnotesize$4$};
        \draw[edge] (0,0)--(1);
        \draw[edge] (1)--(2);
        \draw[edge] (1)--(3);
        \draw[edge] (3)--(r1);
        \draw[edge] (3)--(r3);
        \draw[edge] (2)--(l1);
        \draw[edge] (2)--(l2);
    \end{tikzpicture}}}
    \quad
    \vcenter{\hbox{\begin{tikzpicture}[
        scale=0.7,
        vert/.style={circle,  draw=black!30!black, thick, minimum size=1mm},
        leaf/.style={rectangle, thick, minimum size=1mm},
        edge/.style={-,black!30!black, thick},
        ]
        \node[vert] (1) at (0,1) {\footnotesize$b$};
        \node[leaf] (l1) at (0.75,2) {\footnotesize$2$};
        \node[vert] (2) at (-0.75,2) {\footnotesize$b$};
        \node[leaf] (l2) at (-1.5,3) {\footnotesize$1$};
        \node[vert] (3) at (0,3) {\footnotesize$b$};
        \node[leaf] (l3) at (-0.75,4) {\footnotesize$3$};
        \node[leaf] (l4) at (0.75,4) {\footnotesize$4$};
        \draw[edge] (0,0)--(1);
        \draw[edge] (1)--(2)--(3);
        \draw[edge] (1)--(l1);
        \draw[edge] (2)--(l2);
        \draw[edge] (3)--(l3);
        \draw[edge] (3)--(l4);
    \end{tikzpicture}}}
\end{gather*}
\label{liemon}
\end{figure}
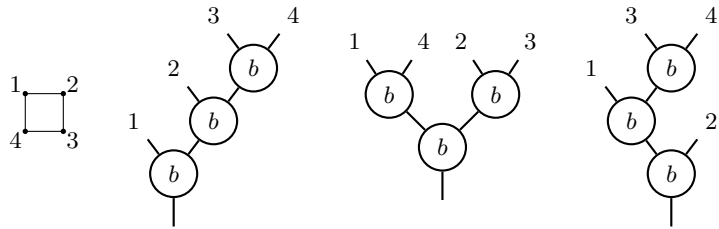

\section{(Co)Homology of the little disks contractads}\label{sec:discs}
In this section, we describe explicitly the homology contractads of the little disks contractads defined in Section~\ref{subsec:disks}. The results of this section generalise the ones in \cite{getzler1994operads}.
\subsection{One-dimensional case}
In this subsection, we describe the homology contractad of the little intervals contractad $\D_1$. This contractad is a graphical counterpart of the operad $\mathsf{Ass}$ of associative algebras. Moreover, we prove that this contractad is Koszul.

Consider the little intervals contractad $\D_1$. By Proposition~\ref{discsconf},  each component of this contractad $\D_1(\Gr) \simeq \Conf_{\Gr}(\mathbb{R}):=\mathbb{R}^n\setminus \B(\Gr)$ is homotopy equivalent to the complement of the real graphic arrangement $\B(\Gr)=\{\{x_v=x_w\}| (v,w)\in E_{\Gr}\}$.

\begin{prop}
\label{dimD1}
The homology groups of little intervals contractad $H_{\bullet}(\D_1)$ are concentrated in degree zero. Moreover, the dimension of each component is given by the formula
\[
\dim H_{0}(\D_1)=\sum_{I \in \parti(\Gr)} \prod_{G \in I}|\mu(\Gr|_G)|,
\] where $\mu(\Gr):=\mu_{\parti(\Gr)}(\hat{0},\hat{1})$.
\end{prop}
\begin{proof}
The  first assertion of the proposition follows from the general fact that connected components of a complement to a real hyperplane arrangement are contractible. Moreover, the number of connected components is uniquely determined by the M\"obius function of the related intersection lattice \cite[Th.~3.11.7]{stanley2011enumerative}
\[
|\pi_0(\mathbb{R}^n\setminus \B(\Gr))|=\sum_{I \in \La(\B(\Gr))} |\mu(\hat{0},I)|=\sum_{I \in \parti(\Gr)} \prod_{G \in I}|\mu(\Gr|_G)|.
\]
\end{proof}
For path $\Path_2$, the space $\D_1(\Path_2)$ consists of two connected components $\{z_1>z_2\}$, $\{z_1<z_2\}$, where $z_1,z_2$ are centers of labeled discs. Let $\nu \in H_0(\D_2(\Path_2))$ be a class of the point belonging to the first connected component. The second generator $\nu^{(12)}$ of this group is obtained from $\nu$ by applying the unique automorphism of the path $\Path_2$.
\begin{lemma}
\label{relD1}
We have
\begin{gather}
    \nu\circ^{\Path_3}_{\{1,2\}}\nu=\nu\circ^{\Path_3}_{\{2,3\}}\nu,
    \\
    \nu^{(12)}\circ^{\Path_3}_{\{1,2\}}\nu=\nu \circ^{\Path_3}_{\{2,3\}}\nu^{(12)},\label{relD12}
    \\
    \nu\circ^{\Path_3}_{\{1,2\}}\nu^{(12)}=\nu^{(12)} \circ^{\Path_3}_{\{2,3\}}\nu,\label{relD22}
    \\
    \nu\circ^{\K_3}_{\{1,2\}}\nu=\nu\circ^{\K_3}_{\{2,3\}}\nu.
\end{gather} 
\end{lemma}
\begin{proof}
Let us examine the second relation. For the path $\Path_3$, the space $\D_1(\Path_3)$ consists of 4 connected components: (\textrm{I}) $\{z_1>z_2>z_3\}$, (\textrm{II}) $\{z_1<z_2>z_3\}$, (\textrm{III}) $\{z_1>z_2<z_3\}$, (\textrm{IV}) $\{z_1<z_2<z_3\}$. We see that representatives of $\nu^{(12)}\circ^{\Path_3}_{\{1,2\}}\nu$, and $\nu \circ^{\Path_3}_{\{2,3\}}\nu^{(12)}$ belong to the same component, specifically to (\textrm{II}). Hence, we have $\nu^{(12)}\circ^{\Path_3}_{\{1,2\}}\nu=\nu \circ^{\Path_3}_{\{2,3\}}\nu^{(12)}$. The other relations are examined in a similar way.
\begin{gather*}
\vcenter{\hbox{\begin{tikzpicture}[
        scale=0.5,
        ]
        \draw (0,0)--(6,0);
        \node (bl) at (0,0) {[};
        \node (i1) at (1.5,0.5) {$3$};
        \node (i1l) at (1,0) {[};
        \node (i1r) at (2,0) {]};
        \node (i2) at (4,0.5) {$\{1,2\}$};
        \node (i2l) at (3,0) {[};
        \node (i2r) at (5,0) {]};
        \node (br) at (6,0) {]};
    \end{tikzpicture}}} \circ^{\Path_3}_{\{1,2\}}
    \vcenter{\hbox{\begin{tikzpicture}[
        scale=0.5,
        ]
        \draw (0,0)--(5,0);
        \node (bl) at (0,0) {[};
        \node (i1) at (1.5,0.5) {$1$};
        \node (i1l) at (1,0) {[};
        \node (i1r) at (2,0) {]};
        \node (i2) at (3.5,0.5) {$2$};
        \node (i2l) at (3,0) {[};
        \node (i2r) at (4,0) {]};
        \node (br) at (5,0) {]};
    \end{tikzpicture}}}=
    \vcenter{\hbox{\begin{tikzpicture}[
        scale=0.5,
        ]
        \draw (0,0)--(7,0);
        \node (bl) at (0,0) {[};
        \node (i1) at (1.5,0.5) {$3$};
        \node (i1l) at (1,0) {[};
        \node (i1r) at (2,0) {]};
        \node (i2) at (3.5,0.5) {$1$};
        \node (i2l) at (3,0) {[};
        \node (i2r) at (4,0) {]};
        \node (i3) at (5.5,0.5) {$2$};
        \node (i3l) at (5,0) {[};
        \node (i3r) at (6,0) {]};
        \node (br) at (7,0) {]};
    \end{tikzpicture}}}
    \\
    \vcenter{\hbox{\begin{tikzpicture}[
        scale=0.5,
        ]
        \draw (0,0)--(6,0);
        \node (bl) at (0,0) {[};
        \node (i1) at (1.5,0.5) {$1$};
        \node (i1l) at (1,0) {[};
        \node (i1r) at (2,0) {]};
        \node (i2) at (4,0.5) {$\{2,3\}$};
        \node (i2l) at (3,0) {[};
        \node (i2r) at (5,0) {]};
        \node (br) at (6,0) {]};
    \end{tikzpicture}}} \circ^{\Path_3}_{\{2,3\}}
    \vcenter{\hbox{\begin{tikzpicture}[
        scale=0.5,
        ]
        \draw (0,0)--(5,0);
        \node (bl) at (0,0) {[};
        \node (i1) at (1.5,0.5) {$3$};
        \node (i1l) at (1,0) {[};
        \node (i1r) at (2,0) {]};
        \node (i2) at (3.5,0.5) {$2$};
        \node (i2l) at (3,0) {[};
        \node (i2r) at (4,0) {]};
        \node (br) at (5,0) {]};
    \end{tikzpicture}}}=
    \vcenter{\hbox{\begin{tikzpicture}[
        scale=0.5,
        ]
        \draw (0,0)--(7,0);
        \node (bl) at (0,0) {[};
        \node (i1) at (1.5,0.5) {$1$};
        \node (i1l) at (1,0) {[};
        \node (i1r) at (2,0) {]};
        \node (i2) at (3.5,0.5) {$3$};
        \node (i2l) at (3,0) {[};
        \node (i2r) at (4,0) {]};
        \node (i3) at (5.5,0.5) {$2$};
        \node (i3l) at (5,0) {[};
        \node (i3r) at (6,0) {]};
        \node (br) at (7,0) {]};
    \end{tikzpicture}}}
\end{gather*} 
Note that relations \eqref{relD12}, \eqref{relD22} do not hold in the complete graph case because the space $\D_1(\K_3)$ has more connected components than $\D_1(\Path_3)$.
\end{proof}

Note that the contractad $H_{0}(\D_1)$ can be described in a purely combinatorial way as follows. For a graph $\Gr$ on $n$ vertices, let $\mathsf{Seq}(V_{\Gr})$ be a set consisting of ordered $n$-tuples $(v_1,v_2,\cdots,v_n)$ of distinct vertices. We say that two tuples $\sigma$ and $\sigma'$ are equivalent if $\sigma'$ is obtained from $\sigma$ by a sequence of permutations of the form $(\cdots,v_i,v_{i+1},\cdots)\rightsquigarrow(\cdots,v_{i+1},v_i,\cdots)$, where vertices $v_i,v_j$ are not adjacent. Let us define the set of $\Gr$-tuples $\mathsf{Seq}(\Gr)$ as the set of class equivalences $\mathsf{Seq}(V_{\Gr})/\sim$. For example, for the path $\Path_3$, we have 4 equivalence classes $\mathsf{Seq}(\Path_3)=\{[(1,2,3)],[(1,3,2)],[(2,1,3)],[(3,2,1)]\}$, where classes $[(1,3,2)],[(2,1,3)]$ consist of two elements, while $[(1,2,3)],[(3,2,1)]$ of only one.
\begin{lemma}
For a graph $\Gr$, there is a one-to-one correspondence between connected components of $\Conf_{\Gr}(\mathbb{R})$ and $\Gr$-tuples
\[
\mathsf{Seq}(\Gr)\cong \pi_0(\Conf_{\Gr}(\mathbb{R})).
\]
\end{lemma}
\begin{proof}
We correspond a $\Gr$-tuple to a configuration $p=(x_v)_{v\in V_{\Gr}}$ from $\Conf_{\Gr}(\mathbb{R})$  by arranging the labels of marked points of $p$ from left to right. It should be noted that this procedure is well-defined for configurations where part of the marked points coincide because all possible tuples associated with this configuration are equivalent. For example, for $\Path_3$ we have
\begin{gather*}
    \vcenter{\hbox{\begin{tikzpicture}[
        scale=0.5,
        ]
        \draw (0,0)--(4,0);
        \fill (0.6,0) circle (2pt);
        \node (i1) at (1,0.5) {$1$};
        \fill (2.5,0) circle (2pt);
        \node (i2) at (2.5,0.5) {$3$};
        \fill (3.2,0) circle (2pt);
        \node (i3) at (3.2,0.5) {$2$};
    \end{tikzpicture}}}\mapsto [(1,3,2)]\quad
    \vcenter{\hbox{\begin{tikzpicture}[
        scale=0.5,
        ]
        \draw (0,0)--(4,0);
        \fill (1,0) circle (2pt);
        \node (i1) at (1,0.5) {$2$};
        \fill (2.5,0) circle (2pt);
        \node (i3) at (2.5,0.42) {$1,3$};
    \end{tikzpicture}}}\mapsto [(2,1,3)]=[(2,3,1)]
\end{gather*} During the construction of the space $\Conf_{\Gr}(\mathbb{R})=\mathbb{R}^n\setminus \cup_{(v,w)\in E_{\Gr}} \{x_v=x_w\}$, we remove only hyperplanes labeled by edges, therefore two points $p,p'\in\Conf_{\Gr}(\mathbb{R})$ belong to the same connected component if and only if the associated tuples are equivalent. Hence, we obtain a well-defined map $\pi_0(\Conf_{\Gr}(\mathbb{R}))\to \mathsf{Seq}(\Gr)$ that is bijective by construction.
\end{proof}
 From the lemma above, we see that the graphical collection $H_{0}(\D_1)$ is obtained from the graphical collection of graph tuples $\mathsf{Seq}$ by linearization. In terms of tuples, the contractad structure on  $H_{\bullet}(\D_1)$ is given by the substitution of graph tuples. For example, the relation~\eqref{relD12} is rewritten in the following way
 \[
[(3,\{1,2\})]\circ^{\Path_3}_{\{1,2\}}[(1,2)]=[(3,1,2)]=[(1,3,2)]=[(1,\{2,3\})]\circ^{\Path_3}_{\{2,3\}}[(3,2)].
\] Also, in the combinatorial description we see that element $\nu=[(1,2)]$ generates $H_{0}(\D_1)$ as a contractad. Indeed, it follows from the observation that each tuple can be obtained by a sequence of substitutions.

Recall that the restriction of the contractad $\K_{*}(H_0(\D_2))$ to complete graphs is isomorphic to the symmetric  operad $\mathsf{Ass}$ of  associative algebras. We shall introduce a graphical counterpart of this operad.
\begin{defi}
The Associative contractad  $\Ass$ is the contractad with a generator $\nu$ in the component $\Path_2$, satisfying the relations
\begin{gather*}
    \nu\circ^{\Path_3}_{\{1,2\}}\nu=\nu\circ^{\Path_3}_{\{2,3\}}\nu,
    \\
    \nu^{(12)}\circ^{\Path_3}_{\{1,2\}}\nu=\nu \circ^{\Path_3}_{\{2,3\}}\nu^{(12)},
    \\
    \nu\circ^{\Path_3}_{\{1,2\}}\nu^{(12)}=\nu^{(12)} \circ^{\Path_3}_{\{2,3\}}\nu,
    \\
    \nu\circ^{\K_3}_{\{1,2\}}\nu=\nu\circ^{\K_3}_{\{2,3\}}\nu.
\end{gather*}
\end{defi}

By Lemma~\ref{relD1}, we have a well-defined morphism of contractads $\Ass \rightarrow H_{0}(\D_1)$. We want to show that this morphism is precisely an isomorphism. We must introduce another presentation of $\Ass$ in order to see this. Let $m:=\nu+\nu^{(12)}$ and $b:=\nu-\nu^{(12)}$ be the "odd" and "even" part of $\nu$. In the operad case, these generators correspond to Jordan and Lie brackets, respectively. By direct computations, we see that the relations in new generators have the following form
\begin{gather*}
    m\circ^{\Path_3}_{\{1,2\}}m-m\circ^{\Path_3}_{\{2,3\}}m=0,
    \\
    b\circ^{\Path_3}_{\{1,2\}}b-b\circ^{\Path_3}_{\{2,3\}}b=0,
    \\
    m\circ^{\Path_3}_{\{1,2\}}b-b\circ^{\Path_3}_{\{2,3\}}m=0,
    \\
    m\circ^{\K_3}_{\{1,2\}}m-m\circ^{\K_3}_{\{2,3\}}m+b\circ^{\K_3}_{\{1,3\}}b=0,
    \\
    b\circ^{\K_3}_{\{1,2\}}b+(b\circ^{\K_3}_{\{1,2\}}b)^{(123)}+(b\circ^{\K_3}_{\{1,2\}}b)^{(132)}=0,
    \\
    b\circ^{\K_3}_{\{1,2\}}m- m\circ^{\K_3}_{\{2,3\}}b-m\circ^{\K_3}_{\{1,3\}}b=0.
\end{gather*}
\begin{lemma}
\label{quantumorder}
    There is a monomial order on tree monomials $\Tree_{\{ m,b\}}$, such that the leading terms of quadratic relations $\R_{\Ass}$ have the form
\[
    \LT(\R_{\Ass})=\LT(\R_{\Com})\cup \LT(\R_{\Lie}) \cup \{b \circ^{\Gr}_{e} m| e \in E_{\Gr}\},
\] where $\LT(\R_{\Com}),  \LT(\R_{\Lie})$ are the leading terms of Gr\"obner basis for the commutative \eqref{leadtermcom} and  Lie contractad \eqref{leadtermlie}, respectively.
\end{lemma}
\begin{proof}
Consider the monoid of quantum monomials $\mathsf{QM}=\langle m,b,q |mq-qm,bq-qb,bm-mbq\rangle$. Each element of this monoid can be written in the canonical form $m^kb^lq^m$. Consider the order on these monomials by putting $m^kb^lq^m\prec m^{k'}b^{l'}q^{m'}$ if $k>k'$ or $k=k'$ and $l<l'$, or $k=k'$ and $l=l'$, and $m<m'$. It was proved in \cite[Th.~2.2]{dotsenko2020word} that this order turns $\mathsf{QM}$ into an ordered monoid, i.e., this order is compatible with a monoid structure. Define the modified $\grpermlex$ order on tree monomials in $\Tree_{\{ m, b \}}$ by the rule:
\begin{itemize}
    \item For a pair of monomials $T', T$, we set $T' > T$ if the number of vertices labeled $m$ from
    the left monomial is strictly greater than the one from the right
    \item If such numbers are equal, we compare $T,T'$ using the $\grpermlex$ extension of quantum monomial order on $\{m,b,q\}^*$.
\end{itemize}
\[
\vcenter{\hbox{\begin{tikzpicture}[
        scale=0.6,
        vert/.style={inner sep=2.5pt, circle,  draw, thick},
        leaf/.style={inner sep=2pt, rectangle},
        edge/.style={-,black!30!black, thick},
        ]
        \node[vert] (1) at (0,1) {\small $b$};
        \node[leaf] (l1) at (0.75,2) {\small$3$};
        \node[vert] (2) at (-0.75,2) {\small$b$};
        \node[leaf] (l2) at (0,3) {\small$2$};
        \node[leaf] (3) at (-1.5,3) {\small$1$};
        \node[leaf] (root) at (0,-0.5) {\scriptsize$(0,(bb,
        bb,b),123)$};
        \draw[edge] (root)--(1);
        \draw[edge] (1)--(2)--(3);
        \draw[edge] (1)--(l1);
        \draw[edge] (2)--(l2);
    \end{tikzpicture}}}<
\vcenter{\hbox{\begin{tikzpicture}[
        scale=0.6,
        vert/.style={inner sep=2.5pt, circle,  draw, thick},
        leaf/.style={inner sep=2pt, rectangle},
        edge/.style={-,black!30!black, thick},
        ]
        \node[vert] (1) at (0,1) {\small$m$};
        \node[leaf] (l1) at (0.75,2) {\small$3$};
        \node[vert] (2) at (-0.75,2) {\small$b$};
        \node[leaf] (l2) at (0,3) {\small$2$};
        \node[leaf] (3) at (-1.5,3) {\small$1$};
        \node[leaf] (root) at (0,-0.5) {\scriptsize$(1,(mb,mb,m),123)$};
        \draw[edge] (root)--(1);
        \draw[edge] (1)--(2)--(3);
        \draw[edge] (1)--(l1);
        \draw[edge] (2)--(l2);
    \end{tikzpicture}}}<
    \vcenter{\hbox{\begin{tikzpicture}[
        scale=0.6,
        vert/.style={inner sep=2.5pt, circle,  draw, thick},
        leaf/.style={inner sep=2pt, rectangle},
        edge/.style={-,black!30!black, thick},
        ]
        \node[vert] (1) at (0,1) {\small$b$};
        \node[leaf] (l1) at (0.75,2) {\small$3$};
        \node[vert] (2) at (-0.75,2) {\small$m$};
        \node[leaf] (l2) at (0,3) {\small$2$};
        \node[leaf] (3) at (-1.5,3) {\small$1$};
        \node[leaf] (root) at (0,-0.5) {\scriptsize$(1,(mbq,mbq,b),123)$};
        \draw[edge] (root)--(1);
        \draw[edge] (1)--(2)--(3);
        \draw[edge] (1)--(l1);
        \draw[edge] (2)--(l2);
    \end{tikzpicture}}}<
    \vcenter{\hbox{\begin{tikzpicture}[
        scale=0.6,
        vert/.style={inner sep=2.5pt, circle,  draw, thick},
        leaf/.style={inner sep=2pt, rectangle},
        edge/.style={-,black!30!black, thick},
        ]
        \node[vert] (1) at (0,1) {\small$m$};
        \node[leaf] (l1) at (0.75,2) {\small$3$};
        \node[vert] (2) at (-0.75,2) {\small$m$};
        \node[leaf] (l2) at (0,3) {\small$2$};
        \node[leaf] (3) at (-1.5,3) {\small$1$};
        \node[leaf] (root) at (0,-0.5) {\scriptsize$(2,(mm,mm,m),123)$};
        \draw[edge] (root)--(1);
        \draw[edge] (1)--(2)--(3);
        \draw[edge] (1)--(l1);
        \draw[edge] (2)--(l2);
    \end{tikzpicture}}}
\] By direct computations, we see that the leading terms of relations $\R^{\forget}_{\Ass}$ concerning this order have the required form.
\end{proof}

\begin{theorem}\label{thm:ass}
The homology contractad of the little intervals contractad is isomorphic to the Associative contractad
\[
H_{0}(\D_1)\cong \Ass.
\] Moreover, this contractad is self dual $\Ass^{!}\cong \Ass$ and Koszul.
\end{theorem}
\begin{proof}
Note that the morphism of contractads $\Ass \rightarrow H_{0}(\D_1)$ described above is surjective since $\nu$ generates $H_{0}(\D_1)$. Consider the shuffle version $\Ass^{\forget}$ and the monomial order described above.  By Lemma~\ref{quantumorder}, the normal monomials in each component $\Tree_{\{ m, b\}}(\Gr)$ with  respect to the quadratic relations $\R_{\Ass}$ have the following form
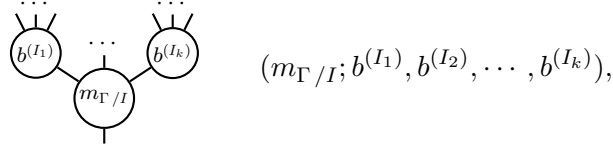
\[
    \vcenter{\hbox{\begin{tikzpicture}[
        scale=0.6,
        vert/.style={inner sep=1pt, circle,  draw, thick},
        leaf/.style={inner sep=2pt,rectangle},
        edge/.style={-,black!30!black, thick},
        ]
        \node[vert] (1) at (0,1) {\scriptsize$m_{\scalebox{0.7}{$\Gr/I$}}$};
        \node[vert] (l) at (-1.5,2) {\scriptsize$b^{\scalebox{0.7}{$(I_1)$}}$};
        \node[leaf] (ll) at (-2,3) {\space};
        \node[leaf] (lm) at (-1.5,3.1) {\footnotesize$\cdots$};
        \node[leaf] (lr) at (-1,3) {\space};
        \node[leaf] (mid) at (0,2.2) {\footnotesize$\cdots$};
        \node[vert] (r) at (1.5,2) {\scriptsize$b^{\scalebox{0.7}{$(I_k)$}}$};
        \node[leaf] (rl) at (1,3) {\space};
        \node[leaf] (rm) at (1.5,3.1) {\footnotesize$\cdots$};
        \node[leaf] (rr) at (2,3) {\space};
        \draw[edge] (0,0)--(1);
        \draw[edge] (1)--(l);
        \draw[edge] (l)--(lm);
        \draw[edge] (l)--(ll);
        \draw[edge] (l)--(lr);
        \draw[edge] (1)--(mid);
        \draw[edge] (1)--(r);
        \draw[edge] (r)--(rm);
        \draw[edge] (r)--(rl);
        \draw[edge] (r)--(rr);
    \end{tikzpicture}}} \quad \quad (m_{\Gr/I};b^{(I_1)},b^{(I_2)},\cdots,b^{(I_k)}),
\]
where $I=\{I_1,I_2,\cdots,I_k\}$ is ranged over all partitions of $\Gr$, $m_{\Gr/I}$ is the unique $\Com$-monomial, and $b^{(I_j)}$ are $\Lie$-monomials in the corresponding components. Indeed, the relations of the form $\{b \circ^{\Gr}_{e} m| e \in E_{\Gr}\}$ ensures us that the vertices labeled by $m$ in a normal monomial are concentrated at the bottom. Recall that $\Lie$-monomials form a  basis of $\Lie^{\forget}$, so the number of normal monomials in $\Ass^{\forget}$ is equal to
\[
\sum_{I  \in \parti(\Gr)} \prod_{G \in I}\dim\Lie(\Gr|_G)=\sum_{I \in \parti(\Gr)} \prod_{G \in I}|\mu(\Gr|_G)|.
\]
By Proposition~\ref{dimD1} and Proposition~\ref{coset}, we conclude that the morphism $\Ass\twoheadrightarrow H_{\bullet}(\D_1)$ is an isomorphism and the quadratic relations form a Gr\"obner basis. Hence, by Theorem~ \ref{quadgrob}, this contractad is Koszul.
\end{proof}

\subsection{Two-dimensional case}\label{subsec:2disks}
In this subsection, we describe the homology contractad of the little 2-disks contractad $\D_2$. This contractad is a graphical counterpart of the operad $\mathsf{Gerst}$ of Gerstenhaber algebras. Moreover, we prove that this contractad is Koszul.

Consider the little 2-disks  contractad $\D_2$. By Proposition~\ref{discsconf},  each component of this contractad $\D_2(\Gr) \simeq \Conf_{\Gr}(\mathbb{R}^2)\cong \mathbb{C}^n\setminus \B(\Gr)$ is homotopy equivalent to the complement of the complex graphic arrangement.

Recall that cohomology ring $H^{\bullet}(V\setminus \bigcup \mathcal{A})$ of a complement to a central complex hyperplane arrangement $\mathcal{A}$ is isomorphic to an \textit{Orlik-Solomon algebra} $\OS(\mathcal{A})$ \cite{yuzvinsky2001orlik}. The latter algebra is the $\mathbb{Z}$-algebra generated by logarithmic $1$-forms $\omega_{H}=\frac{d\alpha_{H}}{\alpha_H}$, where $H$ is a hyperplane from the arrangement, and $\alpha_H$ is a non-zero linear function from the annihilator of $H^{\bot}$. Note that the element $\omega_H$ is determined uniquely since the annihilator $H^{\bot}$ is a one-dimensional subspace and the form $\omega_H$ is invariant under rescaling. A subset $S \subset \mathcal{A}$ is called \textit{dependent} if there is an element $H' \in S$ such that the intersection $\bigcap_{H \in S\setminus \{H'\}} H=\bigcap_{H \in S} H$. Note that the independent sets defines the intersection matroid $M(\A)$ on the set $\A$. It is known that, for each dependent set $S=\{H_1,\cdots, H_n\}$, there is the relation of form $ \sum_{i=1}^n (-1)^{i-1}\omega_{H_1}\omega_{H_2}...\hat{\omega}_{H_i}...\omega_{H_n}=0$ and such relations determine the algebra $\OS(\mathcal{A})$ uniquely. Recall that the Hilbert series of the Orlik-Solomon algebra $\OS(\A)$ is uniquely determined by the intersection matroid $M(\A)$
\[
H_{\OS(\A)}(t)=\sum_{X\in \La(\A)} t^{\rk X}\mu_{\La(\A)}(\hat{0},X)=t^{\rk \A}\chi_{M(\A)}(\frac{1}{t}),
\] where $\chi_{M(\A)}(t)$ is the characteristic polynomial of the intersection matroid $M(\A)$~\cite[Th.~4.2]{eschenbrenner1999orlik}.

\begin{prop}
\label{discring}
\begin{enumerate}
\item For a graph $\Gr$, the cohomology ring of the graphical configuration space $H^{\bullet}(\Conf_{\Gr}(\mathbb{C}))$ is generated by logarithmic forms $\omega_e$
\[
\omega_e=\frac{d(x_v-x_w)}{x_v-x_w}\text{, for } e=(v,w)\in E_{\Gr},
\]satisfying the relations
\[
    \sum_{i=1}^n (-1)^{i-1}\omega_{e_1}\omega_{e_2}...\hat{\omega}_{e_i}...\omega_{e_n}=0, \quad \text{if}\quad \{e_1,e_2,\cdots,e_n\} \subset E_{\Gr}\quad \text{contains a cycle.}
\]
\item The Hilbert series $H_{\Conf_{\Gr}(\mathbb{C})}(t)=\sum_{i=0} (-t)^i\dim H^i(\Conf_{\Gr}(\mathbb{C}))$ is given by the formula
\[
H_{\Conf_{\Gr}(\mathbb{C})}(t)=\sum_{I \in \parti(\Gr)} t^{\rk I}\prod_{G \in I}\mu(\Gr|_G)=t^{|V_{\Gr}|}\chi_{\Gr}(\frac{1}{t}),
\] where $\chi_{\Gr}(t)$ is the chromatic polynomial of the graph.
\end{enumerate}
\end{prop}
\begin{proof}
As we have mentioned before, the graphical configuration space $\Conf_{\Gr}(\mathbb{C})$ is precisely the complement to the graphic arrangement $\B(\Gr)$, hence the cohomology ring $H^{\bullet}(\Conf_{\Gr}(\mathbb{C}))$ is isomorphic to the Orlik-Solomon algebra $\OS(\B(\Gr))$. The relations above follow from the fact that a subset of hyperplanes $\{H_{e_1},\cdots, H_{e_n}\} \subset \B(\Gr)$ is dependent if the underlying set of edges $\{e_1, \cdots, e_n\}$ contains a cycle. From the description of independent sets, we see that the intersection matroid of the graphic arrangement $M(\B(\Gr))$ coincides with the so-called graphic matroid $M_{\Gr}$. Recall that the characteristic polynomial of the graphic matroid is given by the rule $\chi_{M_{\Gr}}(t)=(\frac{1}{t})^{\pi_0(\Gr)}\chi_{\Gr}(t)$, where $\pi_0(\Gr)$ is the number of connected components, hence $t^{\rk \B(\Gr)}\chi_{M_{\Gr}}(\frac{1}{t})=t^{|V_{\Gr}|}\chi_{\Gr}(\frac{1}{t})$. So, we have  
\[
H_{\OS(\B(\Gr))}(t)=t^{|V_{\Gr}|}\chi_{\Gr}(\frac{1}{t})=\sum_{I \in \La(\B(\Gr))} t^{\rk I}\mu(\hat{0},I)=\sum_{I \in \parti(\Gr)} t^{\rk I}\prod_{G \in I}\mu(\Gr|_G).
\]
\end{proof}

As an immediate consequence, from the identity $\sum_{I \in \parti(\Gr)} t^{\rk I}\prod_{G \in I}\mu(\Gr|_G)=t^{|V_{\Gr}|}\chi_{\Gr}(\frac{1}{t})$, we have
\begin{cor}\label{sled:hilbertseries}
For a connected graph $\Gr$, we have
\begin{enumerate}
    \item $\mu(\Gr)=\chi_{\Gr}'(0)$.
    \item $\dim H_0(\D_1)(\Gr) =(-1)^{|V_{\Gr}|}\chi_{\Gr}(-1)$. 
\end{enumerate}
\end{cor}

Let us denote by $\OS$ the graphical collection of cohomology rings $\OS(\Gr):=H^{\bullet}(\Conf_{\Gr}(\mathbb{C}))$ in the category of graded commutative rings. This graphical collection is endowed with a cocontractad structure concerning the homotopy equivalence $\D_2(\Gr) \cong \Conf_{\Gr}(\mathbb{C})$.
\begin{prop}
\label{osmaps}
The contractad structure on the little 2-disks $\D_2$ induces a cocontractad structure on the graphical collection of Orlik-Solomon algebras $\OS$. The infinitesimal compositions are homomorphisms of algebras given by the rule:
\begin{gather*}
    \triangle_G^{\Gr}\colon\OS(\Gr) \rightarrow \OS(\Gr/G)\otimes\OS(\Gr|_G)
    \\
    \omega_{e}\mapsto \begin{cases}
    \omega_{e'} \otimes 1, \text{ if } e \not\subset G
    \\
    1\otimes \omega_{e}, \text{ if } e \subset G
    \\
    \end{cases},
\end{gather*} where $e'$ is the image of $e$ under contraction $\Gr \rightarrow \Gr/G$.
\end{prop}
\begin{proof}
As cohomology is a contravariant functor, it sends topological contractad $\D_2$ to the cocontractad $H^{\bullet}(\D_2)\cong \OS$ in the category of $\mathbb{Z}$-algebras. For a graph $\Gr$ and tube $G$, consider the infinitesimal composition $\triangle_G^{\Gr}\colon \OS(\Gr) \rightarrow \OS(\Gr/G)\otimes \OS(\Gr|_G)$. This map is uniquely determined by the images of generators $\triangle_G^{\Gr}\omega_e$ since it is a  ring homomorphism. By the definition, each generator $\omega_e \in \OS(\Gr)$ is a pullback $f_e^* \omega_e$ of the 1-form on $\Conf_{\Gr|_e}(\mathbb{C})$, where $f_e: \Conf_{\Gr}(\mathbb{C})\rightarrow \Conf_{\Gr|_e}(\mathbb{C})$ is the map forgetting points out of $e$. For each edge $e$, define the map $g_e\colon \Conf_{\Gr/G}(\mathbb{C})\times\Conf_{\Gr|_G}(\mathbb{C})\rightarrow \Conf_{e}(\mathbb{C})$ as follows:  if $e \subset G$ take the composition of the projection on the right factor with the forgetful map $\Conf_{\Gr|_G}(\mathbb{C}) \rightarrow \Conf_{\Gr|_e}(\Gr)$, otherwise, take the composition of the projection on the left factor with the forgetful map $\Conf_{\Gr/G}(\mathbb{C}) \rightarrow \Conf_{\Gr|_{e'}}(\Gr)$. We have the commutative diagram
\[\begin{tikzcd}
	{\mathcal{D}_2(\Gamma/G)\times \mathcal{D}_2(\Gamma|_G)} && {\mathcal{D}_2(\Gamma)} \\
	\\
	{\Conf_{\Gr/G}(\mathbb{C})\times\Conf_{\Gr|_G}(\mathbb{C}) } && {\Conf_{\Gr}(\mathbb{C})} \\
	& {\Conf_{\Gr|_e}(\mathbb{C})}
	\arrow["\pi", from=1-1, to=3-1]
	\arrow["{\circ_G^{\Gamma}}", from=1-1, to=1-3]
	\arrow["\pi", from=1-3, to=3-3]
	\arrow["{g_e}"', from=3-1, to=4-2]
	\arrow["{f_e}", from=3-3, to=4-2]
\end{tikzcd}\] from which we deduce
\[
\triangle^{\Gr}_G\omega_{e}:=(\pi^*)^{-1}(\circ^{\Gr}_G)^*\pi^*(f^*_e\omega_e)=g^*_e\omega_e= \begin{cases}
    \omega_{e'} \otimes 1, \text{ if } e \not\subset G
    \\
    1\otimes \omega_{e}, \text{ if } e \subset G
    \\
    \end{cases},
\]
\end{proof}

From the description of cohomology rings, we see that the natural pairing $H^{\bullet}(\D_2)\otimes H_{\bullet}(\D_2) \rightarrow \mathbb{Z}$ is perfect. So the homology contractad $H_{\bullet}(\D_2)$ is isomorphic to the dual contractad $H^{\bullet}(\D_2)^{\vee}$ concerning the pairing:
\begin{equation*}
    H_{\bullet}(\D_2)\cong H^{\bullet}(\D_2)^{\vee}=\OS^{\vee}.
\end{equation*} Let $m, b \in H_{\bullet}(\D_2(\Path_2))$ be the classes dual to $1$ and $\omega_{12}$, respectively.
\begin{lemma}
\label{relD2}
We have
\begin{gather}
    m \circ_{\{1,2\}}^{\Path_3} m -  m \circ_{\{2,3\}}^{\Path_3} m=0
    \\
    b \circ_{\{1,2\}}^{\Path_3} b  +b \circ_{\{2,3\}}^{\Path_3} b=0
    \\
    b \circ_{\{1,2\}}^{\Path_3} m -  m \circ_{\{2,3\}}^{\Path_3} b=0
    \\
    m \circ_{\{1,2\}}^{\K_3}m  -  m \circ_{\{2,3\}}^{\K_3} m=0
    \\
    b\circ_{\{1,2\}}^{\K_3} b + (b\circ_{\{1,2\}}^{\K_3} b)^{(123)} + (b\circ_{\{1,2\}}^{\K_3} b)^{(321)}=0
    \\
    b\circ_{\{1,2\}}^{\K_3} m - m\circ_{\{2,3\}}^{\K_3} b - (m\circ_{\{1,2\}}^{\K_3}b)^{(23)}=0
\end{gather}
\end{lemma}
\begin{proof}
By direct computations using Proposition~\ref{discring} and Proposition~\ref{osmaps}.
\end{proof}
Recall that the restriction of this contractad $\K_{*}(H_{\bullet}(\D_2))$ to complete graphs is isomorphic to the symmetric  operad $\mathsf{Gerst}$ of Gerstenhaber algebras~\cite{cohen1976homology}. We introduce a graphical counterpart of the latter operad.
\begin{defi}
\label{gerst}
The Gerstenhaber contractad $\Gerst$ is the contractad with two symmetric generators $m$ and $b$ in the component $\mathsf{P}_2$ of degree $0$ and $1$, respectively, satisfying the relations
\begin{gather*}
    m \circ_{\{1,2\}}^{\Path_3} m =  m \circ_{\{2,3\}}^{\Path_3} m 
    \\
    b \circ_{\{1,2\}}^{\Path_3} b = -b \circ_{\{2,3\}}^{\Path_3} b
    \\
    b \circ_{\{1,2\}}^{\Path_3} m =  m \circ_{\{2,3\}}^{\Path_3} b
    \\
    m \circ_{\{1,2\}}^{\K_3}m  =  m \circ_{\{2,3\}}^{\K_3} m
    \\
    b\circ_{\{1,2\}}^{\K_3} b + (b\circ_{\{1,2\}}^{\K_3} b)^{(123)} + (b\circ_{\{1,2\}}^{\K_3} b)^{(321)}=0
    \\
    b\circ_{\{1,2\}}^{\K_3} m = m\circ_{\{2,3\}}^{\K_3} b + (m\circ_{\{1,2\}}^{\K_3}b)^{(23)}
\end{gather*}
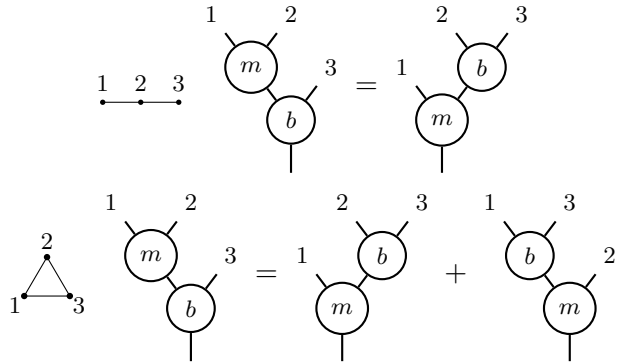
\begin{figure}[ht]
    \centering
    \caption{Rewriting rules in $\Gerst$.}
    \begin{gather*}
    \vcenter{\hbox{\begin{tikzpicture}[scale=0.5]
    \fill (0,0) circle (2pt);
    \node at (0,0.5) {\footnotesize$1$};
    \fill (1,0) circle (2pt);
    \node at (1,0.5) {\footnotesize$2$};
    \fill (2,0) circle (2pt);
    \node at (2,0.5) {\footnotesize$3$};
    \draw (0,0)--(1,0)--(2,0);    
    \end{tikzpicture}}}
    \vcenter{\hbox{\begin{tikzpicture}[
        scale=0.7,
        vert/.style={circle,  draw=black!30!black, thick, minimum size=1mm},
        leaf/.style={rectangle, thick, minimum size=1mm},
        edge/.style={-,black!30!black, thick},
        ]
        \node[vert] (1) at (0,1) {\footnotesize$b$};
        \node[leaf] (l1) at (0.75,2) {\footnotesize$3$};
        \node[vert] (2) at (-0.75,2) {\footnotesize$m$};
        \node[leaf] (l2) at (0,3) {\footnotesize$2$};
        \node[leaf] (3) at (-1.5,3) {\footnotesize$1$};
        \draw[edge] (0,0)--(1);
        \draw[edge] (1)--(2)--(3);
        \draw[edge] (1)--(l1);
        \draw[edge] (2)--(l2);
    \end{tikzpicture}}}=
    \vcenter{\hbox{\begin{tikzpicture}[
        scale=0.7,
        vert/.style={circle,  draw=black!30!black, thick, minimum size=1mm},
        leaf/.style={rectangle, thick, minimum size=1mm},
        edge/.style={-,black!30!black, thick},
        ]
        \node[vert] (1) at (0,1) {\footnotesize$m$};
        \node[leaf] (l1) at (-0.75,2) {\footnotesize$1$};
        \node[vert] (2) at (0.75,2) {\footnotesize$b$};
        \node[leaf] (l2) at (1.5,3) {\footnotesize$3$};
        \node[leaf] (3) at (0,3) {\footnotesize$2$};
        \draw[edge] (0,0)--(1);
        \draw[edge] (1)--(2)--(3);
        \draw[edge] (1)--(l1);
        \draw[edge] (2)--(l2);
    \end{tikzpicture}}}
    \\
     \vcenter{\hbox{\begin{tikzpicture}[scale=0.6]
    \fill (-0.5,0) circle (2pt);
    \node at (-0.7,-0.2) {\footnotesize$1$};
    \fill (0.5,0) circle (2pt);
    \node at (0.7,-0.2) {\footnotesize$3$};
    \fill (0,0.86) circle (2pt);
    \node at (0,1.2) {\footnotesize$2$};
    \draw (-0.5,0)--(0,0.86)--(0.5,0)--cycle;    
    \end{tikzpicture}}}
    \vcenter{\hbox{\begin{tikzpicture}[
        scale=0.7,
        vert/.style={circle,  draw=black!30!black, thick, minimum size=1mm},
        leaf/.style={rectangle, thick, minimum size=1mm},
        edge/.style={-,black!30!black, thick},
        ]
        \node[vert] (1) at (0,1) {\footnotesize$b$};
        \node[leaf] (l1) at (0.75,2) {\footnotesize$3$};
        \node[vert] (2) at (-0.75,2) {\footnotesize$m$};
        \node[leaf] (l2) at (0,3) {\footnotesize$2$};
        \node[leaf] (3) at (-1.5,3) {\footnotesize$1$};
        \draw[edge] (0,0)--(1);
        \draw[edge] (1)--(2)--(3);
        \draw[edge] (1)--(l1);
        \draw[edge] (2)--(l2);
    \end{tikzpicture}}}=
    \vcenter{\hbox{\begin{tikzpicture}[
        scale=0.7,
        vert/.style={circle,  draw=black!30!black, thick, minimum size=1mm},
        leaf/.style={rectangle, thick, minimum size=1mm},
        edge/.style={-,black!30!black, thick},
        ]
        \node[vert] (1) at (0,1) {\footnotesize$m$};
        \node[leaf] (l1) at (-0.75,2) {\footnotesize$1$};
        \node[vert] (2) at (0.75,2) {\footnotesize$b$};
        \node[leaf] (l2) at (1.5,3) {\footnotesize$3$};
        \node[leaf] (3) at (0,3) {\footnotesize$2$};
        \draw[edge] (0,0)--(1);
        \draw[edge] (1)--(2)--(3);
        \draw[edge] (1)--(l1);
        \draw[edge] (2)--(l2);
    \end{tikzpicture}}}+
    \vcenter{\hbox{\begin{tikzpicture}[
        scale=0.7,
        vert/.style={circle,  draw=black!30!black, thick, minimum size=1mm},
        leaf/.style={rectangle, thick, minimum size=1mm},
        edge/.style={-,black!30!black, thick},
        ]
        \node[vert] (1) at (0,1) {\footnotesize$m$};
        \node[leaf] (l1) at (0.75,2) {\footnotesize$2$};
        \node[vert] (2) at (-0.75,2) {\footnotesize$b$};
        \node[leaf] (l2) at (0,3) {\footnotesize$3$};
        \node[leaf] (3) at (-1.5,3) {\footnotesize$1$};
        \draw[edge] (0,0)--(1);
        \draw[edge] (1)--(2)--(3);
        \draw[edge] (1)--(l1);
        \draw[edge] (2)--(l2);
    \end{tikzpicture}}}
    \end{gather*}
\end{figure}
\end{defi}

\begin{theorem}\label{thm:gerst}
The homology contractad of the Little 2-disks is isomorphic to the Gerstenhaber contractad
\[
H_{\bullet}(\D_2)\cong \Gerst.
\] Moreover, this contractad is self dual up to suspension $\Gerst^{!}\cong \Susp\Gerst$ and Koszul.
\end{theorem}
\begin{proof}
By Lemma~\ref{relD2}, we have a well-defined morphism of contractads $\tau\colon \Gerst \rightarrow H_{\bullet}(\D_2)$. Similarly to the case of the contractad $\Ass$, consider the quantum order on tree-monomials $\Tree_{\{ m, b\}}$. It is easy to see that leading terms of quadratic relations for the contractad $\Gerst^{\forget}$ coincide with the ones for the Associative contractad: $\LT(\R^{\forget}_{\Ass})=\LT(\R^{\forget}_{\Gerst}).$  By Proposition~\ref{discring}, we obtain that the number of normal monomials in each component coincides with the dimension of Orlik-Solomon algebra $\OS(\Gr)$. Hence, we have component-wise inequality $\dim \Gerst(\Gr) \leq \dim H_{\bullet}(\D_2(\Gr))$ in each component. To complete the proof, we use the following lemma, the proof of which is we leave to Appendix~\ref{subsec:lemma}.
\begin{lemma}
\label{onto}
The morphism $\tau\colon\Gerst \rightarrow H_{\bullet}(\D_2)$ is surjective.
\end{lemma} \noindent By dimension reasons, we conclude that the morphism $\tau\colon\Gerst \rightarrow H_{\bullet}(\D_2)$ is an isomorphism and, moreover, $\Gerst^{\forget}$ has a quadratic Gr\"obner basis. Hence, this contractad is Koszul.
\end{proof}

\subsection{Higher dimensional cases.}

In this subsection, we complete the description of the homology contractads $H_{\bullet}(\D_n)$. For $n\geq 2$, the description of $H_{\bullet}(\D_n)$ mimics the case $n=2$.

An $n$-arrangement of a real vector space $V$ is an arrangement  $\A$ of subspaces of codimension $n$. By Proposition~\ref{discsconf},  each component of the contractad $\D_n(\Gr) \simeq \Conf_{\Gr}(\mathbb{R}^n)$ is homotopy equivalent to the complement of graphic $n$-arrangement $\B(\Gr)$.

\begin{lemma}
Let $\Gr$ be a connected simple graph, and $n\geq2$. The cohomology ring of the configuration space $H^{\bullet}(\Conf_{\Gr}(\mathbb{R}^n))$ is generated by elements $\omega_e$ of degree $n-1$ labeled by edges of the graph, satisfying the relations
\begin{gather*}
    \omega^2_{e}=0
    \\
    \sum_{i=1}^n (-1)^{(n-1)(i-1)}\omega_{e_1}\omega_{e_2}...\hat{\omega}_{e_i}...\omega_{e_n}=0\text{, if } \{e_1,e_2,\cdots,e_n\} \subset E_{\Gr} \text{ contains a cycle.}
\end{gather*}
\end{lemma}
\begin{proof}
This follows from the description of the cohomology ring of a complement to $n$-arrangement~\cite[Th.~5.5]{de2001cohomology}.
\end{proof}
Note that for distinct $n,m$, the cohomology rings $H^{\bullet}(\Conf_{\Gr}(\mathbb{R}^n))$ and $H^{\bullet}(\Conf_{\Gr}(\mathbb{R}^m))$ differ only in the grading. Hence, we can adapt the results of Section~\ref{subsec:2disks} to the general case. The proof of the following proposition mimics that of Theorem~\ref{thm:gerst}.
\begin{theorem}\label{thm:en}
For $n\geq  2$, the homology contractad $H_{\bullet}(\D_n)$ is the contractad generated by a symmetric generator $m$  of degree $0$ and generator $c_{n}^{(12)}=(-1)^{n}c_{n}$ of degree $(n-1)$ in the component $\Path_2$, satisfying  the relations
\begin{gather*}
    m \circ_{\{1,2\}}^{\Path_3} m =  m \circ_{\{2,3\}}^{\Path_3} m,
    \\
    c_{n} \circ_{\{1,2\}}^{\Path_3} c_{n}=(-1)^{n-1}  c_{n} \circ_{\{2,3\}}^{\Path_3} c_{n},
    \\
    c_{n} \circ_{\{1,2\}}^{\Path_3} m =  m \circ_{\{2,3\}}^{\Path_3} c_{n},
    \\
    m \circ_{\{1,2\}}^{\K_3}m  =  m \circ_{\{2,3\}}^{\K_3} m,
    \\
    c_{n}\circ_{\{1,2\}}^{\K_3} c_{n} + (c_{n}\circ_{\{1,2\}}^{\K_3} c_{n})^{(123)} + (c_{n}\circ_{\{1,2\}}^{\K_3} c_{n})^{(321)}=0,
    \\
    c_{n}\circ_{\{1,2\}}^{\K_3} m = m\circ_{\{2,3\}}^{\K_3} c_{n} + (m\circ_{\{1,2\}}^{\K_3}c_{n})^{(23)}.
\end{gather*}
These contractads are self  dual up to suspension
\[
H_{\bullet}(\D_n)^{!} \cong \Susp^{n-1}H_{\bullet}(\D_n)
\] and Koszul.
\end{theorem}
\appendix
\section{\space} 

\subsection{Quadratic Gr\"obner basis implies Koszulity}\label{subsec:pbw}
In this subsection, we prove Theorem~\ref{quadgrob}. The proof is just an adaptation of that of \cite[Sec.~4-5]{hoffbeck2010poincare}. The proof is completed in several steps.\\

\noindent \textbf{Step 1:} Consider the free shuffle contractad $\T_{\Sha}(\E)$ together with monomials $\Tree_{\B}$. Let us describe its Bar-construction $\mathsf{B}\T_{\Sha}(\E)$ explicitly. By definition, we have $\mathsf{B}\T_{\Sha}(\E)(\Gr)=\bigoplus_{T\in \Tree(\Gr)} \T_{\Sha}(\E)(T)$, where $\T_{\Sha}(\E)(T)=\bigotimes_{v\in \Ver(T)}\T_{\Sha}(\E)(\In(v))$. Explicitly, a generator of $\T_{\Sha}(\E)(T)$ corresponds to a tree $T$ labeled by tree monomials $\Tree_{\B}$. We can represent it by a large tree-monomial $\tau\in \Tree_{\B}$ equipped with a splitting in subtrees $\tau_{\mathrm{comp}}$, which we can see as connected components. The $\tau_{\mathrm{comp}}$ are separated by cutting edges which form a subset $D \subset \Edge(\tau)$. The union of the internal edges of the subtrees $\tau_{\mathrm{comp}}$ forms a set $S \subset \Edge(\tau)$ such that $S\coprod D=\Edge(\tau)$. We will work with $S$, the set of marking edges. So, each monomial in $\mathsf{B}\T_{\Sha}(\E)$ is identified with a pair $(\alpha, S)$, where $\alpha$ is a monomial from $\Tree_{\B}$ together with a subset $S\subset \Edge(\alpha)$ of edges of the underlying tree of $\alpha$. In this notation, the differential has the form
\[
\delta((\alpha,S))=\sum_{e\in \Edge(\alpha)\setminus S} \pm
(\alpha, S\coprod e).\]
Note that the zero syzygy degree of the Bar-construction consists of monomials of the form $(\alpha,\varnothing)$.\\

\noindent \textbf{Step 2:} Let $\Pop=\Pop(\E,\R)$ be a quadratic shuffle contractad which has a quadratic Gr\"obner basis with respect to some monomial order $<$ on $\Tree_{\B}$. Since $\Pop=\T(\E)/\langle \R \rangle$, its bar construction $\mathsf{B}\Pop$ is a quotient of $\mathsf{B}\T_{\Sha}(\E)$. For each ordered graph $\Gr$, consider the filtration of $\mathsf{B}\Pop(\Gr)$ indexed by monomials
\begin{gather*}
    \mathsf{B}\Pop(\Gr)=\bigcup_{\alpha \in \Tree_{\B}(\Gr)} \mathsf{B}\Pop(\Gr)_{\alpha},
\end{gather*} where the subcomplex $\mathsf{B}\Pop(\Gr)_{\alpha}$ is generated by monomials $(\beta,S)$ such that $\beta\leq \alpha$. Note that this filtration  preserves differential. Consider $E^0$-page of the spectral sequence associated with the filtration. The complex $E_{\alpha}^0\mathsf{B}\Pop(\Gr)$ is generated by monomials $(\alpha,S)$, such that each subtree from $\alpha_{\mathrm{comp}}$ is  a normal monomial with respect to the ideal  of relations.

For a monomial $\alpha\in \Tree_{\B}(\Gr)$, define a \textit{normal edge} as an edge $e\in \Edge(\alpha)$, such that the restricted monomial $\alpha|_e$ is a normal monomial with  respect  to $\R$. Denote by $\mathsf{Adm}(\alpha)$ the set of normal edges. Since $\Pop$ has a quadratic Gr\"obner basis, we have $(\alpha,S)\in E_{\alpha}^0\mathsf{B}\Pop(\Gr)$ if $S\subset \mathsf{Adm}(\alpha)$. Moreover, the differential in $E_{\alpha}^0\mathsf{B}\Pop(\Gr)$ has the form
\[
\delta^0((\alpha,S))=\sum_{e\in \mathsf{Adm}(\alpha)\setminus S} \pm
(\alpha, S\coprod e).\]\\

\noindent \textbf{Step 3:} From the description above, we see that the complex $E_{\alpha}^0\mathsf{B}\Pop(\Gr)$ is isomorphic to the augmented dual of the combinatorial complex $C_{\bullet}(\triangle_{\mathsf{Adm}(\alpha)})$. Note that this complex has trivial homology except when $\mathsf{Adm}(\alpha)=\varnothing$, in which case the complex is reduced to a one-dimensional space (with generator $(\alpha,\varnothing)$). By a standard spectral sequence argument, we conclude that homology  of $\mathsf{B}\Pop$ are concentrated  in the zero syzygy degree. Hence, $\Pop$ is  Koszul. Furthermore, we see that the collection  of monomials $\{\gamma|\mathsf{Adm}(\gamma)=\varnothing\}$ forms a basis of $H^0(\mathsf{B}\Pop)$. Since $\Pop$ is  Koszul, we  obtain an isomorphism of graphical collections $H^{0}(\mathsf{B}\Pop)\cong \Pop^!$(if we ignore grading). Note that monomials without normal edges correspond to normal monomials with respect to the quadratic  relations $\R^{\bot}$ and reverse monomial order. Hence, by Proposition~\ref{coset}, the  quadratic relations of $\Pop^!$ form a Gr\"obner  basis of the ideal of relations. 

\subsection{PBW theorem for contractads}\label{subsec:diamond} We shall briefly discuss some tools for computing Gr\"obner basis of an ideal. We leave proofs and constructions to the reader, see  for instance \cite[Sec.~5.5]{bremner2016algebraic}. Firstly, we can adapt Buchberger's algorithm to contractads, analogously to the case of operads \cite[Sec.~3.7]{dotsenko2010grobner}. So, the task of finding Gr\"obner basis of an ideal reduces to the computation of $S$-polynomials. Furthermore, this technique gives us an efficient criterion to recognize when a contractad has a quadratic Gr\"obner basis. 

\begin{theorem}[PBW theorem for contractads]\label{thm:pbw}
    Let $\Pop=\Pop(\E,\R)$ be a quadratic shuffle contractad with a monomial order. This contractad has a quadratic Gr\"obner basis, if the morphism of non-symmetric graphical collections
    \[
    \Pop(\E,\LT(\R))\twoheadrightarrow \Pop(\E,\R),
    \] is an isomorphism in the weight 3
    \[
    \Pop^{(3)}(\E,\LT(\R))\overset{\cong}{\rightarrow} \Pop^{(3)}(\E,\R).
    \]
\end{theorem}
\begin{proof}
The condition that the morphism $\Pop(\E,\LT(\R))\twoheadrightarrow \Pop(\E,\R)$ is an isomorphism in the weight 3 is equivalent to the fact that all $S$-polynomials between quadratic relations are reduced to zero with respect to the set of quadratic relations. Hence, by Buchberer's algorithm, the quadratic relations of $\Pop$ forms a Gr\"obner basis of the ideal  of relations.
\end{proof}

Let us discuss the statement of this theorem in the case when a contractad $\Pop$ is binary, i.e., generators are concentrated in component $\Path_2$. In this case, the weight 3 component $\Pop^{(3)}$  corresponds to components of $\Pop$ labeled by ordered graphs on 4 vertices. There are 5 non-isomorphic connected graphs on 4 vertices, presented in the figure below. Since a connected graph admits different non-isomorphic orderings, we obtain $12+4+12+3+6+1=38$ non-isomorphic ordered graphs on $4$ vertices. Hence, to check that a binary shuffle contractad has a quadratic Gr\"obner basis, we need to find normal monomials with respect to quadratic relations for $38$ different ordered graphs.
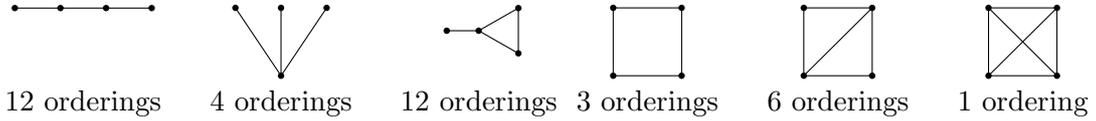
\[
\vcenter{\hbox{\begin{tikzpicture}[scale=0.6]
    \fill (0,0) circle (2pt);
    \fill (1,0) circle (2pt);
    \fill (2,0) circle (2pt);
    \fill (3,0) circle (2pt);
    \draw (0,0)--(1,0)--(2,0)--(3,0);
    \node at (1.5,-2.1) {$12$ orderings};
    \end{tikzpicture}}}
    \quad
\vcenter{\hbox{\begin{tikzpicture}[scale=0.6]
    \fill (0,0) circle (2pt);
    \fill (-1,1.5) circle (2pt);
    \fill (0,1.5) circle (2pt);
    \fill (1,1.5) circle (2pt);
    \draw (0,0)--(-1,1.5);
    \draw (0,0)--(0,1.5);
    \draw (0,0)--(1,1.5);
    \node at (0,-0.6) {$4$ orderings};
    \end{tikzpicture}}}
    \quad
\vcenter{\hbox{\begin{tikzpicture}[scale=0.6]
    \draw (0,0)--(0.87,0.5)--(0.87,-0.5)--cycle;
   \draw (-0.7,0)--(0,0);
   \fill (0,0) circle (2pt);
   \fill (-0.7,0) circle (2pt);
    \fill (0.87,0.5) circle (2pt);
    \fill (0.87,-0.5) circle (2pt);
    \node at (0,-1.6) {$12$ orderings};
\end{tikzpicture}}}
\vcenter{\hbox{\begin{tikzpicture}[scale=0.6]
    \fill (0,0) circle (2pt);
    \fill (0,1.5) circle (2pt);
    \fill (1.5,0) circle (2pt);
    \fill (1.5,1.5) circle (2pt);
    \draw (0,0)--(1.5,0)--(1.5,1.5)--(0,1.5)-- cycle;
    \node at (0.75,-0.6) {$3$ orderings};
    \end{tikzpicture}}}
    \quad
\vcenter{\hbox{\begin{tikzpicture}[scale=0.6]
    \fill (0,0) circle (2pt);
    \fill (0,1.5) circle (2pt);
    \fill (1.5,0) circle (2pt);
    \fill (1.5,1.5) circle (2pt);
    \draw (0,0)--(1.5,0)--(1.5,1.5)--(0,1.5)-- cycle;
    \draw (0,0)--(1.5,1.5);
    \node at (0.75,-0.6) {$6$ orderings};
    \end{tikzpicture}}}
    \quad
\vcenter{\hbox{\begin{tikzpicture}[scale=0.6]
    \fill (0,0) circle (2pt);
    \fill (0,1.5) circle (2pt);
    \fill (1.5,0) circle (2pt);
    \fill (1.5,1.5) circle (2pt);
    \draw (0,0)--(1.5,0)--(1.5,1.5)--(0,1.5)-- cycle;
    \draw (0,0)--(1.5,1.5);
    \draw (1.5,0)--(0,1.5);
    \node at (0.75,-0.6) {$1$ ordering};
    \end{tikzpicture}}}
\]

\noindent As an example, let us construct an alternative basis of the contractad $\Lie$.
\begin{prop}
    The shuffle contractad $\Lie^{\forget}$ has a quadratic Gr\"obner basis with respect to reverse $\grpermlex$-order.
\end{prop}
\begin{proof} By Theorem~\ref{thm:pbw} and Corollary~\ref{cor:liekoszul}, it suffices to check that, for each ordered graph on 4 vertices $\Gr$, the number of normal $\Gr$-monomials coincides with the M\"obius function $|\mu(\Gr)|$. For graphs on $4$ vertices, we have
\[
|\mu(\vcenter{\hbox{\begin{tikzpicture}[scale=0.3]
    \fill (0,0) circle (2.3pt);
    \fill (1,0) circle (2.3pt);
    \fill (2,0) circle (2.3pt);
    \fill (3,0) circle (2.3pt);
    \draw (0,0)--(1,0)--(2,0)--(3,0);
    \end{tikzpicture}}})|=|\mu(\vcenter{\hbox{\begin{tikzpicture}[scale=0.3]
    \fill (0,0) circle (2pt);
    \fill (-1,1.5) circle (2pt);
    \fill (0,1.5) circle (2pt);
    \fill (1,1.5) circle (2pt);
    \draw (0,0)--(-1,1.5);
    \draw (0,0)--(0,1.5);
    \draw (0,0)--(1,1.5);
    \end{tikzpicture}}})|=1, \quad |\mu(\vcenter{\hbox{\begin{tikzpicture}[scale=0.35]
    \draw (0,0)--(0.87,0.5)--(0.87,-0.5)--cycle;
   \draw (-0.7,0)--(0,0);
   \fill (0,0) circle (2pt);
   \fill (-0.7,0) circle (2pt);
    \fill (0.87,0.5) circle (2pt);
    \fill (0.87,-0.5) circle (2pt);
\end{tikzpicture}}})|=2, \quad |\mu(\vcenter{\hbox{\begin{tikzpicture}[scale=0.3]
    \fill (0,0) circle (2pt);
    \fill (0,1.5) circle (2pt);
    \fill (1.5,0) circle (2pt);
    \fill (1.5,1.5) circle (2pt);
    \draw (0,0)--(1.5,0)--(1.5,1.5)--(0,1.5)-- cycle;
    \end{tikzpicture}}})|=3,  \quad |\mu(\vcenter{\hbox{\begin{tikzpicture}[scale=0.3]
    \fill (0,0) circle (2pt);
    \fill (0,1.5) circle (2pt);
    \fill (1.5,0) circle (2pt);
    \fill (1.5,1.5) circle (2pt);
    \draw (0,0)--(1.5,0)--(1.5,1.5)--(0,1.5)-- cycle;
    \draw (0,0)--(1.5,1.5);
    \end{tikzpicture}}})|=4, \quad |\mu(\vcenter{\hbox{\begin{tikzpicture}[scale=0.3]
    \fill (0,0) circle (2pt);
    \fill (0,1.5) circle (2pt);
    \fill (1.5,0) circle (2pt);
    \fill (1.5,1.5) circle (2pt);
    \draw (0,0)--(1.5,0)--(1.5,1.5)--(0,1.5)-- cycle;
    \draw (0,0)--(1.5,1.5);
    \draw (1.5,0)--(0,1.5);
    \end{tikzpicture}}})|=6.
\] The normal monomials with respect to the quadratic relations $\R_{\Lie}^{\forget}$ have the form
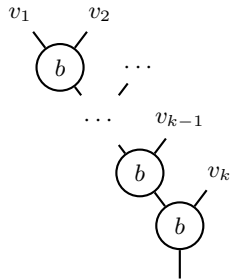
\[
\vcenter{\hbox{\begin{tikzpicture}[
        scale=0.7,
        vert/.style={circle,  draw=black!30!black, thick, minimum size=1mm},
        emptyvert/.style={circle, thick, minimum size=1mm},
        leaf/.style={rectangle, thick, minimum size=1mm},
        edge/.style={-,black!30!black, thick},
        ]
        \node[vert] (k) at (0.75,0) {\footnotesize$b$};
        \node[leaf] (lk) at (1.5,1) {\footnotesize$v_k$};
        \node[vert] (1) at (0,1) {\footnotesize$b$};
        \node[leaf] (l1) at (0.75,2) {\footnotesize$v_{k-1}$};
        \node[emptyvert] (2) at (-0.75,2) {\footnotesize$\cdots$};
        \node[leaf] (l2) at (0,3) {\footnotesize$\cdots$};
        \node[vert] (3) at (-1.5,3) {\footnotesize$b$};
        \node[leaf] (l3) at (-0.75,4) {\footnotesize$v_2$};
        \node[leaf] (l4) at (-2.25,4) {\footnotesize$v_1$};
        \draw[edge] (0.75,-1)--(k);
        \draw[edge] (k)--(1)--(2)--(3);
        \draw[edge] (k)--(lk);
        \draw[edge] (1)--(l1);
        \draw[edge] (2)--(l2);
        \draw[edge] (3)--(l3);
        \draw[edge] (3)--(l4);
    \end{tikzpicture}}}
 \] with an additional condition on the leaf labeling: if leaves $v_i$, $v_{i+1}$ are not adjacent in the underlying graph, then we have $v_i<v_{i+1}$. We leave it to the reader to check that the condition of Theorem~\ref{thm:pbw} is satisfied. Hence, the normal monomials with respect to quadratic relations form a monomial basis of $\Lie^{\forget}$.
 \end{proof}

\subsection{NBC-sets and pairing}\label{subsec:lemma}

In this subsection, we prove that the morphism $\tau\colon \Gerst \twoheadrightarrow H_{\bullet}(\D_2)$ is surjective. To establish the surjection, we construct the perfect pairing between $\Gerst$-monomials and monomials from Orlik-Solomon algebras, which is a direct generalisation of the pairing constructed by Sinha~\cite{sinha2005pairing}.

\noindent \textbf{Step 1:} The morphism of contractads $\tau\colon \Gerst \rightarrow H_{\bullet}(\D_2)$ induces the pairing
\begin{gather*}
   \langle-,-\rangle\colon \Gerst\otimes \OS \rightarrow \mathbb{Z}
   \\
   \langle T,\omega \rangle:=\tau(T)(\omega). 
\end{gather*} The fact that $\tau$ is a contractad morphism ensures the compatibility of this pairing with (co)contractad structures
\[\langle T\circ^{\Gr}_{G} T', \omega \rangle=\langle T\otimes T', \triangle^{\Gr}_{G}(\omega)\rangle.\] Moreover, the fact that $\tau$ is onto is equivalent to the fact  that this pairing is non-degenerate from the right.

\noindent \textbf{Step 2:} Let us describe this pairing explicitly. For a tree monomial $T$ and a non-zero algebraic monomial $\omega_S=\omega_{e_1}\omega_{e_2}\cdots\omega_{e_k}$ in the same component, we define the map $\phi_{T,S}\colon S \rightarrow \Ver(T)$ from the underlying set of edges to the set of vertices that sends each edge $\{v,w\} \in S$ to the vertex at the nadir of the shortest path in $T$ between the leaves with labels $v$ and $w$.

\begin{claim*} For a tree monomial $T$, such that vertices labeled by $m$ are concentrated at the bottom, and a non-zero algebraic monomial $\omega_S$ in the same component, the pairing $\langle T, \omega_S\rangle$ is non-zero if and only if the function $\phi_{T,S}$ provides a bijection between edges from $S$ and vertices of $T$ labeled by $b$. In this case, the pairing is equal to $1$ up to sign.
\end{claim*}
\begin{proof}
The proof is by induction on the number of vertices in $\Gr$. The base $\Gr=\Path_2$ is obvious. Consider a pair $T, \omega_S$ such that pairing $\langle T, \omega_S\rangle$ is non-zero. If $T$ is made up only of $m$-vertices,  the only choice is $\omega_S=1$. In this case, the map $\phi_{T,S}$ is bijective, since both $S$ and $\Ver_{b}(T)$ are empty. Otherwise, the monomial has the form $T=T'\circ^{\Gr}_e b $ for some monomial $T'$. By compatibility, we have
\[
\langle T, \omega_S \rangle= \pm \langle T', \triangle_e^{(1)}(\omega_S) \rangle \langle b, \triangle_e^{(2)}(\omega_S)\rangle,
\] where the signs come from the Koszul rules. Both factors on the right are non-zero since the pairing from the left is non-zero. Hence, we must have $\triangle_e^{(2)}(\omega_S)=b$, or equivalently, the set $S$ contains edge $e$. By Proposition~\ref{osmaps}, we have $\triangle_e^{(1)}(\omega_S)=\omega_{S'}$, where $S'$ is the image of the complement $S\setminus \{e\}$ under the contraction $\Gr \rightarrow \Gr/e$. By the induction assumption, the map $\phi_{T',S'}$ is a bijection between vertices labeled by $b$ in $T'$ and elements of $S'$. Both assertions imply that the map $\phi_{T,S}$ is also bijective. 
\end{proof}

\noindent \textbf{Step 3:} Recall that for each ordered hyperplane arrangement $\mathcal{A}$, there is a monomial basis of the Orlik-Solomon algebra $\OS(\mathcal{A})$ defined as follows. A broken circuit is a circuit (minimal dependent set) with its smallest element deleted. In the case of the graphic arrangement $\B(\Gr)$, a circuit is a cycle with its smallest edge deleted. An \textbf{nbc} set is a set containing no broken circuits. It is known that the collection of monomials $\{\omega_S| S-\textbf{nbc}\}$ forms a basis of the Orlik-Solomon algebra $\OS(\mathcal{A})$~\cite{yuzvinsky2001orlik}.
\begin{figure}[ht]
    \centering
    \[
    \vcenter{\hbox{\begin{tikzpicture}[scale=0.8]
    \fill (0,1) circle (2pt);
    \fill (1,1) circle (2pt);
    \fill (1,0) circle (2pt);
    \fill (0,0) circle (2pt);
    \draw[dashed] (1,1)--(0,1);
    \draw (0,1)--(0,0);
    \draw[dashed] (1,0)--(1,1);
    \draw[dashed] (0,0)--(1,1);
    \draw (0,0)--(1,0);
    \node at (-0.25,1.25) {\footnotesize$1$};
    \node at (1.25,1.25) {\footnotesize$2$};
    \node at (1.25,-0.25) {\footnotesize$3$};
    \node at (-0.25,-0.25) {\footnotesize$4$};
    \end{tikzpicture}}}
    \quad
    \vcenter{\hbox{\begin{tikzpicture}[scale=0.8]
    \fill (0,1) circle (2pt);
    \fill (1,1) circle (2pt);
    \fill (1,0) circle (2pt);
    \fill (0,0) circle (2pt);
    \draw[dashed] (1,1)--(0,1);
    \draw[dashed] (0,1)--(0,0);
    \draw[dashed] (1,0)--(1,1);
    \draw (0,0)--(1,1);
    \draw (0,0)--(1,0);
    \node at (-0.25,1.25) {\footnotesize$1$};
    \node at (1.25,1.25) {\footnotesize$2$};
    \node at (1.25,-0.25) {\footnotesize$3$};
    \node at (-0.25,-0.25) {\footnotesize$4$};
    \end{tikzpicture}}}
    \quad
    \vcenter{\hbox{\begin{tikzpicture}[scale=0.8]
    \fill (0,1) circle (2pt);
    \fill (1,1) circle (2pt);
    \fill (1,0) circle (2pt);
    \fill (0,0) circle (2pt);
    \draw[dashed] (1,1)--(0,1);
    \draw[dashed] (0,1)--(0,0);
    \draw (1,0)--(1,1);
    \draw (0,0)--(1,1);
    \draw[dashed] (0,0)--(1,0);
    \node at (-0.25,1.25) {\footnotesize$1$};
    \node at (1.25,1.25) {\footnotesize$2$};
    \node at (1.25,-0.25) {\footnotesize$3$};
    \node at (-0.25,-0.25) {\footnotesize$4$};
    \end{tikzpicture}}}
    \quad
    \vcenter{\hbox{\begin{tikzpicture}[scale=0.8]
    \fill (0,1) circle (2pt);
    \fill (1,1) circle (2pt);
    \fill (1,0) circle (2pt);
    \fill (0,0) circle (2pt);
    \draw[dashed] (1,1)--(0,1);
    \draw (0,1)--(0,0);
    \draw[dashed] (1,0)--(1,1);
    \draw (0,0)--(1,1);
    \draw[dashed] (0,0)--(1,0);
    \node at (-0.25,1.25) {\footnotesize$1$};
    \node at (1.25,1.25) {\footnotesize$2$};
    \node at (1.25,-0.25) {\footnotesize$3$};
    \node at (-0.25,-0.25) {\footnotesize$4$};
    \end{tikzpicture}}}
    \]
    \caption{List of maximal \textbf{nbc}-subsets for $\K_{(1^2,2)}$. Edges from \textbf{nbc}-subsets are dashed. The order on edges: $(1,2)<(1,4)<(2,3)<(2,4)<(3,4)$.}
    \label{fig:nbc}
\end{figure}
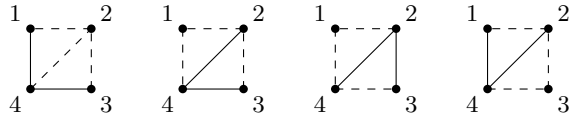

Let $\Gr$ be a graph with an ordered set of edges. For an \textbf{nbc}-subset $S=\{e_1<e_2<\cdots <e_n\}$ of a graph $\Gr$, we define the tree monomial $T(S)$ as follows. Consider the forest $F$ composed of corollas for each edge $e_i$ with leaves labeled with the vertices of this edge. On the $i$-th step, we take the edge $e_{n-i+1}$ and some vertex $v$ from this edge. If there exists a tree from the forest constructed in the previous step that has a leaf labeled by $v$ and whose bottom vertex is labeled by a larger edge, we join the leaf $v$ of the corolla labeled by $e_{n-i+1}$ to the root of such a tree. Do the same for the remaining vertex of $e_{n-i+1}$. At the end of the process, we get a disjoint union $\{T_1,T_2,\cdots T_k\}$ of trees labeled by $b$ and the underlying collection of leaf sets $I=\{L_1,L_2,\cdots,L_k\}$ forms a partition of the graph $\Gr$. Finally, we join the roots of these trees to the corresponding leaves of $\Com$-monomial $m_{\Gr/I}$. We denote the resulting $\Gr$-input monomial by $T(S)$. By the construction of this tree, we see that the associated map 
\[
\phi_{T(S),S}\colon S \rightarrow \Ver_b(T(S))=S
\] is identity. Hence, we have $\langle T(S),\omega_{S}\rangle=\pm 1$.

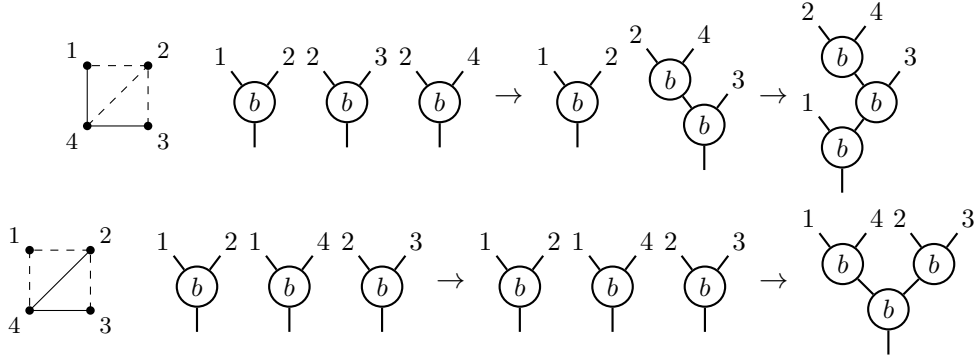
\begin{figure}[ht]
    \centering
    \begin{gather*}
    \vcenter{\hbox{\begin{tikzpicture}[scale=0.8]
    \fill (0,1) circle (2pt);
    \fill (1,1) circle (2pt);
    \fill (1,0) circle (2pt);
    \fill (0,0) circle (2pt);
    \draw[dashed] (1,1)--(0,1);
    \draw (0,1)--(0,0);
    \draw[dashed] (1,0)--(1,1);
    \draw[dashed] (0,0)--(1,1);
    \draw (0,0)--(1,0);
    \node at (-0.25,1.25) {\footnotesize$1$};
    \node at (1.25,1.25) {\footnotesize$2$};
    \node at (1.25,-0.25) {\footnotesize$3$};
    \node at (-0.25,-0.25) {\footnotesize$4$};
    \end{tikzpicture}}}
    \quad
        \vcenter{\hbox{\begin{tikzpicture}[
        scale=0.6,
        vert/.style={inner sep=2.5pt, circle,  draw, thick},
        leaf/.style={inner sep=2pt, rectangle},
        edge/.style={-,black!30!black, thick},
        ]
        \node[vert] (1) at (0,1) {\small $b$};
        \node[leaf] (l1) at (0.75,2) {\small$2$};
        \node[leaf] (l2) at (-0.75,2) {\small$1$};
        \draw[edge] (0,0)--(1);
        \draw[edge] (1)--(l1);
        \draw[edge] (1)--(l2);
    \end{tikzpicture}}}
        \vcenter{\hbox{\begin{tikzpicture}[
        scale=0.6,
        vert/.style={inner sep=2.5pt, circle,  draw, thick},
        leaf/.style={inner sep=2pt, rectangle},
        edge/.style={-,black!30!black, thick},
        ]
        \node[vert] (1) at (0,1) {\small $b$};
        \node[leaf] (l1) at (0.75,2) {\small$3$};
        \node[leaf] (l2) at (-0.75,2) {\small$2$};
        \draw[edge] (0,0)--(1);
        \draw[edge] (1)--(l1);
        \draw[edge] (1)--(l2);
    \end{tikzpicture}}}
    \vcenter{\hbox{\begin{tikzpicture}[
        scale=0.6,
        vert/.style={inner sep=2.5pt, circle,  draw, thick},
        leaf/.style={inner sep=2pt, rectangle},
        edge/.style={-,black!30!black, thick},
        ]
        \node[vert] (1) at (0,1) {\small $b$};
        \node[leaf] (l1) at (0.75,2) {\small$4$};
        \node[leaf] (l2) at (-0.75,2) {\small$2$};
        \draw[edge] (0,0)--(1);
        \draw[edge] (1)--(l1);
        \draw[edge] (1)--(l2);
    \end{tikzpicture}}}
    \to
        \vcenter{\hbox{\begin{tikzpicture}[
        scale=0.6,
        vert/.style={inner sep=2.5pt, circle,  draw, thick},
        leaf/.style={inner sep=2pt, rectangle},
        edge/.style={-,black!30!black, thick},
        ]
        \node[vert] (1) at (0,1) {\small $b$};
        \node[leaf] (l1) at (0.75,2) {\small$2$};
        \node[leaf] (l2) at (-0.75,2) {\small$1$};
        \draw[edge] (0,0)--(1);
        \draw[edge] (1)--(l1);
        \draw[edge] (1)--(l2);
    \end{tikzpicture}}}
    \vcenter{\hbox{\begin{tikzpicture}[
        scale=0.6,
        vert/.style={inner sep=2.5pt, circle,  draw, thick},
        leaf/.style={inner sep=2pt, rectangle},
        edge/.style={-,black!30!black, thick},
        ]
        \node[vert] (1) at (0,1) {\small $b$};
        \node[leaf] (l1) at (0.75,2) {\small$3$};
        \node[vert] (2) at (-0.75,2) {\small$b$};
        \node[leaf] (l2) at (0,3) {\small$4$};
        \node[leaf] (3) at (-1.5,3) {\small$2$};
        \draw[edge] (0,0)--(1);
        \draw[edge] (1)--(2)--(3);
        \draw[edge] (1)--(l1);
        \draw[edge] (2)--(l2);
    \end{tikzpicture}}}
    \to
        \vcenter{\hbox{\begin{tikzpicture}[
        scale=0.6,
        vert/.style={inner sep=2.5pt, circle,  draw, thick},
        leaf/.style={inner sep=2pt, rectangle},
        edge/.style={-,black!30!black, thick},
        ]
        \node[vert] (1) at (-1.5,1) {\small $b$};
        \node[leaf] (l1) at (-2.25,2) {\small$1$};
        \node[vert] (2) at (-0.75,2) {\small$b$};
        \node[leaf] (l2) at (0,3) {\small$3$};
        \node[vert] (3) at (-1.5,3) {\small$b$};
        \node[leaf] (l3) at (-2.25,4) {\small$2$};
        \node[leaf] (l4) at (-0.75,4) {\small$4$};
        \draw[edge] (-1.5,0)--(1);
        \draw[edge] (1)--(2)--(3);
        \draw[edge] (1)--(l1);
        \draw[edge] (2)--(l2);
        \draw[edge] (3)--(l3);
        \draw[edge] (3)--(l4);
    \end{tikzpicture}}}
    \\
        \vcenter{\hbox{\begin{tikzpicture}[scale=0.8]
    \fill (0,1) circle (2pt);
    \fill (1,1) circle (2pt);
    \fill (1,0) circle (2pt);
    \fill (0,0) circle (2pt);
    \draw[dashed] (1,1)--(0,1);
    \draw[dashed] (0,1)--(0,0);
    \draw[dashed] (1,0)--(1,1);
    \draw (0,0)--(1,1);
    \draw (0,0)--(1,0);
    \node at (-0.25,1.25) {\footnotesize$1$};
    \node at (1.25,1.25) {\footnotesize$2$};
    \node at (1.25,-0.25) {\footnotesize$3$};
    \node at (-0.25,-0.25) {\footnotesize$4$};
    \end{tikzpicture}}}
    \quad
        \vcenter{\hbox{\begin{tikzpicture}[
        scale=0.6,
        vert/.style={inner sep=2.5pt, circle,  draw, thick},
        leaf/.style={inner sep=2pt, rectangle},
        edge/.style={-,black!30!black, thick},
        ]
        \node[vert] (1) at (0,1) {\small $b$};
        \node[leaf] (l1) at (0.75,2) {\small$2$};
        \node[leaf] (l2) at (-0.75,2) {\small$1$};
        \draw[edge] (0,0)--(1);
        \draw[edge] (1)--(l1);
        \draw[edge] (1)--(l2);
    \end{tikzpicture}}}
        \vcenter{\hbox{\begin{tikzpicture}[
        scale=0.6,
        vert/.style={inner sep=2.5pt, circle,  draw, thick},
        leaf/.style={inner sep=2pt, rectangle},
        edge/.style={-,black!30!black, thick},
        ]
        \node[vert] (1) at (0,1) {\small $b$};
        \node[leaf] (l1) at (0.75,2) {\small$4$};
        \node[leaf] (l2) at (-0.75,2) {\small$1$};
        \draw[edge] (0,0)--(1);
        \draw[edge] (1)--(l1);
        \draw[edge] (1)--(l2);
    \end{tikzpicture}}}
    \vcenter{\hbox{\begin{tikzpicture}[
        scale=0.6,
        vert/.style={inner sep=2.5pt, circle,  draw, thick},
        leaf/.style={inner sep=2pt, rectangle},
        edge/.style={-,black!30!black, thick},
        ]
        \node[vert] (1) at (0,1) {\small $b$};
        \node[leaf] (l1) at (0.75,2) {\small$3$};
        \node[leaf] (l2) at (-0.75,2) {\small$2$};
        \draw[edge] (0,0)--(1);
        \draw[edge] (1)--(l1);
        \draw[edge] (1)--(l2);
    \end{tikzpicture}}}
    \to
    \vcenter{\hbox{\begin{tikzpicture}[
        scale=0.6,
        vert/.style={inner sep=2.5pt, circle,  draw, thick},
        leaf/.style={inner sep=2pt, rectangle},
        edge/.style={-,black!30!black, thick},
        ]
        \node[vert] (1) at (0,1) {\small $b$};
        \node[leaf] (l1) at (0.75,2) {\small$2$};
        \node[leaf] (l2) at (-0.75,2) {\small$1$};
        \draw[edge] (0,0)--(1);
        \draw[edge] (1)--(l1);
        \draw[edge] (1)--(l2);
    \end{tikzpicture}}}
        \vcenter{\hbox{\begin{tikzpicture}[
        scale=0.6,
        vert/.style={inner sep=2.5pt, circle,  draw, thick},
        leaf/.style={inner sep=2pt, rectangle},
        edge/.style={-,black!30!black, thick},
        ]
        \node[vert] (1) at (0,1) {\small $b$};
        \node[leaf] (l1) at (0.75,2) {\small$4$};
        \node[leaf] (l2) at (-0.75,2) {\small$1$};
        \draw[edge] (0,0)--(1);
        \draw[edge] (1)--(l1);
        \draw[edge] (1)--(l2);
    \end{tikzpicture}}}
        \vcenter{\hbox{\begin{tikzpicture}[
        scale=0.6,
        vert/.style={inner sep=2.5pt, circle,  draw, thick},
        leaf/.style={inner sep=2pt, rectangle},
        edge/.style={-,black!30!black, thick},
        ]
        \node[vert] (1) at (0,1) {\small $b$};
        \node[leaf] (l1) at (0.75,2) {\small$3$};
        \node[leaf] (l2) at (-0.75,2) {\small$2$};
        \draw[edge] (0,0)--(1);
        \draw[edge] (1)--(l1);
        \draw[edge] (1)--(l2);
    \end{tikzpicture}}}
    \to
    \vcenter{\hbox{\begin{tikzpicture}[
        scale=0.6,
        vert/.style={inner sep=2.5pt, circle,  draw, thick},
        leaf/.style={inner sep=2pt, rectangle},
        edge/.style={-,black!30!black, thick},
        ]
    \node[vert] (1) at (0,1) {\small $b$};
    \node[vert] (2) at (-1,2) {\small $b$};
    \node[vert] (3) at (1,2) {\small $b$};
    \node[leaf] (l1) at (-1.75,3) {\small $1$};
    \node[leaf] (l2) at (-0.25,3) {\small $4$};
    \node[leaf] (l3) at (0.25,3) {\small $2$};
    \node[leaf] (l4) at (1.75,3) {\small $3$};
    \draw[edge] (0,0)--(1);
    \draw[edge] (1)--(2);
    \draw[edge] (1)--(3);
    \draw[edge] (2)--(l1);
    \draw[edge] (2)--(l2);
    \draw[edge] (3)--(l3);
    \draw[edge] (3)--(l4);
    \end{tikzpicture}}}
    \end{gather*}
    \caption{Construction of monomials for \textbf{nbc}-subsets from \textbf{Figure}~\ref{fig:nbc}.}
    \label{fig:T(S)}
\end{figure}

\begin{claim*}
For a given graph $\Gr$ and pair of \textbf{nbc}-monomials $\omega_S,\omega_{S'} \in \OS(\Gr)$, the pairing $\langle T(S),\omega_{S'}\rangle$ differs from zero if and only if these monomials coincide. 
\end{claim*}
\noindent Note that this claim completes the proof of Lemma~\ref{onto}. Indeed, the restriction of the pairing to the linear span $\langle T(S) \rangle$ on the left part is non-degenerate, hence the pairing is non-degenerate from the right. Furthermore, the proof of Theorem~\ref{thm:gerst} guarantees that the monomial collection $\{T(S)\}$ forms a monomial basis of $\Gerst$. 
\begin{proof} Assume the converse. Let $S\neq S'$ be two distinct \textbf{nbc}-monomials such that the pairing $\langle T(S),\omega_{S'}\rangle$ is non-zero. By the previous claim, map $\phi_{T(\omega_S),S'}: S' \rightarrow S$ is bijective. Since $S\neq S'$, there is a subtree $T_0 \subset T(S)$ with a bottom vertex $e_0\in S$, such that $f_{T(S),S'}^{-1}(e_0)=e_0'\neq e_0$, and, for each non-root vertex of the subtree $T_0$, we have $f_{T(\omega_S),S'}^{-1}(e')=e'$ . Hence, we have two distinct \textbf{nbc}-subsets $S_0=\Ver_b(T')\subset S$ and $S'_0 = \{e_1\}\cup (\Ver_b(T')\setminus\{e'_0\}) \subset S'$. Moreover, the union $S_0\cup S'_0 \subset E_{\Gr}$ contains some cycle $\Cyc$ containing edges $e_0,e'_0$. Furthermore, by the construction of the tree $T(S)$, we have condition $e_0=\min S_0$. Therefore, we have $e_1>e_0$ since $S_0\cap \Cyc$ can not be broken circuit. So, the subset $S_0'\cap \Cyc$ is a broken circuit, hence we get a contradiction to the \textbf{nbc}-property of $S'_0$. 
\end{proof}

\begin{cor}
    For a graph $\Gr$, the monomials $T(S)$ ranging over \textbf{nbc}-subsets form a basis of $\Gerst^{\forget}(\Gr)$.
\end{cor}

\subsection*{Acknowledgements} I am grateful to my advisor Anton Khoroshkin for his guidance and inspiration in the process of writing. I would like to thank Vladimir Dotsenko and Adam Keilthy for useful discussions at various stages of the preparation of this paper and comments on its draft. 
\subsection*{Funding} This paper was supported by the grant RSF 22-21-00912 of Russian Science Foundation.
\bibliographystyle{alpha}
\bibliography{biblio.bib}

\end{document}